\documentclass[11pt,a4paper]{article}
\pdfoutput=1

\usepackage{amssymb}
\usepackage{amsmath}
\usepackage{amsfonts}
\usepackage{bbm}
\usepackage{amsthm}
\usepackage{mathrsfs}
\usepackage{hyperref}
\usepackage{color}
\usepackage[margin=2.35cm]{geometry}
\usepackage[all,cmtip]{xy}
\usepackage[utf8]{inputenc}
\usepackage{graphicx}
\usepackage{varwidth}
\usepackage{comment}
\usepackage{enumitem}

\usepackage{mathtools}

\usepackage{upgreek}
\usepackage{rotating}

\usepackage{tikz}
\usetikzlibrary{shapes.geometric}

\usepackage{tikz-cd}
\tikzset{Rightarrow/.style={double equal sign distance,>={Implies},->},
triple/.style={-,preaction={draw,Rightarrow}},
quadruple/.style={preaction={draw,Rightarrow,shorten >=0pt},shorten >=1pt,-,double,double
distance=0.2pt}}

\usepackage{quiver}

\mathchardef\mhyphen="2D

\definecolor{darkred}{rgb}{0.8,0.1,0.1}
\hypersetup{
	colorlinks=true,         
	linkcolor=darkred,
	citecolor=blue,
}

\theoremstyle{plain}
\newtheorem{theo}{Theorem}[section]
\newtheorem{lem}[theo]{Lemma}
\newtheorem{propo}[theo]{Proposition}
\newtheorem{cor}[theo]{Corollary}

\theoremstyle{definition}
\newtheorem{defi}[theo]{Definition}

\newtheorem{exx}[theo]{Example}
\newenvironment{ex}
{\pushQED{\qed}\exx}
{\popQED\endexx}

\newtheorem{remm}[theo]{Remark}
\newenvironment{rem}
{\pushQED{\qed}\remm}
{\popQED\endremm}

\numberwithin{equation}{section}

\def\nn{\nonumber}

\def\bbK{\mathbb{K}}

\def\bbZ{\mathbb{Z}}

\def\Hom{\mathrm{Hom}}
\def\hom{\underline{\mathrm{hom}}}

\def\Sym{\mathrm{Sym}}

\def\id{\mathrm{id}}
\def\Id{\mathrm{Id}}
\def\ID{\mathrm{ID}}

\def\dd{\mathrm{d}}

\def\dR{\mathrm{dR}}

\def\dim{\mathrm{dim}}
\def\1{I}
\def\oone{\mathbbm{1}}

\def\Ch{\mathbf{Ch}}

\def\dgMod{\mathbf{dgMod}}

\def\CC{\mathbf{C}}
\def\DD{\mathbf{D}}

\def\EE{\mathbf{E}}

\def\Cat{\mathbf{Cat}}

\def\dgCat{\mathbf{dgCat}}

\def\2Cat{\mathbf{2Cat}}
\def\3Cat{\mathbf{3Cat}}

\def\g{\mathfrak{g}}

\def\C{\mathcal{C}}
\def\D{\mathcal{D}}

\def\colim{\mathrm{colim}}

\def\Pol{\mathrm{Pol}}
\def\CE{\mathrm{CE}}

\newcommand\und[1]{\underline{#1}}

\newcommand{\ip}[2]{\langle #1,#2 \rangle}
\newcommand{\ipbig}[2]{\big\langle #1,#2 \big\rangle}

\makeatletter
\newcommand{\xRrightarrow}[2][]{\ext@arrow 0359\Rrightarrowfill@{#1}{#2}}
\newcommand{\Rrightarrowfill@}{\arrowfill@\equiv\equiv\Rrightarrow}
\newcommand{\xLleftarrow}[2][]{\ext@arrow 3095\Lleftarrowfill@{#1}{#2}}
\newcommand{\Lleftarrowfill@}{\arrowfill@\Lleftarrow\equiv\equiv}
\makeatother

\def\sk{\vspace{2mm}}

\makeatletter
\let\@fnsymbol\@alph
\makeatother

\title{%
Infinitesimal $2$-braidings from $2$-shifted Poisson structures
}

\author{%
Cameron Kemp$^{a}$, Robert Laugwitz$^{b}$\ and\ Alexander Schenkel$^{c}$\vspace{4mm}\\
{\small School of Mathematical Sciences, University of Nottingham,}\\
{\small University Park, Nottingham NG7 2RD, United Kingdom.}\vspace{4mm}\\
{\small \begin{tabular}{ll}
Email: & ${}^a$~\href{mailto:cameron.kemp@nottingham.ac.uk}{\texttt{cameron.kemp@nottingham.ac.uk}}\\
&${}^b$~\href{mailto:robert.laugwitz@nottingham.ac.uk}{\texttt{robert.laugwitz@nottingham.ac.uk}}\\
& ${}^c$~\href{mailto:alexander.schenkel@nottingham.ac.uk}{\texttt{alexander.schenkel@nottingham.ac.uk}}
\end{tabular}
}
}

\date{February 2025}

\begin{document}

\maketitle

\begin{abstract}
\noindent It is shown that every $2$-shifted Poisson structure on a finitely generated semi-free commutative differential graded algebra $A$ defines a very explicit infinitesimal $2$-braiding on the homotopy $2$-category of the symmetric monoidal dg-category of finitely generated semi-free $A$-dg-modules. This provides a concrete realization, to first order in the deformation parameter $\hbar$, of the abstract deformation quantization results in derived algebraic geometry due to Calaque, Pantev, To{\"e}n, Vaqui{\'e} and Vezzosi. Of particular interest is the case when $A$ is the Chevalley-Eilenberg algebra of a Lie $N$-algebra, where the braided monoidal deformations developed in this paper may be interpreted as candidates for representation categories of `higher quantum groups'.
\end{abstract}
\vspace{-3mm}

\vspace{-2mm}

\vspace{-2mm}

\tableofcontents


\section{Introduction and summary}
\label{sec:intro}
Quantum groups are interesting and important mathematical objects
which enjoy a wide range of applications in representation
theory, topological quantum field theory
and integrability, see e.g.\ \cite{Majid,Kassel,EtingofSchiffmann} for surveys.
There exist two complementary main approaches to the study of quantum groups:
The algebraic approach based on quasi-triangular Hopf algebras
and the categorical approach studying linear braided monoidal categories.
The relationship between these approaches is given by the observation 
that a quasi-triangular Hopf algebra has a linear braided monoidal category
of representations from which it can be recovered by Tannakian reconstruction.
An important class of examples of quantum groups are the quantized 
universal enveloping algebras $U_q(\g)$ which are defined
as quasi-triangular Hopf algebra deformations of the universal enveloping algebra $U(\g)$
of a Lie algebra $\g$. In the case of formal deformations $q=e^\hbar$, 
their linear braided monoidal representation categories $\mathbf{Rep}_\hbar(\g)$ can also be obtained 
directly as braided monoidal deformations of the linear symmetric monoidal representation
category $\mathbf{Rep}(\g)$ of the Lie algebra $\g$, see
e.g.\ \cite{Cartier} and \cite[Chapter XX]{Kassel}.
\sk

Given the importance of quantum groups, it is natural to search 
for a generalization of such structures to higher category theory.
Such `higher quantum groups' are then expected to have applications
in higher-categorical representation theory, higher-dimensional
topological quantum field theory and higher-dimensional integrability.
To the best of our knowledge, such applications have still not yet materialized fully,
but there are recent research programs in this direction, see e.g.\
\cite{Schenkel,Chen} for applications in the 
context of higher gauge theory and higher-dimensional integrability
and also \cite{ChenPhD} for an informative overview.
From a mathematical perspective, there have been various proposals to develop a theory
of `higher quantum groups' based on `higher Lie algebras' such as Lie $2$-algebras, 
reaching from algebraic approaches,
see e.g.\ \cite{Majid2,Zhu,Voronov,Girelli}, to more categorical ones,
see e.g.\ \cite{Joao1,Joao,Joao2}. These attempts are distinct from 
the approach to `higher quantum groups' via `categorification' in the sense 
of \cite{CraneFrenkel,ChuangRouquier,Rouquier,KhovanovLauda} where 
`higher quantum groups' are constructed based on classical Lie algebras.
\sk

In the present paper, we initiate a new approach to `higher quantum groups'
which is rooted in recent developments in derived algebraic geometry,
in particular the theory of shifted symplectic and shifted Poisson structures
 \cite{PTVV,CPTVV,Pridham}. Our inspiration 
comes from Safronov's results in \cite{Safronov} which show that the semi-classical
data associated with quantum groups, namely Lie bialgebra structures
and invariant symmetric tensors, arise naturally as the $1$-shifted and $2$-shifted Poisson structures
on the classifying stack $\mathrm{B}G=[\mathrm{pt}/G]$ of an algebraic group $G$,
or its infinitesimal analogue $\mathrm{B}\g=[\mathrm{pt}/\g]$ where $\g$ denotes 
the Lie algebra of $G$.
\sk

Given an $n$-shifted Poisson structure on a derived stack $X$, for an integer 
$n\geq 1$, the symmetric monoidal dg-category  $\mathbf{Perf}(X)$ 
of perfect modules over $X$ can be quantized using the $\mathbb{E}_n$-monoidal 
deformation quantization results from \cite{CPTVV}. 
In the case where $X = \mathrm{B}\g$, one has an equivalence
of symmetric monoidal dg-categories
$\mathbf{Perf}(\mathrm{B}\g) \simeq \mathbf{dgRep}(\g)$
identifying perfect modules over $\mathrm{B}\g$ with (homotopy-coherent) 
representations of $\g$ on cochain complexes. Hence, every 
$2$-shifted Poisson structure on $\mathrm{B}\g$, which by \cite{Safronov}
is an invariant symmetric tensor in $(\Sym^2\g)^\g$,
defines a braided monoidal deformation $\mathbf{Perf}_\hbar (\mathrm{B}\g)
\simeq \mathbf{dgRep}_\hbar(\g)$ which is expected to generalize
the constructions in \cite{Cartier} and \cite[Chapter XX]{Kassel} 
to representations of $\g$ on cochain complexes. In Corollary~\ref{cor:Lieallorder}
we provide a concrete realization of this braided monoidal deformation to all orders of $\hbar$.
\sk

Our working hypothesis is that the above picture should generalize
to the case where the Lie algebra $\g$ is replaced by a `higher Lie algebra', i.e.\ 
a Lie $N$-algebra $L$ for some $N\geq 2$. Inspired by \cite{Safronov}, we will \textit{define} 
the semi-classical `higher quantum group' data
on $L$ to be given by the $2$-shifted Poisson structures
on the infinitesimal classifying stack $\mathrm{B}L = [\mathrm{pt}/L]$,
and \textit{propose} the braided monoidal deformation $\mathbf{Perf}_\hbar (\mathrm{B}L)
\simeq \mathbf{dgRep}_\hbar(L)$ from \cite{CPTVV} as a model for
the representation category of the `higher quantum group $U_q(L)$' with $q=e^\hbar$.
\sk

Instead of making use of the abstract deformation quantization results from \cite{CPTVV},
we will approach this deformation problem similarly to the case of ordinary Lie algebras 
in \cite{Cartier} and \cite[Chapter XX]{Kassel} by breaking it into two steps:
The first step is to consider only first-order deformations in $\hbar$ of the braiding, 
which will be controlled by a higher-categorical analogue of an infinitesimal braiding.
We show that this step is rather explicit and computationally accessible.
The second step consists of lifting these first-order deformations
to all orders in $\hbar$, which is more difficult and points towards the
necessity to develop a homotopy-coherent generalization of the concept
of Drinfeld associators. Our paper provides a detailed
solution for the first step and an exploratory study of the second step.
The problem of developing explicit 
categorical deformation constructions was also addressed 
in the recent paper \cite{KKMP} by using complementary methods
from factorization homology.
\sk

We will now explain our results by outlining the content of this paper.
In Section \ref{sec:DAG}, we recall some basic aspects of differential graded algebraic geometry
that are needed in our work. In particular, we provide a brief discussion of the crucial concept of an 
$n$-shifted Poisson structure (Definition \ref{def:nshiftedpoisson})
on a finitely generated semi-free commutative differential graded algebra 
(CDGA) $A$. Since we are mainly interested in the Chevalley-Eilenberg algebras
of Lie $N$-algebras, which carry only non-trivial stacky structure but a trivial derived structure,
we can restrict ourselves to the so-called `stacky affine' context
from \cite{PridhamNotes,Pridham,PridhamOutline}. Hence, our focus will be 
on singly-graded CDGAs, which in our applications are always 
concentrated in non-negative cohomological degrees.
The technically more involved concepts 
of the doubly-graded `stacky CDGAs' from \cite{PridhamNotes,Pridham,PridhamOutline}
or the graded mixed CDGAs from \cite{CPTVV} will not be needed in our present work.
\sk

In Section \ref{sec:2braidings}, we study deformations of the symmetric monoidal dg-category 
${}_A\dgMod^{\mathrm{fgsf}}$ of finitely generated semi-free
$A$-dg-modules into a braided monoidal category.
More precisely, we restrict our attention to the simpler problem of deforming 
only the homotopy $2$-category
$$
{}_A\CC \,:=\, \mathsf{Ho}_2\big({}_A\dgMod^{\mathrm{fgsf}}\big)
$$
of this symmetric monoidal dg-category.
The homotopy $2$-category ${}_A\CC$ has the same underlying objects, 
i.e.\ finitely generated semi-free $A$-dg-modules, and its morphisms
are given by the good truncations to cohomological degrees $\{-1,0\}$ of the 
morphism complexes in ${}_A\dgMod^{\mathrm{fgsf}}$, see Appendix \ref{app:Ho2Cat}.
Our motivation behind this truncation of dg-categories
to their homotopy $2$-categories is that the latter are $2$-categorical objects 
(in contrast to $\infty$-categorical ones) whose deformations can be described very 
explicitly in the context of the $3$-category $\dgCat_R^{[-1,0],\mathrm{ps}}$ from Proposition \ref{prop:Chtrunccat}.
Most importantly, in this context every tower of homotopy coherence data is automatically 
truncated, therefore making it accessible with our rather direct computational techniques. 
A further motivation is that our main examples of interest arise from 
Lie $2$-algebras and their representations (see Section \ref{sec:examples}), 
which can be considered naturally in this truncated $2$-categorical context.
We show in Proposition \ref{prop:hexagoninf} that, 
to first order in the deformation parameter $\hbar$, the braided monoidal deformations
of the homotopy $2$-category ${}_A\CC$ are governed by infinitesimal 
$2$-braidings as in Definition \ref{def:inf2braiding}. 
Our concept of infinitesimal $2$-braidings is compatible with the one proposed
by Cirio and Faria Martins \cite{Joao} in the context of categories enriched over $2$-vector spaces,
see Remark \ref{rem:comparisonJoao}.
\sk

Our main results are presented in Subsection \ref{subsec:construction},
where we show that every $2$-shifted Poisson structure on a finitely generated semi-free CDGA
$A$ defines, by a very explicit construction, a $\gamma$-equivariant 
infinitesimal $2$-braiding on ${}_A\CC$, see in particular Theorem \ref{theo:mainresult}. This means that
every $2$-shifted Poisson structure on $A$ defines an explicit first-order
braided monoidal deformation of the homotopy $2$-category of the symmetric monoidal
dg-category ${}_A\dgMod^{\mathrm{fgsf}}$, which realizes concretely 
the abstract deformation results from \cite{CPTVV} to first order in $\hbar$.
\sk

In Subsection \ref{subsec:higherorders}, we explore the behavior
of these first-order deformations at higher orders in the deformation parameter $\hbar$.
Following the conventional theory of infinitesimal braidings from
\cite{Cartier} and \cite[Chapter XX]{Kassel},
we observe in Proposition \ref{propo:allorderdeformation} that
a $\gamma$-equivariant infinitesimal $2$-braiding defines 
an all-order in $\hbar$ braided monoidal deformation of the homotopy $2$-category ${}_A\CC$ 
provided that certain identities which are similar to Drinfeld associators hold true strictly.
We observe in Proposition \ref{propo:Gamma4term} that, in general, these identities
do not hold true strictly, but only up to homotopies which we 
construct explicitly up to order $\hbar^2$. This suggests that lifting our first-order deformation construction
to all orders in $\hbar$ requires a homotopy-coherent generalization
of the concept of Drinfeld associators, involving suitable
hexagonators and pentagonators,  which we however currently 
do not know how to develop. We believe that combining the construction of Drinfeld associators
from parallel transports (see e.g.\ \cite{Bordemann}) with the
categorified Knizhnik-Zamolodchikov connections from \cite{Joao1,Joao2}
will provide a suitable framework to attack this problem in future work.
\sk

In Section \ref{sec:examples}, we apply our deformation quantization construction 
from Section \ref{sec:2braidings} to explicit examples 
associated with Lie $N$-algebras.
Following \cite{Pridham,PridhamOutline}, we describe
the infinitesimal classifying stack $\mathrm{B}L = [\mathrm{pt}/L]$ of a Lie $N$-algebra $L$ 
by its associated Chevalley-Eilenberg algebra $\CE^\bullet(L)$.
For finite-dimensional $L$, this is a finitely generated semi-free CDGA,
hence our techniques from the previous sections apply to this case.
In particular, our main Theorem \ref{theo:mainresult}
shows that every $2$-shifted Poisson structure
on $\CE^\bullet(L)$ defines an explicit first-order braided monoidal deformation
of the homotopy $2$-category 
$$
{}_{\CE^\bullet(L)}\CC\,\simeq\, \mathsf{Ho}_2\big(\mathbf{dgRep}(L)\big)
$$
of the symmetric monoidal dg-category of representations of the Lie $N$-algebra $L$
and their $\infty$-morphisms. In Subsection \ref{subsec:Lie}, we will specialize this result
to the case where $L=\g$ is an ordinary
Lie algebra, which leads to a homotopical generalization of
the traditional constructions in \cite{Cartier} and \cite[Chapter XX]{Kassel},
and in Subsection \ref{subsec:Lie2} to the case where $L=\g_\kappa$ is the string Lie $2$-algebra
obtained from an ordinary Lie algebra $\g$ and a $3$-cocycle $\kappa:\bigwedge^3\g\to\bbK$.
\sk

In the technical Appendix \ref{app:categorical}, we develop and present
a suitable higher-categorical framework in which our deformation
constructions in this paper are carried out. The main concept defined here
is the symmetric monoidal $3$-category $\dgCat_R^{[-1,0],\mathrm{ps}}$ from 
Proposition \ref{prop:Chtrunccat} whose objects
are categories enriched over $2$-term cochain complexes concentrated
in degrees $-1$ and $0$, morphisms are enriched functors,
$2$-morphisms are enriched \textit{pseudo}-natural transformations
and $3$-morphisms are enriched modifications. Working with pseudo-natural transformations provides 
the required flexibility for the deformation constructions of this paper. Indeed, 
the homotopy $2$-categories ${}_A\CC\,:=\,\mathsf{Ho}_2\big({}_A\dgMod^{\mathrm{fgsf}}\big)$
and their braided monoidal deformations from Section \ref{sec:2braidings}
are symmetric or, respectively, braided monoidal category objects in 
the symmetric monoidal $3$-category $\dgCat_R^{[-1,0],\mathrm{ps}}$.


\section{Differential graded algebraic geometry}
\label{sec:DAG}
\subsection{Cochain complexes of modules}
\label{subsec:cochain}
Throughout this work, we fix a field $\bbK$ of characteristic $0$.
Let $R$ be a commutative, associative and unital $\bbK$-algebra.
We denote the category of (not necessarily bounded) cochain complexes 
of $R$-modules by $\Ch_R$. Concretely, an object $V=(V,\dd_V)\in \Ch_R$ is a pair consisting of 
a family $V=\{V^i\}_{i\in\bbZ}$ of $R$-modules labeled by their degrees $i\in\bbZ$
and a family $\dd_V = \{\dd_V^i : V^i \to V^{i+1}\}_{i\in\bbZ}$ of degree-increasing
$R$-linear maps (called the differential) which squares to zero, i.e.\ $\dd_V\,\dd_V =0$. A morphism $f :V\to W$ in $\Ch_R$
is a family $f=\{f^i : V^i\to W^i\}_{i\in\bbZ}$ of degree-preserving $R$-linear maps
which commutes with the differentials, i.e.\ $\dd_W\,f = f\,\dd_V$.
\sk

The category $\Ch_R$ is closed symmetric monoidal with respect to the following structures:
The tensor product $V\otimes_R W\in\Ch_R$ of two cochain complexes $V,W\in\Ch_R$ is defined by
the relative tensor product of $R$-modules
\begin{subequations}
\begin{flalign}
(V\otimes_R W)^i\,:=\,\bigoplus_{j\in\bbZ} \big(V^j \otimes_R W^{i-j}\big)\quad,
\end{flalign}
for all $i\in\bbZ$, and the differential 
\begin{flalign}
\dd_{V\otimes_R W}(v\otimes_R w)\,:=\, \dd_V(v)\otimes_R w + (-1)^{\vert v\vert}\,v\otimes_R \dd_W(w)\quad,
\end{flalign}
\end{subequations}
for all homogeneous $v\in V$ and $w\in W$, where $\vert v\vert\in\bbZ$ denotes the degree of $v$.
The monoidal unit is $R\in \Ch_R$ concentrated in degree $0$ with trivial differential $\dd_R=0$.
The symmetric braiding is defined by the Koszul sign rule
\begin{flalign}
\gamma_{V,W}^{}\,:\, V\otimes_R W~\longrightarrow~W\otimes_R  V~,~~ v\otimes_R w ~\longmapsto~(-1)^{\vert v\vert\,\vert w\vert}\,w\otimes_R v\quad,
\end{flalign}
for all homogeneous $v\in V$ and $w\in W$. The internal hom $\hom_R(V,W)\in\Ch_R$ between 
two cochain complexes $V,W\in\Ch_R$ is defined by the internal hom of $R$-modules
\begin{subequations}
\begin{flalign}
\hom_R(V,W)^i\,:=\, \prod_{j\in\bbZ}\mathrm{hom}_R\big(V^j,W^{j+i}\big)\quad,
\end{flalign}
for all $i\in\bbZ$, and the differential 
\begin{flalign}
\dd_{\hom_R(V,W)}(f)\,:=\,\dd_W\,f - (-1)^{\vert f \vert}\,f\,\dd_V\quad,
\end{flalign}
\end{subequations}
for all homogeneous $f\in \hom_R(V,W)$.
\sk

In the special case where $R=\bbK$ is the base field, we shall simply write
$(\Ch_\bbK,\otimes,\bbK,\hom)$ for the closed symmetric monoidal structure, 
suppressing the subscript ${}_\bbK$ from our notations.
\sk

Given any integer $n\in\bbZ$ and any cochain complex $V\in\Ch_R$, we define the $n$-shifted 
cochain complex $V[n]\in\Ch_R$ by $V[n]^i:= V^{i+n}$, for all $i\in\bbZ$, and $\dd_{V[n]}:= (-1)^n\,\dd_V$.
Using the identification $V[n]\cong R[n]\otimes_R V$, we shall write elements of the $n$-shifted cochain
complex as $s^{-n}\,v\in V[n]$, where $s^{-n}\in R[n]$ denotes the degree $-n$ element represented
by the unit of $R$.

\subsection{Semi-free CDGAs and dg-modules}
\label{subsec:semi-free}
A commutative differential graded algebra over the field $\bbK$ (in short, a CDGA) is a commutative
algebra object $A = (A,\dd_A,\mu_A,\eta_A)$ in the closed symmetric monoidal category $\Ch_\bbK$.
More explicitly, a CDGA is a cochain complex $A=(A,\dd_A)\in\Ch_\bbK$ 
endowed with $\Ch_\bbK$-morphisms $\mu_A : A\otimes A\to A$ (called multiplication)
and $\eta_A : \bbK\to A$ (called unit) which satisfy the usual associativity,
unitality and commutativity axioms internal to $\Ch_\bbK$.
\begin{defi}\label{def:semifreeCDGA}
A CDGA $A= (A,\dd_A,\mu_A,\eta_A)$ is called \textit{semi-free}
if its underlying $\bbZ$-graded commutative algebra $A^\sharp = (A,\mu_A,\eta_A)$ is free, 
i.e.\ 
\begin{flalign}
A^\sharp \,\cong\,\Sym(V^\sharp)
\end{flalign}
is isomorphic to the symmetric algebra of a $\bbZ$-graded vector space $V^\sharp$.
A semi-free CDGA is called \textit{finitely generated} if the total dimension
$\dim(V^\sharp) := \sum_{i\in\bbZ}\dim(V^i)<\infty$ of the $\bbZ$-graded 
vector space $V^\sharp$ is finite.
\end{defi}

Associated with any CDGA $A$ is its dg-category ${}_A\dgMod$ of (left) $A$-dg-modules.
More explicitly, an object $M = (M,\dd_M,\ell_M)$ in this category is
a cochain complex $M=(M,\dd_M)\in\Ch_\bbK$ endowed with a $\Ch_\bbK$-morphism 
$\ell_M : A\otimes M\to M$ (called left action) which satisfies the 
usual left module axioms internal to $\Ch_\bbK$. The cochain complex
of morphisms between two objects $M,N\in {}_A\dgMod$ is defined as the subcomplex
\begin{subequations}\label{eqn:hom_A}
\begin{flalign}
\hom_A(M,N)\,\subseteq \, \hom(M,N)
\end{flalign}
of the internal hom in $\Ch_\bbK$ containing the homogeneous elements
$f\in \hom(M,N)$ which commute with the left actions, i.e.\
\begin{flalign}
f\big(\ell_M(a\otimes m)\big)\,=\,(-1)^{\vert f\vert\,\vert a\vert}~\ell_N\big(a\otimes f(m)\big)\quad,
\end{flalign}
\end{subequations}
for all $a\in A$ and $m\in M$. The dg-category ${}_A\dgMod$ is closed symmetric monoidal
with respect to the relative tensor product
\begin{flalign}\label{eqn:otimesA}
M\otimes_A N\,:=\,\colim\bigg(
\xymatrix@C=8em{
M\otimes N ~&~ \ar@<1ex>[l]^-{\id_M\otimes \ell_N}\ar@<-1ex>[l]_-{(\ell_M\otimes \id_N)\,(\gamma_{M,A}^{}\otimes\id_N)}M\otimes A\otimes N
}
\bigg)\,\in\,{}_A\dgMod\quad,
\end{flalign}
the monoidal unit $A=(A,\dd_A,\mu_A)\in {}_A\dgMod$, the symmetric braiding
\begin{flalign}\label{eqn:gammaA}
\gamma_{M,N}^{}\,:\, M\otimes_A N ~\longrightarrow~N\otimes_A M~,~~m\otimes_A n~\longmapsto~(-1)^{\vert m\vert\,\vert n\vert}~n\otimes_A m\quad,
\end{flalign}
for all homogeneous $m\in M$ and $n\in N$, and the internal hom
\begin{flalign}
\hom_A(M,N)\,\in\,{}_A\dgMod
\end{flalign}
given by the cochain complex \eqref{eqn:hom_A} and the left action 
$A\otimes \hom_A(M,N)\to \hom_A(M,N)\,,~a\otimes f \mapsto a\cdot f$
defined by $(a\cdot f)(m)\,:=\,\ell_N(a\otimes f(m))$, for all $a\in A$,
$f\in \hom_A(M,N)$ and $m\in M$.
\begin{defi}\label{def:semifreedgMod}
An $A$-dg-module $M = (M,\dd_M,\ell_M)$ is called \textit{semi-free}
if its underlying $A^\sharp$-module $M^\sharp = (M,\ell_M)$ is free, i.e.\ 
\begin{flalign}
M^\sharp \,\cong\,A^\sharp\otimes W^\sharp
\end{flalign}
for a $\bbZ$-graded vector space $W^\sharp$. 
A semi-free $A$-dg-module is called \textit{finitely generated} if the total dimension
$\dim(W^\sharp) := \sum_{i\in\bbZ}\dim(W^i)<\infty$ of the $\bbZ$-graded 
vector space $W^\sharp$ is finite. We denote by
\begin{flalign}
{}_A\dgMod^{\mathrm{fgsf}}\,\subseteq \, {}_A\dgMod
\end{flalign}
the full dg-subcategory whose objects are all finitely generated 
semi-free $A$-dg-modules.
\end{defi}

\begin{rem}\label{rem:semifreedgMod}
The closed symmetric monoidal structure on ${}_A\dgMod$
restricts to a closed symmetric monoidal structure
on the full dg-subcategory ${}_A\dgMod^{\mathrm{fgsf}}\subseteq {}_A\dgMod$.
Indeed, the monoidal unit $A\in {}_A\dgMod$ is finitely generated semi-free and,
given any two finitely generated semi-free $A$-dg-modules
$M,N\in {}_A\dgMod^{\mathrm{fgsf}}$ with 
$M^\sharp \cong A^\sharp\otimes W^\sharp$ and $N^\sharp\cong A^\sharp\otimes U^\sharp$, 
we have isomorphisms
\begin{subequations}
\begin{flalign}
\big(M\otimes_A N\big)^\sharp \,\cong \,
A^\sharp\otimes \big(W^\sharp \otimes U^\sharp\big)
\end{flalign}
and
\begin{flalign}
\hom_A(M,N)^\sharp\,\cong\,A^\sharp\otimes\big(U^\sharp\otimes (W^\sharp)^\ast\big)\quad,
\end{flalign}
\end{subequations}
which imply that $M\otimes_A N\in {}_A\dgMod^{\mathrm{fgsf}}$ and
$\hom_A(M,N)\in {}_A\dgMod^{\mathrm{fgsf}}$ are finitely generated semi-free.
Here $(W^\sharp)^\ast$ denotes the dual of the $\bbZ$-graded vector space $W^\sharp$,
which exists since $W^\sharp$ is, by hypothesis, of finite total dimension.
\end{rem}

\subsection{Derivations, K\"ahler differentials and shifted Poisson structures}
\label{subsec:shiftedPoisson}
We recall some basic aspects of the geometry of a finitely generated semi-free CDGA $A$, i.e.\
$A^\sharp \cong \Sym (V^\sharp)$ for some $\bbZ$-graded vector space $V^\sharp$ of finite total dimension.
The dg-module of derivations $\mathsf{T}_{\! A}\in {}_A\dgMod$ of $A$ is defined as the subcomplex
\begin{subequations}
\begin{flalign}
\mathsf{T}_{\! A}\,\subseteq \, \hom(A,A)
\end{flalign}
of the internal hom in $\Ch_\bbK$ containing the homogeneous elements $D\in\hom(A,A)$ which satisfy the
Leibniz rule
\begin{flalign}
D(a\,a^\prime)\,=\, D(a)\,a^\prime + (-1)^{\vert D\vert\, \vert a\vert}~a\,D(a^\prime)\quad,
\end{flalign}
\end{subequations}
for all homogeneous $a,a^\prime\in A$. The $A$-dg-module structure 
$A\otimes \mathsf{T}_{\! A} \to \mathsf{T}_{\! A}\,,~a\otimes D\mapsto a\cdot D$ is given by
$(a\cdot D)(a^\prime):= a\,D(a^\prime)$, for all $a,a^\prime\in A$ and $D\in \mathsf{T}_{\! A}$.
Since $A$ is finitely generated semi-free, each derivation is completely 
specified by its action on the generators $V^\sharp$, which yields an isomorphism
\begin{flalign}
(\mathsf{T}_{\! A})^\sharp\,\cong\,A^\sharp \otimes (V^\sharp)^\ast\quad,
\end{flalign}
where $(V^\sharp)^\ast$ denotes the dual of the $\bbZ$-graded vector space $V^\sharp$.
Hence, $\mathsf{T}_{\! A} \in{}_A\dgMod^{\mathrm{fgsf}}$ is a finitely generated semi-free $A$-dg-module.
\sk

The underlying cochain complex of $\mathsf{T}_{\! A}$ can be endowed with the Lie algebra structure
given by the commutator
\begin{flalign}\label{eqn:Liebracket}
[\,\cdot\,,\,\cdot\,]\,:\,\mathsf{T}_{\! A}\otimes \mathsf{T}_{\! A}  ~\longrightarrow~ \mathsf{T}_{\! A}
~,~~D\otimes D^\prime~\longmapsto~[D,D^\prime]\,:=\,
D\,D^\prime - (-1)^{\vert D\vert\,\vert D^\prime\vert}~D^\prime\,D\quad,
\end{flalign}
for all homogeneous $D,D^\prime\in \mathsf{T}_{\! A}$.
The Lie algebra and $A$-dg-module structures on $\mathsf{T}_{\! A}$ satisfy the compatibility condition
\begin{flalign}
[D,a\cdot D^\prime]\,=\, D(a)\cdot D^\prime + (-1)^{\vert D\vert\,\vert a\vert}~a\cdot[D,D^\prime]\quad,
\end{flalign}
for all homogeneous $D,D^\prime\in \mathsf{T}_{\! A}$ and $a\in A$.
\sk

As a consequence of our hypothesis that $A$ is a finitely generated semi-free CDGA, 
the dg-module of derivations $\mathsf{T}_{\! A} \in {}_A\dgMod^{\mathrm{fgsf}}$ is dualizable. We define
the dg-module of K\"ahler differentials by the internal hom
\begin{subequations}\label{eqn:Kaehlerdifferentials}
\begin{flalign}
\Omega_A\,:=\,\hom_A(\mathsf{T}_{\! A} , A )\,\in\, {}_A\dgMod^{\mathrm{fgsf}}\quad,
\end{flalign}
which is finitely generated semi-free as a consequence of Remark \ref{rem:semifreedgMod} and
\begin{flalign}
(\Omega_A)^\sharp\,\cong\,A^\sharp\otimes (V^\sharp)^{\ast\ast} \,\cong\,A^\sharp\otimes V^\sharp\quad.
\end{flalign}
\end{subequations}
We denote the duality pairings between $\Omega_A$ and $\mathsf{T}_{\! A}$ by
\begin{subequations}\label{eqn:dualitypairing}
\begin{flalign}
\ip{\,\cdot\,}{\,\cdot\,}\,:\,\Omega_A\otimes_A \mathsf{T}_{\! A}~\longrightarrow~A~,~~
\omega\otimes_A D~\longmapsto~\ip{\omega}{D}\,:=\,\omega(D)
\end{flalign}
and 
\begin{flalign}
\ip{\,\cdot\,}{\,\cdot\,}\,:\,\mathsf{T}_{\! A} \otimes_A\Omega_A~\longrightarrow~A~,~~
D\otimes_A \omega~\longmapsto~\ip{D}{\omega}\,:=\,(-1)^{\vert D\vert\,\vert \omega\vert}~\omega(D)\quad,
\end{flalign}
\end{subequations}
for all homogeneous $\omega\in\Omega_A$ and $D\in \mathsf{T}_{\! A}$. The de Rham differential
$\dd_\dR : A\to \Omega_A$ is defined by
\begin{flalign}
\ipbig{D}{\dd_\dR(a)}\,:=\,D(a)\quad,
\end{flalign}
for all $a\in A$ and $D\in \mathsf{T}_{\! A}$. From this definition one directly
checks that $\dd_\dR(a\,a^\prime) = \dd_\dR(a)\cdot a^\prime + a\cdot \dd_{\dR}(a^\prime)$,
for all $a,a^\prime\in A$, and that $\dd_{\Omega_A}\,\dd_{\dR} = \dd_{\dR}\,\dd_{A}$.
\sk

A crucial preliminary concept for the definition of shifted Poisson structures
is that of shifted polyvectors.
\begin{defi}\label{def:poly}
Let $n\in\bbZ$ be an integer. The CDGA of \textit{$n$-shifted polyvectors} on
a finitely generated semi-free CDGA $A$ is defined as the relative symmetric algebra
\begin{flalign}
\Pol(A,n)\,:=\, \Sym_A\big(\mathsf{T}_{\! A}[-n-1]\big)\,=\, \bigoplus_{m\geq 0}\Sym_A^m\big(\mathsf{T}_{\! A}[-n-1]\big)
\end{flalign}
of the $(-n-1)$-shift of the $A$-dg-module of derivations $\mathsf{T}_{\! A}$. The non-negative integer
$m\in\bbZ^{\geq 0}$ in the direct sum decomposition in terms of symmetric powers is called
the \textit{weight} of polyvectors.
\end{defi}

The CDGA of $n$-shifted polyvectors $\Pol(A,n)$ can be endowed with a canonical $\mathbb{P}_{n+2}$-algebra
structure (called Schouten-Nijenhuis bracket)
\begin{flalign}
\{\,\cdot\,,\,\cdot\,\}\,\in\,\hom\Big(\Pol(A,n)\otimes \Pol(A,n),\Pol(A,n)\Big)^{-n-1}
\end{flalign}
which is defined on the generators by
\begin{subequations}\label{eqn:Schoutendef}
\begin{flalign}
\{a,a^\prime\}\,&:=\,0\quad,\\[3pt]
\{s^{n+1}D,a\}\,&:=\,D(a)\quad,\\[3pt]
\{s^{n+1}D,s^{n+1}D^\prime\}\,&:=\,(-1)^{\vert D\vert\,(n+1)}~s^{n+1}[D,D^\prime]\quad,
\end{flalign}
\end{subequations}
for all homogeneous $a,a^\prime\in A$ and $D,D^\prime \in \mathsf{T}_{\! A}$, where we recall that
$D(a)\in A$ denotes the evaluation of the derivation $D$ on $a\in A$ 
and $[D,D^\prime]\in \mathsf{T}_{\! A}$ denotes
the Lie bracket \eqref{eqn:Liebracket} on $ \mathsf{T}_{\! A}$. The bracket
$\{\,\cdot\,,\,\cdot\,\}$ is then extended to $\Pol(A,n)$ by the defining
properties of a $\mathbb{P}_{n+2}$-algebra structure, i.e.\
\begin{subequations}\label{eqn:Schoutenproperties}
\begin{itemize}
\item[(i)] Antisymmetry: For all homogeneous $P,Q\in\Pol(A,n)$,
\begin{flalign}
\{P,Q\}\,=\,-(-1)^{(n+1)}~(-1)^{\vert P\vert\,\vert Q\vert}~\{Q,P\}\quad.
\end{flalign}

\item[(ii)] Derivation property: For all homogeneous $P,Q,R\in\Pol(A,n)$,
\begin{flalign}
\{P,Q\,R\}\,=\,\{P,Q\}\,R + (-1)^{(\vert P\vert -n-1)\, \vert Q\vert}~Q\,\{P,R\}\quad.
\end{flalign}

\item[(iii)] Jacobi identity: For all homogeneous $P,Q,R\in\Pol(A,n)$,
\begin{flalign}
\big\{ P,\{Q,R\} \big\}\,=\,(-1)^{(\vert P\vert -n-1)\,(n+1)}~\big\{\{P,Q\} ,R \big\} + (-1)^{(\vert P\vert -n-1)\,(\vert Q\vert -n-1)}~\big\{Q,\{P,R\}\big\}\quad.
\end{flalign}
\end{itemize}
\end{subequations}
Note that the Schouten-Nijenhuis bracket decreases the polyvector weight by $1$, i.e.\ for $P,Q\in\Pol(A,n)$
of weight $m_P$ and $m_Q$, the weight of $\{P,Q\}\in\Pol(A,n)$ is $m_P + m_Q -1$.
\sk

The definition of $n$-shifted Poisson structures in 
\cite{CPTVV} and \cite{Pridham,PridhamOutline} uses a completion
of the $\mathbb{P}_{n+2}$-algebra $\Pol(A,n)$
in order to avoid bounds on the weights of polyvectors.
The analogous definition in our context is as follows.
\begin{defi}\label{def:polycompleted}
The $\mathbb{P}_{n+2}$-algebra of \textit{completed $n$-shifted polyvectors}
on a finitely generated semi-free CDGA $A$ is defined by
\begin{flalign}
\widehat{\Pol}(A,n)\,:=\, \prod_{m\geq 0} \Sym^m_A\big(\mathsf{T}_{\! A}[-n-1]\big)\quad.
\end{flalign}
A completed $n$-shifted polyvector is thus a formal sum $P = \sum_{m\geq 0} P^{(m)}$
of homogeneous weight components $P^{(m)}\in\Sym^m_A\big(\mathsf{T}_{\! A}[-n-1]\big)$.
The $\mathbb{P}_{n+2}$-algebra structure on $\widehat{\Pol}(A,n)$ is given by
\begin{subequations}
\begin{flalign}
P\,Q\,&:=\, \sum_{m\geq 0}\bigg(\sum_{k+l=m} P^{(k)}\,Q^{(l)}\bigg)\quad,\\[3pt]
\{P,Q\}\,&:=\,\sum_{m\geq 0}\bigg(\sum_{k+l-1=m} \big\{P^{(k)},Q^{(l)}\big\}\bigg)\quad,
\end{flalign}
\end{subequations}
which is well-defined because the weight is bounded from below and hence
the sums in the parentheses are finite.
\end{defi}

We now define the key concept of an $n$-shifted Poisson structure following
\cite[Definition 1.5]{Pridham}, see also \cite[Definition 2.5]{PridhamOutline}.
\begin{defi}\label{def:nshiftedpoisson}
An \textit{$n$-shifted Poisson structure} on a finitely generated semi-free CDGA $A$ 
is a completed $n$-shifted polyvector
\begin{flalign}
\pi\,=\,\sum_{m\geq 2} \pi^{(m)}\,\in \,\widehat{\Pol}(A,n)^{n+2}
\end{flalign}
of degree $n+2$ and weight $\geq 2$ which satisfies the Maurer-Cartan equation
\begin{flalign}
\dd_{\widehat{\Pol}}(\pi) + \tfrac{1}{2}\,\{\pi,\pi\}\,=\,0\quad.
\end{flalign}
\end{defi}
\begin{rem}\label{rem:nshiftedpoisson}
Decomposing the Maurer-Cartan equation into homogeneous weight components
one obtains a tower of conditions
\begin{flalign}
\nn \dd_{\widehat{\Pol}}(\pi^{(2)})\,&=\,0\quad,\\[3pt]
\nn \dd_{\widehat{\Pol}}(\pi^{(3)})+\tfrac{1}{2}\,\{\pi^{(2)},\pi^{(2)}\}\,&=\,0\quad,\\[3pt]
\nn &~\vdots\\[3pt]
\dd_{\widehat{\Pol}}(\pi^{(m)})+\tfrac{1}{2}\,\sum_{k+l-1=m}\{\pi^{(k)},\pi^{(l)}\}\,&=\,0\quad,
\end{flalign}
for all $m\geq 3$. The first condition states that the bivector component
$\pi^{(2)}$ is closed with respect to the differential $\dd_{\widehat{\Pol}}$, i.e.\ it
defines an $(n+2)$-cocycle in $\Sym^2_A\big(\mathsf{T}_{\! A}[-n-1]\big)$.
The second condition provides a homotopical relaxation of the usual condition
that a Poisson bivector has a trivial Schouten-Nijenhuis bracket with itself,
with the trivector $\pi^{(3)}$ playing the role of a homotopy witnessing this condition. 
The higher weight components $\pi^{(m)}$ 
for $m\geq 4$ describe a coherent tower of higher homotopies for this relaxation.
\end{rem}


\section{Infinitesimal \texorpdfstring{$2$}{2}-braidings}
\label{sec:2braidings}
The aim of this section is to study a class of first-order deformations,
into a braided monoidal category, of the symmetric monoidal category ${}_A\dgMod^{\mathrm{fgsf}}$
of finitely generated semi-free dg-modules over a finitely generated
semi-free CDGA $A$ over the field $\bbK$. The precise context
in which our deformation constructions take place is given by the symmetric monoidal
$3$-category $\dgCat_R^{[-1,0],\mathrm{ps}}$ from Proposition \ref{prop:Chtrunccat},
consisting of categories, functors, \textit{pseudo-natural} transformations
and modifications that are suitably enriched over the symmetric monoidal category
$\Ch^{[-1,0]}_{R}$ of $2$-term cochain complexes concentrated in degrees $-1$ and $0$. 
Our starting point is thus given by the homotopy $2$-category
\begin{flalign}\label{eqn:Ho2dgMod}
{}_A\CC\,:=\,\mathsf{Ho}_2\big({}_A\dgMod^{\mathrm{fgsf}}\big)\,\in\, \dgCat^{[-1,0],\mathrm{ps}}_\bbK
\end{flalign}
from Appendix \ref{app:Ho2Cat}, which we endow with the symmetric monoidal
structure that is induced from the standard one on ${}_A\dgMod^{\mathrm{fgsf}}$ given in
\eqref{eqn:otimesA} and \eqref{eqn:gammaA}.
\sk 

The effect of forming the homotopy $2$-category
is that the cochain complex $\hom_A(M,N)\in\Ch_\bbK$ of morphisms 
between any two dg-modules $M,N\in {}_A\dgMod^{\mathrm{fgsf}}$ from \eqref{eqn:hom_A} 
gets replaced by its good truncation ${}_A\CC(M,N):=\tau^{[-1,0]}\big(\hom_A(M,N)\big)\in\Ch_\bbK^{[-1,0]}$
to degrees $-1$ and $0$. Concretely, this means that the homotopy $2$-category 
${}_A\CC = \mathsf{Ho}_2\big({}_A\dgMod^{\mathrm{fgsf}}\big)$  describes only 
degree $0$ morphisms between $A$-dg-modules which commute 
with the differentials and $1$-homotopies (modulo $2$-homotopies) between such morphisms,  
but it does not include information about higher homotopies. Our
motivation behind this truncation of the dg-category ${}_A\dgMod^{\mathrm{fgsf}}$ 
is that the homotopy $2$-category ${}_A\CC = \mathsf{Ho}_2\big({}_A\dgMod^{\mathrm{fgsf}}\big)$ 
is a $2$-categorical object whose deformations can be described very explicitly 
in the context of the $3$-category $\dgCat_R^{[-1,0],\mathrm{ps}}$ from Proposition \ref{prop:Chtrunccat}.
Most importantly, in this context every tower of homotopy coherence data is automatically 
truncated, therefore making it accessible with our more direct computational techniques. 
A further motivation is that our main examples of interest arise from 
Lie $2$-algebras and their representations (see Section \ref{sec:examples}), 
which can be considered naturally in this truncated $2$-categorical context.
Our main result is that every $2$-shifted Poisson structure on the CDGA $A$ as in Definition \ref{def:nshiftedpoisson}
defines a very explicit first-order deformation of the braiding on ${}_A\CC$.

\subsection{Setup and definitions}
\label{subsec:2braidingsdef}
To set up a framework for studying first-order deformations
of the symmetric monoidal category object ${}_A\CC$ from \eqref{eqn:Ho2dgMod} in the symmetric 
monoidal $3$-category $\dgCat_R^{[-1,0],\mathrm{ps}}$ from Proposition \ref{prop:Chtrunccat},
we introduce a deformation parameter $\hbar$ which squares to zero $\hbar^2=0$.
Performing a change of base along the 
commutative $\bbK$-algebra morphism $\bbK\to \bbK[\hbar]/(\hbar^2)$ defines
an object
\begin{flalign}
{}_A\CC^\hbar\,:=\, {}_A\CC\otimes\bbK[\hbar]/(\hbar^2)\,\in\, \dgCat_{\bbK[\hbar]/(\hbar^2)}^{[-1,0],\mathrm{ps}}\quad,
\end{flalign}
which we endow with the symmetric monoidal structure that is induced from the one on ${}_A\CC$.
More concretely, ${}_A\CC^\hbar$ is the $\Ch^{[-1,0]}_{\bbK[\hbar]/(\hbar^2)}$-enriched
category with the same objects as ${}_A\CC$, i.e.\  finitely generated
semi-free $A$-dg-modules, and whose morphisms 
are given by ${}_A\CC^\hbar(M,N) = {}_A\CC(M,N)\otimes\bbK[\hbar]/(\hbar^2)\in\Ch^{[-1,0]}_{\bbK[\hbar]/(\hbar^2)}$,
for all $M,N\in {}_A\CC^\hbar$. Following the notation from Appendix \ref{app:ChCat},
we shall denote the symmetric braiding on ${}_A\CC^\hbar$ by
\begin{equation}
\begin{tikzcd}
	{{}_A\CC^\hbar\boxtimes {}_A\CC^\hbar} && {{}_A\CC^\hbar} \\
	& {{}_A\CC^\hbar\boxtimes {}_A\CC^\hbar}
	\arrow[""{name=0, anchor=center, inner sep=0}, "\otimes_A", from=1-1, to=1-3]
	\arrow["{\tau}"', from=1-1, to=2-2]
	\arrow["\otimes_A"', from=2-2, to=1-3]
	\arrow["\gamma", shorten <=6pt, shorten >=3pt, Rightarrow, from=0, to=2-2]
\end{tikzcd}\quad,
\end{equation}
where $\tau$ is an abbreviation for the component $\tau_{{}_A\CC^\hbar,{}_A\CC^\hbar}$
of the symmetric braiding in \eqref{eqn:tauChCat} of the ambient $3$-category 
$\dgCat_{\bbK[\hbar]/(\hbar^2)}^{[-1,0],\mathrm{ps}}$. 
When written in terms of the $3$-categorical composition operations
from Appendix \ref{app:ChCat},
the symmetric braiding reads as $\gamma: \otimes_A\Rightarrow \otimes_A \ast\tau$.
Note that, by construction, this is a \textit{strict} 
$\Ch^{[-1,0]}_{\bbK[\hbar]/(\hbar^2)}$-enriched natural transformation.
\sk

In the context of (semi-strict) braided monoidal category
objects in $\dgCat_{\bbK[\hbar]/(\hbar^2)}^{[-1,0],\mathrm{ps}}$
from Definition \ref{def:2TermBraidedCat}, we would like to study deformations of
the braiding $\gamma: \otimes_A\Rightarrow \otimes_A \ast\tau$ of ${}_A\CC^\hbar$
which keep all the other structures (i.e.\ the tensor product, monoidal unit, associator and unitors) 
fixed. The general ansatz for a first-order deformation of the 
braiding $\gamma: \otimes_A\Rightarrow \otimes_A \ast\tau$ of ${}_A\CC^\hbar$
is obtained by considering a $\Ch^{[-1,0]}_{\bbK[\hbar]/(\hbar^2)}$-enriched \textit{pseudo-natural} 
transformation\footnote{We would like to emphasize that pseudo-naturality, in contrast to strict naturality, 
is crucial here in order to accommodate our deformation quantization constructions
in Subsection \ref{subsec:construction} below.} of the form
\begin{flalign}\label{eqn:gammahbar}
\gamma^\hbar\,:=\, \gamma\circ \big(\Id + \tfrac{\hbar}{2}\, t\big)\,:\,
\otimes_A~\Longrightarrow ~\otimes_A \ast\tau\,:\,{}_A\CC^\hbar\boxtimes{}_A\CC^\hbar~\longrightarrow~{}_A\CC^\hbar
\end{flalign}
which satisfies the two hexagon identities
\begin{subequations}\label{eqn:hexagon}
\begin{flalign}\label{eqn:hexagon1}
\begin{gathered}
\resizebox{.9\hsize}{!}{$
\xymatrix@C=2em@R=2em{
& \otimes_A\ast (\id\boxtimes \otimes_A)\ar@{=>}[r]^-{\gamma^\hbar\ast\Id} &  \otimes_A\ast (\otimes_A\boxtimes \id) \ast\tau_{23}\ast\tau_{12}\ar@{=>}[dr]^-{~~~~~\alpha\ast\Id\ast\Id}&\\
\otimes_A\ast (\otimes_A\boxtimes\id) \ar@{=>}[ru]^-{\alpha~}
\ar@{=>}[dr]_-{\Id\ast (\gamma^\hbar\boxtimes \Id)~~~~~}&& &\otimes_A\ast (\id\boxtimes\otimes_A) \ast\tau_{23}\ast\tau_{12}\\
&\otimes_A\ast (\otimes_A\boxtimes\id) \ast\tau_{12} \ar@{=>}[r]_-{\alpha\ast\Id}& \otimes_A\ast (\id\boxtimes \otimes_A) \ast\tau_{12} \ar@{=>}[ru]_-{~~~~~\Id\ast (\Id\boxtimes\gamma^\hbar)\ast\Id}&
}$}
\end{gathered}
\end{flalign}
and
\begin{flalign}\label{eqn:hexagon2}
\begin{gathered}
\resizebox{.9\hsize}{!}{$
\xymatrix@C=2em@R=2em{
&\otimes_A\ast (\otimes_A\boxtimes\id)\ar@{=>}[r]^-{\gamma^\hbar\ast\Id} &\otimes_A\ast (\id\boxtimes \otimes_A)\ast \tau_{12}\ast\tau_{23} \ar@{=>}[dr]^-{~~~~~\alpha^{-1}\ast\Id\ast\Id}&\\
\otimes_A\ast (\id\boxtimes\otimes_A) \ar@{=>}[ru]^-{\alpha^{-1}~}
\ar@{=>}[dr]_-{\Id\ast (\Id\boxtimes \gamma^\hbar)~~~~~}&&& \otimes_A\ast (\otimes_A\boxtimes \id)\ast \tau_{12}\ast\tau_{23}\quad ,\\
&\otimes_A\ast (\id\boxtimes\otimes_A) \ast\tau_{23} \ar@{=>}[r]_-{\alpha^{-1}\ast\Id}& \otimes_A\ast (\otimes_A\boxtimes \id) \ast\tau_{23} \ar@{=>}[ru]_-{~~~~~\Id\ast (\gamma^\hbar\boxtimes\Id)\ast\Id}&
}$}
\end{gathered}
\end{flalign}
\end{subequations}
where $\alpha$ denotes the associator and we use 
the typical short-hand notations $\tau_{12}:= \tau\boxtimes\id$ and 
$\tau_{23}:=\id\boxtimes\tau$. Note that in the top horizontal
arrows of these diagrams we applied the 
properties \eqref{eqn:tauChCat-Hexs} and \eqref{eqn:tauChCat-natural}
of the symmetric braiding $\tau$.
The composition of the  $\Ch^{[-1,0]}_{\bbK[\hbar]/(\hbar^2)}$-enriched 
pseudo-natural transformations along the hexagons \eqref{eqn:hexagon} is given by the 
$\circ$-composition from Appendix \ref{app:ChCat} and these diagrams are required to
commute strictly, i.e.\ up to equality of enriched pseudo-natural transformations. 
Strictness of the hexagon identities turns out to be sufficient
for our first-order deformations in Subsection \ref{subsec:construction},
but we will illustrate in Subsection \ref{subsec:higherorders} that higher-order
deformations require in general non-trivial 
hexagonator modifications for the diagrams \eqref{eqn:hexagon}.
\sk

Leveraging the fact that $\hbar^2=0$, the hexagon identities are equivalent
to two linear relations for the deformation $t$, which can be regarded as 
the datum of a $\Ch_\bbK^{[-1,0]}$-enriched pseudo-natural
transformation (without change of base along $\bbK\to \bbK[\hbar]/(\hbar^2)$)
of the form $t : \otimes_A \Rightarrow \otimes_A$ for the relative tensor
product on ${}_A\CC$. 
\begin{propo}\label{prop:hexagoninf}
The first hexagon identity \eqref{eqn:hexagon1} for the deformed braiding \eqref{eqn:gammahbar} 
is equivalent to the identity
\begin{subequations}\label{eqn:hexagoninf}
\begin{multline}\label{eqn:hexagon1inf}
\alpha^{-1}\circ(t\ast\Id)\circ\alpha\,=\,\Id\ast(t\boxtimes\Id) \\
+ \big(\Id\ast(\gamma\boxtimes\Id)^{-1}\big)\circ
\big(\big(\alpha^{-1}\circ \big(\Id\ast(\Id\boxtimes t)\big)\circ\alpha\big) \ast \Id\big)\circ \big(\Id\ast(\gamma\boxtimes\Id)\big)
\end{multline}
of $\Ch_\bbK^{[-1,0]}$-enriched 
pseudo-natural transformations 
$\otimes_A\ast (\otimes_A\boxtimes\id)\Rightarrow\otimes_A\ast (\otimes_A\boxtimes\id)$.
The second hexagon identity \eqref{eqn:hexagon2} for the deformed braiding \eqref{eqn:gammahbar} 
is equivalent to the identity
\begin{multline}\label{eqn:hexagon2inf}
\alpha\circ(t\ast\Id)\circ\alpha^{-1}\,=\,\Id\ast(\Id\boxtimes t) \\
+ \big(\Id\ast(\Id\boxtimes\gamma)^{-1}\big)\circ
\big(\big(\alpha\circ \big(\Id\ast(t\boxtimes \Id)\big)\circ\alpha^{-1}\big) \ast \Id\big)\circ \big(\Id\ast(\Id\boxtimes\gamma)\big)
\end{multline}
\end{subequations}
of $\Ch_\bbK^{[-1,0]}$-enriched 
pseudo-natural transformations 
$\otimes_A\ast (\id\boxtimes\otimes_A)\Rightarrow\otimes_A\ast (\id\boxtimes\otimes_A)$.
\end{propo}
\begin{proof}
The proof follows directly by computing and simplifying the order $\hbar$ 
terms of the hexagon identities. The key observation is 
that the pseudo-functor coherences \eqref{eqn:compositioncoherences}
which describe the interchange law between the $\ast$ and $\circ$ compositions
are trivial in all instances appearing in this calculation. Let us illustrate this
through an example: The order $\hbar$ term corresponding to the upper path of \eqref{eqn:hexagon1}
involves the pseudo-natural transformation $(\gamma\circ t)\ast\Id$, which arises
from a particular composition of the pasting diagram
\begin{equation}
\begin{tikzcd}
	& {} && {} \\
	{{}_A\mathbf{C}\boxtimes{}_A\mathbf{C}\boxtimes{}_A\mathbf{C}} && {{}_A\mathbf{C}\boxtimes{}_A\mathbf{C}} && {{}_A\mathbf{C}} \\
	& {} && {}
	\arrow["{\mathrm{id}\boxtimes\otimes_A}", curve={height=-30pt}, from=2-1, to=2-3]
	\arrow["{\mathrm{id}\boxtimes\otimes_A}"', curve={height=30pt}, from=2-1, to=2-3]
	\arrow[""{name=0, anchor=center, inner sep=0}, "{\mathrm{id}\boxtimes\otimes_A}"{description}, from=2-1, to=2-3]
	\arrow["{\otimes_A}", curve={height=-30pt}, from=2-3, to=2-5]
	\arrow["{\otimes_A\ast\tau}"', curve={height=30pt}, from=2-3, to=2-5]
	\arrow[""{name=1, anchor=center, inner sep=0}, "{\otimes_A}"{description}, from=2-3, to=2-5]
	\arrow["{\mathrm{Id}\,}"', shorten <=7pt, shorten >=7pt, Rightarrow, from=1-2, to=0]
	\arrow["t\,"', shorten <=7pt, shorten >=7pt, Rightarrow, from=1-4, to=1]
	\arrow["{\mathrm{Id}\,}"', shorten <=7pt, shorten >=7pt, Rightarrow, from=0, to=3-2]
	\arrow["\gamma\,"', shorten <=7pt, shorten >=7pt, Rightarrow, from=1, to=3-4]
\end{tikzcd}\qquad.
\end{equation}
The pseudo-functor coherences \eqref{eqn:compositioncoherences} for this pasting diagram are trivial
because of the identity $t_{(M,N\otimes_A L),(M, N\otimes_A L)}\big(\id_{(M,N\otimes_A L)}\big) =0$,
for all $(M,N,L)\in {}_A\mathbf{C}\boxtimes{}_A\mathbf{C}\boxtimes{}_A\mathbf{C}$,
given by property (2) of the Definition \ref{def:pseudonatural}
of $\Ch_\bbK^{[-1,0]}$-enriched pseudo-natural transformations.
It then follows that 
\begin{flalign}
(\gamma\circ t)\ast\Id\,=\,(\gamma\circ t)\ast(\Id\circ\Id) \,=\,(\gamma\ast\Id)\circ (t\ast\Id) \quad,
\end{flalign}
which is the key step to identify the left-hand side of \eqref{eqn:hexagon1inf}. 
In addition to arguments of this form, one also has to use that 
the pseudo-functor coherences \eqref{eqn:compositioncoherences} are trivial
whenever $\zeta^\prime$ is strictly natural, i.e.\ $\zeta^\prime_{d,d^\prime}=0$ for all $d,d^\prime\in\DD$.
\end{proof}

This result justifies the following 
\begin{defi}\label{def:inf2braiding}
A (semi-strict) \textit{infinitesimal $2$-braiding} on the symmetric monoidal category object
${}_A\CC=\mathsf{Ho}_2\big({}_A\dgMod^{\mathrm{fgsf}}\big)$ in the symmetric monoidal $3$-category
$\dgCat^{[-1,0],\mathrm{ps}}_\bbK$ from Proposition \ref{prop:Chtrunccat}
is a $\Ch_\bbK^{[-1,0]}$-enriched pseudo-natural transformation 
$t: \otimes_A\Rightarrow\otimes_A : {}_A\CC\boxtimes{}_A\CC\rightarrow {}_A\CC$ which satisfies 
the two infinitesimal hexagon identities in \eqref{eqn:hexagoninf}.
\end{defi}

\begin{rem}\label{rem:inf2braiding=>deformation}
As a consequence of Proposition \ref{prop:hexagoninf}, the datum of an infinitesimal
$2$-braiding $t$ is equivalent to a first-order deformation
$\gamma^\hbar = \gamma\circ (\Id +\tfrac{\hbar}{2} \,t) : \otimes_A\Rightarrow\otimes_A\ast\tau$ of the symmetric braiding
on ${}_A\CC$ to a braiding on ${}_A\CC^\hbar = {}_A\CC\otimes\bbK[\hbar]/(\hbar^2)$.
Note that the deformed braiding is not necessarily symmetric because
\begin{flalign}
(\gamma^\hbar\ast\Id)\circ\gamma^\hbar \,=\,\Id + \tfrac{\hbar}{2}\,\Big( t + (\gamma\ast\Id)\circ (t\ast \Id)\circ \gamma\Big) 
\,:\, \otimes_A~\Longrightarrow ~\otimes_A
\end{flalign}
is in general different from $\Id$.
\end{rem}

\begin{rem}\label{rem:infbraidingcomponents}
Using the explicit composition formulas from Appendix \ref{app:ChCat}, we can 
express the infinitesimal hexagon identities \eqref{eqn:hexagoninf}
in terms of the components of the $\Ch_\bbK^{[-1,0]}$-enriched pseudo-natural transformations involved.
(Recall also Definition \ref{def:pseudonatural}.) In this calculation,
various simplifications arise from the fact that all transformations but $t$
are strictly natural, i.e.\ their homotopy data given by the double-indexed components
\eqref{eqn:doubleindexedcomponents} is trivial. To ease notations, we suppress 
the components of the associator $\alpha$ in what follows. One then finds that the first
infinitesimal hexagon identity \eqref{eqn:hexagon1inf} is equivalent to
\begin{subequations}\label{eqn:hexagon1infexplicit}
\begin{flalign}
t_{M,N\otimes_A L}\,=\, t_{M,N}\otimes_A \id_L +  
\big(\gamma_{N,M}\otimes_A\id_L\big)~\big(\id_N\otimes_A t_{M,L}\big)~\big(\gamma_{M,N}\otimes_A\id_L\big)
\quad,
\end{flalign}
for all $M,N,L\in{}_A\CC$, and
\begin{multline}
t_{(M,N\otimes_A L),(M^\prime,N^\prime\otimes_A L^\prime)}\big(h\widetilde{\otimes}(k\otimes_A l)\big)\,=\,
t_{(M,N),(M^\prime,N^\prime)}\big(h\widetilde{\otimes}k\big)\otimes_A l \\
+\big(\gamma_{N^\prime,M^\prime}\otimes_A\id_{L^\prime}\big)~
\big(k \otimes_A t_{(M,L),(M^\prime,L^\prime)}\big(h\widetilde{\otimes} l\big)\big)
~\big(\gamma_{M,N}\otimes_A\id_L\big)
\quad,
\end{multline}
\end{subequations}
for all $M,N,L,M^\prime,N^\prime,L^\prime\in{}_A\CC$, $h\in {}_A\CC(M,M^\prime)^0$, $k\in{}_A\CC(N,N^\prime)^0$
and $l\in {}_A\CC(L,L^\prime)^0$.
The second infinitesimal hexagon identity \eqref{eqn:hexagon2inf} is equivalent to
\begin{subequations}\label{eqn:hexagon2infexplicit}
\begin{flalign}
t_{M\otimes_A N,L}\,=\, \id_M\otimes_A t_{N,L} + 
\big(\id_M\otimes_A\gamma_{L,N}\big)~\big(t_{M,L}\otimes_A\id_N\big)
~\big(\id_M\otimes_A\gamma_{N,L}\big)
\quad,
\end{flalign}
for all $M,N,L\in{}_A\CC$, and
\begin{multline}
t_{(M\otimes_A N,L),(M^\prime\otimes_A N^\prime,L^\prime)}\big((h\otimes_A k)\widetilde{\otimes}l\big)\,=\,
h\otimes_A t_{(N,L),(N^\prime,L^\prime)}\big(k\widetilde{\otimes}l\big) \\
+\big(\id_{M^\prime}\otimes_A\gamma_{L^\prime,N^\prime}\big)~
\big(t_{(M,L),(M^\prime,L^\prime)}\big(h\widetilde{\otimes}l\big)\otimes_A k\big)
~\big(\id_M\otimes_A\gamma_{N,L}\big)
\quad,
\end{multline}
\end{subequations}
for all $M,N,L,M^\prime,N^\prime,L^\prime\in{}_A\CC$, $h\in {}_A\CC(M,M^\prime)^0$, $k\in{}_A\CC(N,N^\prime)^0$
and $l\in {}_A\CC(L,L^\prime)^0$.
\end{rem}

\begin{rem}\label{rem:comparisonJoao}
Our Definition \ref{def:inf2braiding} of an infinitesimal $2$-braiding
is compatible with the one proposed in \cite[Definition 14]{Joao} in the context
of categories enriched over $2$-vector spaces,
in contrast to our $\Ch_\bbK^{[-1,0]}$-enriched categories from Appendix \ref{app:categorical}.
Indeed, the datum $(r,T)$ of an infinitesimal $2$-braiding in \cite[Definition 14]{Joao} 
consists of a family of $1$-morphisms $r_{M,N}$ and a family of $2$-morphisms $T_{h,k}$,
which in our context correspond to the single-indexed components
$t_{M,N}$ and the double-indexed components $t_{(M,N),(M^\prime,N^\prime)}(h\widetilde{\otimes}k)$
of the pseudo-natural transformation $t : \otimes_A\Rightarrow\otimes_A
: {}_A\CC\boxtimes{}_A\CC\rightarrow {}_A\CC$, see also Definition \ref{def:pseudonatural}.
Our $3$-categorical calculus from Proposition \ref{prop:Chtrunccat} allows us to
treat both of these components simultaneously by working with the various composition
operations of pseudo-natural transformations.
\end{rem}

Our examples of infinitesimal $2$-braidings in Subsection \ref{subsec:construction}
will satisfy a useful additional property that simplifies their analysis.
\begin{defi}\label{def:inf2braidinggammaequivariant}
An infinitesimal $2$-braiding $t: \otimes_A\Rightarrow\otimes_A$ is 
called \textit{$\gamma$-equivariant} if it satisfies the identity
\begin{flalign}\label{eqn:inf2braidinggammaequivariant}
\gamma \circ t\,=\,(t\ast \Id)\circ \gamma
\end{flalign}
of $\Ch_\bbK^{[-1,0]}$-enriched pseudo-natural transformations $\otimes_A\Rightarrow\otimes_A\ast\tau$.
\end{defi}

\begin{rem}
It is important to emphasize that $\gamma$-equivariance does \textit{not} imply
that the associated deformed braiding $\gamma^\hbar$ is symmetric. Indeed,
recalling the calculation in Remark \ref{rem:inf2braiding=>deformation}, we find that
\begin{flalign}
(\gamma^\hbar\ast\Id)\circ\gamma^\hbar\,=\,\Id + \hbar\, t\quad,
\end{flalign}
for every $\gamma$-equivariant infinitesimal $2$-braiding $t$.
\end{rem}

\begin{propo}\label{prop:equivariancehexagons}
Let $t: \otimes_A\Rightarrow\otimes_A$ be any $\Ch_\bbK^{[-1,0]}$-enriched pseudo-natural transformation
which satisfies the $\gamma$-equivariance property \eqref{eqn:inf2braidinggammaequivariant}.
Then the two infinitesimal hexagon identities \eqref{eqn:hexagoninf} for $t$ are equivalent.
\end{propo}
\begin{proof}
This can be verified directly by using the equivalent component-wise
description of the infinitesimal hexagon identities from Remark \ref{rem:infbraidingcomponents}
and the fact that the $\gamma$-equivariance property reads in components as
\begin{subequations}\label{eqn:gammaequivariancecomponents}
\begin{flalign}\label{eqn:gammaequivariancecomponents1}
\gamma_{M,N}~t_{M,N} \,=\, t_{N,M}~\gamma_{M,N}\quad,
\end{flalign}
for all $M,N\in{}_A\CC$, and
\begin{flalign}\label{eqn:gammaequivariancecomponents2}
\gamma_{M^\prime,N^\prime}~t_{(M,N),(M^\prime,N^\prime)}(h\widetilde{\otimes}k) \,=\,
t_{(N,M),(N^\prime,M^\prime)}(k\widetilde{\otimes}h)~\gamma_{M,N}\quad,
\end{flalign}
\end{subequations}
for all
$M,N,M^\prime,N^\prime\in{}_A\CC$, $h\in {}_A\CC(M,M^\prime)^0$ and $k\in{}_A\CC(N,N^\prime)^0$.
\end{proof}

\subsection{Construction from \texorpdfstring{$2$}{2}-shifted Poisson structures}
\label{subsec:construction}
Suppose that $\pi = \sum_{m\geq 2}\pi^{(m)}\in \widehat{\Pol}(A,2)^4$
is a $2$-shifted Poisson structure (recall Definition \ref{def:nshiftedpoisson})
on the finitely generated semi-free CDGA $A$. 
In this subsection, we will
present an explicit construction which assigns to
this datum a $\gamma$-equivariant infinitesimal $2$-braiding $t:\otimes_A\Rightarrow\otimes_A$
in the sense of Definitions \ref{def:inf2braiding}
and \ref{def:inf2braidinggammaequivariant}
on the homotopy $2$-category ${}_A\CC =\mathsf{Ho}_2\big({}_A\dgMod^{\mathrm{fgsf}}\big)\in 
\dgCat_\bbK^{[-1,0],\mathrm{ps}}$ of the dg-category of finitely generated semi-free $A$-dg-modules.
We will illustrate later in Section \ref{sec:examples}
that our construction is a non-trivial homotopical generalization
of the standard construction of an infinitesimal braiding on the representation category
$\mathbf{Rep}(\g)$ of a Lie algebra $\g$ from an invariant symmetric tensor in $(\Sym^2\g)^\g$, see
e.g.\ \cite{Cartier} and \cite[Chapter XX]{Kassel}.
\sk

Our infinitesimal $2$-braiding $t$ is built as a composite of 
more primitive $\Ch_\bbK^{[-1,0]}$-enriched pseudo-natural transformations
which we are going to define now. Let us denote by
\begin{flalign}
\Omega_A[1]\otimes_A(\,\cdot\,) \,:\, {}_A\CC~\longrightarrow~{}_A\CC
\end{flalign}
the $\Ch_\bbK^{[-1,0]}$-enriched functor of taking the relative tensor product 
with the $1$-shift $\Omega_A[1]\in{}_A\CC $ of the dg-module of K\"ahler differentials
from \eqref{eqn:Kaehlerdifferentials}. We define a 
$\Ch_\bbK^{[-1,0]}$-enriched pseudo-natural transformation
\begin{subequations}\label{eqn:zetatransformation}
\begin{flalign}
\xi\,:\, \id ~\Longrightarrow~\Omega_A[1]\otimes_A(\,\cdot\,)\,:\, {}_A\CC~\longrightarrow~{}_A\CC
\end{flalign}
from the $\Ch_\bbK^{[-1,0]}$-enriched  identity functor $\id : {}_A\CC\to {}_A\CC$
by specifying its components as in Definition \ref{def:pseudonatural}:
For each object $M\in {}_A\CC$, the component 
\begin{flalign}
\big(\xi_M : M\to \Omega_A[1]\otimes_A M\big)\,\in\, {}_A\CC\big(M,\Omega_A[1]\otimes_A M\big)^0
\end{flalign}
is the $A$-dg-module map 
which is defined on the generators $\oone\otimes w\in M^\sharp \cong A^\sharp \otimes W^\sharp$ 
of $M$ by
\begin{flalign}
\xi_M\big(\oone\otimes w\big)\,:=\, s^{-1}\,(\dd_\dR\otimes \id)\big(\dd_M(\oone\otimes w)\big)\quad,
\end{flalign}
where we recall our notation from Subsection \ref{subsec:cochain}
that $s^{-1}\in\bbK[1]$ denotes the degree $-1$ element represented by the unit of $\bbK$.
For each pair of objects $M,M^\prime\in {}_A\CC$ and each $A$-dg-module map 
$(h: M\to M^\prime)\in {}_A\CC(M,M^\prime)^0$,
the component 
\begin{flalign}
\big(\xi_{M,M^\prime}(h) : M\to \Omega_A[1]\otimes_A M^\prime\big)\,\in\, 
{}_A\CC\big(M,\Omega_A[1]\otimes_A M^\prime\big)^{-1}
\end{flalign}
is the homotopy which is defined on the generators
$\oone\otimes w\in M^\sharp \cong A^\sharp \otimes W^\sharp$ 
of $M$ by
\begin{flalign}
\xi_{M,M^\prime}(h)\big(\oone\otimes w\big)\,:=\,- s^{-1}\,(\dd_\dR\otimes \id)\big(h(\oone\otimes w)\big)\quad.
\end{flalign}
\end{subequations}
\begin{lem}\label{lem:xipseudonatural}
The components in \eqref{eqn:zetatransformation} define a $\Ch_\bbK^{[-1,0]}$-enriched 
pseudo-natural transformation $\xi : \id\Rightarrow \Omega_A[1]\otimes_A(\,\cdot\,)$.
\end{lem}
\begin{proof}
Before we verify the three properties (1), (2) and (3) from Definition \ref{def:pseudonatural},
let us first comment on the well-definedness of the components of $\xi$.
Concerning (i) from Definition \ref{def:pseudonatural}, 
we have to show that $\dd_{\hom_A}(\xi_M) := \dd_{\Omega_A[1]\otimes_A M}\, \xi_M - \xi_M\,\dd_M 
= 0$, i.e.\ that $\xi_M$ commutes with the differentials, which can be verified by picking a basis
$\{w_i\in W^\sharp\}$ and evaluating on the associated generators $\oone\otimes w_i\in M$. 
This allows us to write $\dd_M(\oone\otimes w_i) =: \und{M}_{i}^{~j}\otimes w_j$
and hence $\xi_M(\oone\otimes w_i) = s^{-1}\,\dd_{\dR}(\und{M}_{i}^{~j})\otimes w_j$, 
with $\und{M}_{i}^{~j}\in A$ and summations over repeated indices understood. We then compute
\begin{flalign}
\nn \dd_{\hom_A}(\xi_M)\big(\oone\otimes w_i\big)\,&=\,
\dd_{\Omega_A[1]\otimes_A M}\big(\xi_M(\oone\otimes w_i)\big) - 
\xi_M\big(\dd_M(\oone\otimes w_i) \big)\\
\nn \,&=\, -s^{-1}\dd_{\Omega_A\otimes_A M}\big(\dd_{\dR}(\und{M}_{i}^{~j})\otimes w_j\big)
- \und{M}_{i}^{~j}\,s^{-1}\,\dd_\dR(\und{M}_{j}^{~k})\otimes w_k\\
\nn \, &=\,-s^{-1}\,\Big(\dd_{\dR}\big(\dd_A(\und{M}_{i}^{~j})\big) \otimes w_j
+ (-1)^{\vert \und{M}_{i}^{~j}\vert}\, \dd_\dR\big( \und{M}_{i}^{~j}\,\und{M}_{j}^{~k}\big)\otimes w_k \Big)\\
\,&=\,0\quad,
\end{flalign}
where the last identity follows from expressing the square-zero condition $\dd_M\dd_M(\oone\otimes w_i)=0$
in our index notation. Concerning (ii) from Definition \ref{def:pseudonatural}, 
we note that $\xi_{M,M^\prime}(h)$ as defined 
in \eqref{eqn:zetatransformation} is clearly $\bbK$-linear in the argument $h$.
\sk

Property (2) from Definition \ref{def:pseudonatural} follows immediately
from $\dd_{\dR}\oone=0$. To verify the first property in (1),
we use also the index notations $h(\oone \otimes w_i) =:\und{h}_{i}^{~{j}}\otimes w^{\prime}_{j}$
and $\dd_{M^\prime}(\oone\otimes w^\prime_{i})=:\und{M}_{i}^{\prime~j}\otimes w^\prime_j$,
where $\{w^\prime_i\in W^{\prime\sharp}\}$ is a choice of basis for the generators
of $M^\prime$, and compute
\begin{flalign}
\nn \dd_{\hom_A}\big(\xi_{M,M^\prime}(h)\big)\big(\oone\otimes w_i\big)\,&=\,
\dd_{\Omega_A[1]\otimes_A M^\prime}\big(\xi_{M,M^\prime}(h)\big(\oone\otimes w_i\big)\big)
+ \xi_{M,M^\prime}(h)\big(\dd_M(\oone\otimes w_i)\big)\\
\nn \, &=\,s^{-1}\,\Big(\dd_{\Omega_A\otimes_A M^\prime}\big(\dd_{\dR}(\und{h}_i^{~j})\otimes w^\prime_j\big)
- \und{M}_i^{~j}\,\dd_{\dR}(\und{h}_j^{~k})\otimes w^\prime_k\Big)\\
\nn \, &=\,s^{-1}\,\Big(\dd_{\dR}\big(\dd_A(\und{h}_i^{~j})\big)\otimes w^\prime_j
+ (-1)^{\vert \und{h}_i^{~j}\vert}\,\dd_{\dR}(\und{h}_i^{~j})\,\und{M}_j^{\prime~k}\otimes w^\prime_k\\
\nn &\qquad\qquad
- \und{M}_i^{~j}\,\dd_{\dR}(\und{h}_j^{~k})\otimes w^\prime_k\Big)\\
\nn \,&=\,s^{-1}\Big(- (-1)^{\vert \und{h}_i^{~j}\vert}\,\und{h}_i^{~j}\,\dd_{\dR}(\und{M}_j^{\prime~k})\otimes w^\prime_k
+ \dd_{\dR}(\und{M}_i^{~j})\,\und{h}_j^{~k}\otimes w^\prime_k\Big)\\
\,&=\, (\id_{\Omega_A[1]} \otimes_A h)\big(\xi_M(\oone\otimes w_i)\big) - 
\xi_{M^\prime}\big(h(\oone\otimes w_i)\big)\quad,
\end{flalign}
where the fourth step follows from expressing
the cocycle condition $\dd_{\hom_A}(h)\big(\oone\otimes w_i\big)=0$ in our index notation.
To verify the second property in (1), one performs a similar calculation for 
$(k: M\to M^\prime)\in {}_A\CC(M,M^\prime)^{-1}$ of degree $-1$ and finds that 
\begin{flalign}
\nn &(\id_{\Omega_A[1]} \otimes_A k)\big(\xi_M(\oone\otimes w_i)\big) - 
\xi_{M^\prime}\big(k(\oone\otimes w_i)\big) -\xi_{M,M^\prime}\big(\dd_{\hom_A}(k)\big)\big(\oone\otimes w_i\big)\\
&\quad \,=\, s^{-1}\Big(\dd_{\dR}\big(\dd_A(\und{k}_i^{~j})\big)\otimes w_j^\prime + 
(-1)^{\vert \und{k}_i^{~j}\vert}\, \dd_{\dR}(\und{k}_i^{~j})\, \und{M}_j^{\prime~k}\otimes w_k^\prime + 
(-1)^{\vert \und{M}_i^{~j}\vert}\, \und{M}_{i}^{~j}\,\dd_{\dR}(\und{k}_j^{~k})\otimes w^\prime_k\Big)\quad.
\end{flalign}
The right-hand side can be identified as the differential $\dd_{\hom_A}(q)(\oone\otimes w_i)$ of the 
$2$-homotopy $q\in \hom_A(M,\Omega_A[1]\otimes_A M^\prime)^{-2}$ which is defined on generators by
$q(\oone\otimes w) = -s^{-1}\,(\dd_\dR\otimes \id)\big(k(\oone\otimes w)\big)$,
hence the right-hand side is zero in our truncated cochain complex ${}_A\CC\big(M,\Omega_A[1]\otimes_A M^\prime\big)$.
\sk

It remains to verify property (3) from Definition \ref{def:pseudonatural}.
Given $(h: M\to M^\prime)\in {}_A\CC(M,M^\prime)^{0}$ and $(h^\prime : M^\prime\to M^{\prime\prime})
\in {}_A\CC(M^\prime,M^{\prime\prime})^{0}$, we introduce one more index notation
$h^\prime(\oone\otimes w_i^\prime)=: \und{h}_i^{\prime~j}\otimes w_j^{\prime\prime}$, where 
$\{w^{\prime\prime}_i\in W^{\prime\prime\sharp}\}$ is a choice of basis for the generators
of $M^{\prime\prime}$, and compute
\begin{flalign}
\nn \xi_{M,M^{\prime\prime}}(h^\prime\,h)\big(\oone\otimes w_i\big)\,&=\,
-s^{-1}\,\dd_{\dR}\big(\und{h}_{i}^{~j}\,\und{h}_j^{\prime ~k}\big)\otimes w_k^{\prime\prime}\\
\nn \,&=\,-s^{-1}\Big(\dd_{\dR}(\und{h}_{i}^{~j})\,\und{h}_j^{\prime ~k}\otimes w_k^{\prime\prime} +
\und{h}_{i}^{~j}\,\dd_{\dR}(\und{h}_j^{\prime ~k})\otimes w_k^{\prime\prime}\Big)\\
\,&=\,\xi_{M^\prime,M^{\prime\prime}}(h^\prime)\big(h(\oone\otimes w_i)\big) + 
(\id_{\Omega_A[1]} \otimes_A h^\prime)\big(\xi_{M,M^{\prime}}(h)(\oone\otimes w_i)\big)\quad.
\end{flalign}
This concludes the proof.
\end{proof}

In addition to the $\Ch_\bbK^{[-1,0]}$-enriched pseudo-natural transformation
$\xi$ from \eqref{eqn:zetatransformation}, we require two more \textit{strict}
$\Ch_\bbK^{[-1,0]}$-enriched natural transformations to build our infinitesimal $2$-braiding.
First, using the original symmetric braiding $\gamma$, we obtain a family of isomorphisms
\begin{subequations}
\begin{flalign}
\id\otimes_A\gamma_{M,\Omega_A[1]}\otimes_A\id\,:\,\Omega_A[1]\otimes_A M \otimes_A \Omega_A[1]\otimes_A N
~\longrightarrow~ \Omega_A[1]\otimes_A \Omega_A[1]\otimes_A M \otimes_A N
\end{flalign}
in ${}_A\CC$, for all $M,N\in {}_A\CC$, which 
defines a strict $\Ch_\bbK^{[-1,0]}$-enriched natural isomorphism
\begin{flalign}
\otimes_A \ast\Big(\big(\Omega_A[1]\otimes_A(\,\cdot\,)\big)\boxtimes\big(\Omega_A[1]\otimes_A(\,\cdot\,)\big) \Big)
~\stackrel{\cong}{\Longrightarrow}~\Big(\Omega_A[1]\otimes_A\Omega_A[1]\otimes_A(\,\cdot\,)\Big)\ast \otimes_A\quad.
\end{flalign}
\end{subequations}
Second, consider the extension 
$\langle\,\cdot\,,\,\cdot\,\rangle :\Sym_A^2\big(\mathsf{T}_{\! A}[-1]\big)\otimes_A \Omega_A[1]\otimes_A\Omega_A[1]
\to A$ of the duality pairing \eqref{eqn:dualitypairing} defined by
\begin{multline}\label{eqn:extendedpairing}
\big\langle sD\,sD^\prime, s^{-1}\omega\otimes_A s^{-1}\omega^\prime\big\rangle \,:=\,
(-1)^{\vert\omega\vert + \vert D\vert}\,\big\langle D, \langle D^\prime,\omega\rangle\cdot \omega^\prime\big\rangle 
\\[4pt]
+ (-1)^{(\vert D\vert +1) (\vert D^\prime\vert +1)}\,(-1)^{\vert\omega\vert + \vert D^\prime\vert}\,
\big\langle D^\prime, \langle D,\omega\rangle\cdot \omega^\prime\big\rangle
\quad,
\end{multline} 
for all homogeneous $D, D^\prime\in \mathsf{T}_{\! A}$ 
and $\omega,\omega^\prime\in\Omega_A$.
Considering the weight $2$ component 
$\pi^{(2)}\in\Sym_A^2\big(\mathsf{T}_{\! A}[-3]\big)^4\cong \Sym_A^2\big(\mathsf{T}_{\! A}[-1]\big)^0$
of the $2$-shifted Poisson structure (recall that this is a cocycle, see Remark \ref{rem:nshiftedpoisson}),
we obtain a family of morphisms
\begin{subequations}
\begin{flalign}
\langle \pi^{(2)},\,\cdot\,\rangle\otimes_A\id\otimes_A\id 
\,:\, \Omega_A[1]\otimes_A\Omega_A[1]\otimes_A M\otimes_A N~\longrightarrow~A\otimes_A M\otimes_A N \,\cong\,
M\otimes_A N
\end{flalign}
in ${}_A\CC$, for all $M,N\in {}_A\CC$, which defines a 
strict $\Ch_\bbK^{[-1,0]}$-enriched natural transformation
\begin{flalign}
\iota_{\pi}\,:\, \Big(\Omega_A[1]\otimes_A\Omega_A[1]\otimes_A(\,\cdot\,)\Big)\ast \otimes_A~\Longrightarrow~
\otimes_A\quad.
\end{flalign}
\end{subequations}
Using the $\circ$-composition of 
$\Ch_\bbK^{[-1,0]}$-enriched pseudo-natural transformations from Appendix \ref{app:ChCat},
we define the $\Ch_\bbK^{[-1,0]}$-enriched pseudo-natural transformation 
\begin{flalign}\label{eqn:infinitesimalbraidingfrom2shifted}
\begin{gathered}
\resizebox{.9\hsize}{!}{$
\xymatrix@C=0.1em{
\ar@{=}[d]\otimes_A\ar@{==>}[rr]^-{t} ~&~ ~&~\otimes_A\\
\otimes_A\ast\big(\id\boxtimes\id\big) \ar@{=>}[dr]_-{\Id\ast (\xi\boxtimes\xi)~~~~} ~&~ ~&~ 
\Big(\Omega_A[1]\otimes_A\Omega_A[1]\otimes_A(\,\cdot\,)\Big)\ast \otimes_A\ar@{=>}[u]_-{\iota_{\pi}}\\
~&~\otimes_A \ast\Big(\big(\Omega_A[1]\otimes_A(\,\cdot\,)\big)\boxtimes\big(\Omega_A[1]\otimes_A(\,\cdot\,)\big)\Big) \ar@{=>}[ru]_-{~~\cong}~&~
}$}
\end{gathered}
\end{flalign}
as our candidate for an infinitesimal $2$-braiding on ${}_A\CC$.
\begin{theo}\label{theo:mainresult}
Given any $2$-shifted Poisson structure $\pi\in \widehat{\Pol}(A,2)^4$ on a finitely generated
semi-free CDGA $A$ over the field $\bbK$, the $\Ch_\bbK^{[-1,0]}$-enriched pseudo-natural transformation $t$ 
from \eqref{eqn:infinitesimalbraidingfrom2shifted} defines a $\gamma$-equivariant
infinitesimal $2$-braiding in the sense of Definitions \ref{def:inf2braiding}
and \ref{def:inf2braidinggammaequivariant}
on the homotopy $2$-category ${}_A\CC =\mathsf{Ho}_2\big({}_A\dgMod^{\mathrm{fgsf}}\big)\in 
\dgCat_\bbK^{[-1,0],\mathrm{ps}}$ of the dg-category of finitely generated semi-free $A$-dg-modules.
\end{theo}
\begin{rem}\label{rem:mainresult}
We would like to highlight that the infinitesimal $2$-braiding \eqref{eqn:infinitesimalbraidingfrom2shifted}
depends only on the bivector component $\pi^{(2)}$ of the $2$-shifted Poisson structure 
$\pi = \sum_{m\geq 2}\pi^{(m)}\in \widehat{\Pol}(A,2)^4$. This is to be expected
since infinitesimal $2$-braidings are characterized by the infinitesimal hexagon identities 
from Definition \ref{def:inf2braiding}, which are linear in $t$. Hence, the higher-weight
components $\pi^{(3)}, \pi^{(4)},\dots$ of the $2$-shifted Poisson structure 
are not expected to enter at this linear level, since they control  
non-linear aspects through the Maurer-Cartan equations from Remark \ref{rem:nshiftedpoisson},
which are non-linear identities of the form 
$\dd_{\widehat{\Pol}}(\pi^{(3)}) + \frac{1}{2}\{\pi^{(2)},\pi^{(2)}\}=0$.
However, we expect that these higher-weight components 
and their Maurer-Cartan identities will play an important role
for deformation quantizations at higher orders in $\hbar$.
\end{rem}
\begin{proof}[Proof of Theorem \ref{theo:mainresult}]
The most straightforward proof of this result is to compute
the components of the $\Ch_\bbK^{[-1,0]}$-enriched pseudo-natural transformation $t$
from \eqref{eqn:infinitesimalbraidingfrom2shifted}, and then to verify directly that 
these components satisfy the $\gamma$-equivariance property from \eqref{eqn:gammaequivariancecomponents}
and the infinitesimal hexagons from Remark \ref{rem:infbraidingcomponents}.
Using the composition operations from Appendix \ref{app:ChCat}, one finds
by a straightforward computation that
\begin{subequations}\label{eqn:tfromPoissoncomponents}
\begin{flalign}\label{eqn:tfromPoissoncomponents1}
\begin{gathered}
\xymatrix@C=6.5em@R=3em{
M\otimes_A N \ar[d]_-{\xi_M\otimes_A \xi_N} \ar[r]^-{t_{M,N}}~&~M\otimes_A N \\
\Omega_A[1]\otimes_A M\otimes_A \Omega_A[1]\otimes_A N  \ar[r]_-{\id\otimes_A\gamma_{M,\Omega_A[1]}\otimes_A \id}~&~
\Omega_A[1]\otimes_A \Omega_A[1]\otimes_A M\otimes_A N \ar[u]_-{\langle\pi^{(2)},\,\cdot\,\rangle\otimes_A\id\otimes_A\id}
}
\end{gathered}\qquad,
\end{flalign}
for all $M,N\in{}_A\CC$, and
\begin{flalign}\label{eqn:tfromPoissoncomponents2}
\resizebox{.9\hsize}{!}{$
\begin{gathered}
\xymatrix@C=6.5em@R=4em{
M\otimes_A N \ar[d]_-{\substack{\xi_{M,M^\prime}(h)\otimes_A (\id\otimes_A k)\xi_N\\+ \xi_{M^\prime}h\otimes_A \xi_{N,N^\prime}(k)}} \ar[r]^-{t_{(M,N),(M^\prime,N^\prime)}(h\widetilde{\otimes}k)}~&~M^\prime\otimes_A N^\prime \\
\Omega_A[1]\otimes_A M^\prime \otimes_A \Omega_A[1]\otimes_A N^\prime  
\ar[r]_-{\id\otimes_A\gamma_{M^\prime,\Omega_A[1]}\otimes_A \id}~&~
\Omega_A[1]\otimes_A \Omega_A[1]\otimes_A M^\prime \otimes_A N^\prime \ar[u]_-{\langle\pi^{(2)},\,\cdot\,\rangle\otimes_A\id\otimes_A\id}
}
\end{gathered}\qquad,
$}
\end{flalign}
\end{subequations}
for all $M,N,M^\prime,N^\prime\in{}_A\CC$, $h\in{}_A\CC(M,M^\prime)^0$ and $k\in{}_A\CC(N,N^\prime)^0$.
\sk

The $\gamma$-equivariance property \eqref{eqn:gammaequivariancecomponents1}
for the degree $0$ components $t_{M,N}$ follows immediately by using 
symmetry of the extended pairing \eqref{eqn:extendedpairing}, i.e.\
$\langle\pi^{(2)},\,\cdot\,\rangle\,\gamma  = \langle\pi^{(2)},\,\cdot\,\rangle$,
and the hexagon identities for the original symmetric braiding $\gamma$.
To prove the $\gamma$-equivariance property \eqref{eqn:gammaequivariancecomponents2}
for the degree $-1$ components $t_{(M,N),(M^\prime,N^\prime)}(h\widetilde{\otimes}k)$,
one has to use additionally that
\begin{flalign}
\nn &\gamma_{\Omega_A[1]\otimes_A M^\prime, \Omega_A[1]\otimes_A N^\prime}\,\Big(\xi_{M,M^\prime}(h)\otimes_A (\id\otimes_A k)\xi_N+ \xi_{M^\prime}h\otimes_A \xi_{N,N^\prime}(k)\Big)\,\gamma_{N,M}\\
\nn &\qquad\, =\, (\id\otimes_A k)\xi_N \otimes_A \xi_{M,M^\prime}(h) + 
\xi_{N,N^\prime}(k) \otimes_A \xi_{M^\prime}h\\
\nn &\qquad\, =\,\xi_{N^\prime} k \otimes_A \xi_{M,M^\prime}(h) + 
\xi_{N,N^\prime}(k) \otimes_A (\id\otimes_A h) \xi_{M} + \dd_{\hom_A}\big(\xi_{N,N^\prime}(k)\otimes_A\xi_{M,M^\prime}(h)\big)\\
&\qquad\, =\,\xi_{N^\prime} k \otimes_A \xi_{M,M^\prime}(h) + 
\xi_{N,N^\prime}(k) \otimes_A (\id\otimes_A h) \xi_{M} \quad,
\end{flalign}
where in the second step we used property (1) of the Definition \ref{def:pseudonatural}
of $\Ch_\bbK^{[-1,0]}$-enriched pseudo-natural transformations and in the last step
we used that the morphism complexes of ${}_A\CC$ are by definition truncated to degrees $\{-1,0\}$,
hence the $2$-homotopy in the last term gets identified with $0$.
\sk

It remains to verify the first infinitesimal hexagon identity \eqref{eqn:hexagon1infexplicit},
since the second one \eqref{eqn:hexagon2infexplicit} then 
follow from Proposition \ref{prop:equivariancehexagons} and $\gamma$-equivariance.
For this we recall the definition of the components of $\xi$ from \eqref{eqn:zetatransformation}
and observe that
\begin{subequations}
\begin{flalign}
\xi_{N\otimes_A L} \,=\, \xi_N\otimes_A\id + (\gamma_{N,\Omega_A[1]}\otimes_A\id)\,(\id\otimes_A\xi_L)\quad,
\end{flalign}
for all $N,L\in{}_A\CC$, and
\begin{flalign}
\xi_{N\otimes_A L,N^\prime\otimes_A L^\prime}(k\otimes_A l)\,=\, 
 \xi_{N,N^\prime}(k)\otimes_A l + (\gamma_{N^\prime,\Omega_A[1]}\otimes_A\id)\,(k\otimes_A\xi_{L,L^\prime}(l))\quad,
\end{flalign}
\end{subequations}
for all $N,L,N^\prime,L^\prime\in{}_A\CC$, $k\in{}_A\CC(N,N^\prime)^0$ and $l\in{}_A\CC(L,L^\prime)^0$.
Using also the hexagon identities for the original symmetric braiding $\gamma$,
one verifies the infinitesimal hexagon identity \eqref{eqn:hexagon1infexplicit}.
\end{proof}

\begin{rem}\label{rem:tfromPoissonexplicitdescription}
For later use in our discussion of examples in 
Section \ref{sec:examples}, let us expand the components
$t_{M,N}$ and $t_{(M,N),(M^\prime,N^\prime)}(h\widetilde{\otimes}k)$
of the infinitesimal $2$-braiding \eqref{eqn:infinitesimalbraidingfrom2shifted}
in terms of a choice of bases for the finitely generated
semi-free $A$-dg-modules $M,N,M^\prime,N^\prime\in{}_A\CC$.
Using similar notations as in the proof of Lemma \ref{lem:xipseudonatural},
we denote by $\{w_i\in W^\sharp\}$ a basis for $M^\sharp\cong A^\sharp\otimes W^\sharp$, 
by $\{u_q\in U^\sharp\}$ a basis for $N^\sharp \cong A^\sharp\otimes U^\sharp$,
by $\{w_i^\prime\in W^{\prime \sharp}\}$ a basis for $M^{\prime\sharp}\cong A^\sharp\otimes W^{\prime\sharp}$, 
and by $\{u_q^\prime\in U^{\prime\sharp}\}$ a basis for $N^{\prime\sharp} \cong A^\sharp\otimes U^{\prime\sharp}$.
Making use of these bases, we write
\begin{subequations}
\begin{align}
\dd_M(\oone\otimes w_i) \,&=\, \und{M}_{i}^{~j}\otimes w_j\quad, &
\dd_N(\oone\otimes u_q) \,&=\, \und{N}_q^{~r}\otimes u_r\quad,\\
\dd_{M^\prime}(\oone\otimes w_i^\prime) 
\,&=\, \und{M}_{i}^{\prime~j}\otimes w_j^\prime\quad, & 
\dd_{N^\prime}(\oone\otimes u_q^\prime)
\,&=\, \und{N}_{q}^{\prime~r}\otimes u_r^\prime \quad,
\end{align}
for the differentials, 
and for the $A$-linear maps $h\in{}_A\CC(M,M^\prime)^0$ and 
$k\in{}_A\CC(N,N^\prime)^0$
we write 
\begin{flalign}
h(\oone\otimes w_i) \,=\, \und{h}_{i}^{~j}\otimes w_j^\prime\quad,\qquad 
 \qquad \quad k(\oone\otimes u_q) \, =\, \und{k}_{q}^{~r}\otimes u_r^\prime \quad,
\end{flalign}
\end{subequations}
where again all summations over repeated indices
are suppressed. Evaluating the diagram \eqref{eqn:tfromPoissoncomponents1}, we then find the explicit basis
expression
\begin{subequations}\label{eqn:tfromPoissonbasis}
\begin{flalign}\label{eqn:tfromPoissonbasis1}
\nn &t_{M,N}\big(\oone\otimes w_i\otimes u_q\big)\,=\, (t_{M,N})_{iq}^{~~jr}\otimes w_j\otimes u_r  \\
&\qquad\,=\,(-1)^{\vert w_j\vert\,(\vert \und{N}_q^{~r}\vert-1)}~
\big\langle \pi^{(2)}, s^{-1}\dd_{\dR}(\und{M}_{i}^{~j})\otimes_A s^{-1}\dd_{\dR}(\und{N}_{q}^{~r})\big\rangle\otimes w_j\otimes u_r\quad.
\end{flalign}
Furthermore, evaluating the diagram \eqref{eqn:tfromPoissoncomponents2}, we find
\begin{flalign}\label{eqn:tfromPoissonbasis2}
\nn &t_{(M,N),(M^\prime,N^\prime)}(h\widetilde{\otimes}k)(\oone\otimes w_i\otimes u_q)\,=\,
\big(t_{(M,N),(M^\prime,N^\prime)}(h\widetilde{\otimes}k)\big)_{iq}^{~~jr}\otimes
w_j^\prime\otimes u_r^\prime\\
\nn &\qquad \,=\,-(-1)^{\vert w_j^\prime\vert \,(\vert \und{N}_{q}^{~s}\vert + \vert \und{k}_s^{~r}\vert -1)}~
\big\langle \pi^{(2)}, s^{-1}\dd_{\dR}(\und{h}_i^{~j}) \otimes_A s^{-1}\dd_{\dR}(\und{N}_{q}^{~s})\,\und{k}_s^{~r}\big\rangle\otimes w_j^\prime\otimes u_r^\prime\\
&\qquad ~\, \quad -(-1)^{\vert \und{h}_i^{~k}\vert + \vert w_j^\prime\vert\,(\vert \und{k}_{q}^{~r}\vert-1)}
~\big\langle \pi^{(2)}, s^{-1}\und{h}_{i}^{~k}\,\dd_{\dR}(\und{M}_k^{\prime~j})\otimes_A s^{-1}\dd_{\dR}(\und{k}_q^{~r})\big\rangle\otimes
w_j^\prime \otimes u_r^\prime\quad.
\end{flalign}
\end{subequations}
Given a concrete example for the bivector
$\pi^{(2)}\in\Sym_A^2\big(\mathsf{T}_{\! A}[-3]\big)^4\cong \Sym_A^2\big(\mathsf{T}_{\! A}[-1]\big)^0$,
one can evaluate these expressions further by using the definition \eqref{eqn:extendedpairing}
of the duality pairing.
\end{rem}

\subsection{Towards higher-order deformations}
\label{subsec:higherorders}
A powerful result in the $1$-categorical theory of infinitesimal braidings
is that, upon the choice of a Drinfeld associator \cite{Drinfeld90}, every infinitesimal
braiding lifts to a formal deformation (i.e.\ to all orders in $\hbar$)
of the given symmetric monoidal category into a braided monoidal one.
See e.g.\ \cite{Cartier} and \cite[Chapters XIX--XX]{Kassel} for the details.
The aim of this subsection is to explore which aspects of this construction admit 
a direct generalization to our concept of infinitesimal $2$-braidings on ${}_A\CC$. 
We observe that a direct generalization is only possible under rather 
restrictive assumptions (see Proposition \ref{propo:allorderdeformation}
and Remark \ref{rem:allorderdeformation}), while the general case points towards the necessity
to develop a homotopy-coherent variant of Drinfeld associators 
which also describe pentagonator and hexagonator modifications.
Our study in this subsection has some parallels with \cite[Section 2.2]{Joao},
however we believe that the relations for the modifications in 
\eqref{eqn:higherordermodifications} entering our explicit attempt towards
a deformation construction at order $\hbar^2$ are new.
\sk

To study formal deformations to all orders in $\hbar$, we introduce
the commutative $\bbK$-algebra $\bbK[[\hbar]]$ of formal power series and perform
a base change 
\begin{flalign}
{}_A\CC[[\hbar]]\,:=\,{}_A\CC \otimes \bbK[[\hbar]]\,\in\,\dgCat_{\bbK[[\hbar]]}^{[-1,0],\mathrm{ps}}
\end{flalign}
along the commutative $\bbK$-algebra morphism $\bbK\to \bbK[[\hbar]]$.
The goal is then to deform the symmetric monoidal structure on ${}_A\CC[[\hbar]]$
into a braided monoidal structure. In order to simplify notations,
we suppress in what follows the original associator $\alpha$ and simply write
\begin{flalign}
\otimes_A^2 \,:=\, \otimes_A\ast(\otimes_A\boxtimes \id) \,\cong\, \otimes_A\ast(\id\boxtimes \otimes_A) \,:\, 
{}_A\CC[[\hbar]]\boxtimes {}_A\CC[[\hbar]]\boxtimes {}_A\CC[[\hbar]]~\longrightarrow~ {}_A\CC[[\hbar]]
\end{flalign}
for the $2$-times iterated tensor product, and similarly
$\otimes_A^n : {}_A\CC[[\hbar]]^{\boxtimes(n+1)}\to {}_A\CC[[\hbar]]$ for the $n$-times iterated tensor product.
Given any $\gamma$-equivariant infinitesimal $2$-braiding $t : \otimes_A\Rightarrow\otimes_A$,
we denote by
\begin{subequations}\label{eqn:t_ij}
\begin{flalign}
t_{ij}\,:\,\otimes_A^{n}~\Longrightarrow~\otimes_A^{n}\,:\,  {}_A\CC[[\hbar]]^{\boxtimes(n+1)}~\longrightarrow~ {}_A\CC[[\hbar]]
\end{flalign}
the $\Ch_{\bbK[[\hbar]]}^{[-1,0]}$-enriched pseudo-natural transformation which evaluates
$t$ on the $i$-th and $j$-th tensor factor, for all $i,j\in\{1,\dots,n+1\}$ with $i\neq j$.
Explicitly, the components of $t_{ij}$ read as
\begin{flalign}
\begin{gathered}
\xymatrix@C=13em{
\ar[d]_-{\cong}\bigotimes\limits_{k=1}^{n+1} M_k \ar[r]^-{(t_{ij})_{M_1,\dots,M_{n+1}}}~&~ \bigotimes\limits_{k=1}^{n+1} M_k\\
M_i\otimes_A M_j\otimes_A \bigotimes\limits_{\substack{k=1\\k\neq i,j}}^{n+1} M_k\ar[r]_-{t_{M_i,M_j}\otimes_A\bigotimes_k\id_{M_k}}~&~
M_i\otimes_A M_j\otimes_A \bigotimes\limits_{\substack{k=1\\k\neq i,j}}^{n+1} M_k\ar[u]_-{\cong}
}
\end{gathered}\quad,
\end{flalign}
for all $M_1,\dots,M_{n+1}\in {}_A\CC[[\hbar]]$, and
\begin{flalign}
\begin{gathered}
\xymatrix@C=13em{
\ar[d]_-{\cong}\bigotimes\limits_{k=1}^{n+1} M_k \ar[r]^-{(t_{ij})_{(M_1,\dots,M_{n+1}),(M_1^\prime,\dots,M_{n+1}^\prime)}(h_1\widetilde{\otimes}\cdots \widetilde{\otimes}h_{n+1})}~&~ \bigotimes\limits_{k=1}^{n+1} M^\prime_k\\
M_i\otimes_A M_j\otimes_A \bigotimes\limits_{\substack{k=1\\k\neq i,j}}^{n+1} M_k\ar[r]_-{t_{(M_i,M_j),(M_i^\prime,M_j^\prime)}(h_i\widetilde{\otimes}h_j)\otimes_A\bigotimes_k h_k}~&~
M^\prime_i\otimes_A M^\prime_j\otimes_A \bigotimes\limits_{\substack{k=1\\k\neq i,j}}^{n+1} M^\prime_k\ar[u]_-{\cong}
}
\end{gathered}\quad,
\end{flalign}
\end{subequations}
for all $M_1,\dots,M_{n+1},M^\prime_1,\dots,M^\prime_{n+1}\in {}_A\CC[[\hbar]]$
and $h_k\in {}_A\CC[[\hbar]](M_k,M_k^\prime)^0$, for $k=1,\dots,n+1$,
where the unlabeled vertical isomorphisms are given by the original symmetric braiding $\gamma$.
Let us note that there is an obvious generalization
of the above construction which assigns to every pair 
$I=(i_1\cdots i_I)$ and $J=(j_1\cdots j_J)$ of tuples of pairwise distinct indices in $\{1,\dots,n+1\}$
the $\Ch_{\bbK[[\hbar]]}^{[-1,0]}$-enriched pseudo-natural transformation 
$t_{IJ}\,:\,\otimes_A^{n}~\Longrightarrow~\otimes_A^{n}$ which evaluates $t$ on the 
tensor product of the $I$-factors and the tensor product of the $J$-factors.
In this compact notation, the infinitesimal hexagons from \eqref{eqn:hexagoninf} 
and the $\gamma$-equivariance property from Definition \ref{def:inf2braidinggammaequivariant}
read as
\begin{flalign}\label{eqn:t_ij-identities}
t_{i(jk)}\,=\,t_{ij} + t_{ik}\quad,\qquad
t_{(ij)k}\,=\, t_{ik} + t_{jk}\quad,\qquad t_{ij} \,=\,t_{ji} \quad,
\end{flalign}
for all pairwise distinct $i,j,k\in\{1,\dots,n+1\}$.
Furthermore, one verifies that 
\begin{flalign}\label{eqn:t_ij-t_kl-relations}
[t_{ij},t_{kl}]\,:=\,t_{ij}\circ t_{kl} - t_{kl}\circ t_{ij}\,=\,0\quad,
\end{flalign}
for all pairwise distinct $i,j,k,l\in\{1,\dots,n+1\}$, by explicitly computing the components.
\sk

Following the same strategy as in \cite{Cartier} and \cite[Chapters XIX--XX]{Kassel}, 
we consider a candidate for a deformed associator and braiding
on ${}_A\CC[[\hbar]]$ which is given by the ansatz
\begin{subequations}\label{eqn:deformedassociator+braiding}
\begin{flalign}\label{eqn:deformedassociator}
\alpha^\hbar\,:=\, \Phi(t_{12},t_{23})\,:\,\otimes_A^2~\Longrightarrow~\otimes_A^2\,:\, 
{}_A\CC[[\hbar]]\boxtimes {}_A\CC[[\hbar]]\boxtimes {}_A\CC[[\hbar]]~\longrightarrow~{}_A\CC[[\hbar]]
\end{flalign}
and
\begin{flalign}\label{eqn:deformedbraiding}
\gamma^\hbar\,:=\gamma\circ e^{\frac{\hbar}{2} t}\,:\, 
\otimes_A~\Longrightarrow~\otimes_A\ast\tau\,:\, 
{}_A\CC[[\hbar]]\boxtimes {}_A\CC[[\hbar]]~\longrightarrow~{}_A\CC[[\hbar]]\quad,
\end{flalign}
\end{subequations}
where $\Phi(\mathfrak{a},\mathfrak{b})$ is some formal power series in 
two non-commutative variables $\mathfrak{a}$ and $\mathfrak{b}$ satisfying $\Phi(0,0)=1$.
In these expressions, both $\Phi$ and the exponential function $e$, which we describe by its formal Taylor series,
are evaluated by using the $\circ$-composition of pseudo-natural transformations from 
Appendix \ref{app:ChCat}, hence $\alpha^\hbar$ and $\gamma^\hbar$ are by construction
$\Ch_{\bbK[[\hbar]]}^{[-1,0]}$-enriched pseudo-natural transformations.
\begin{propo}\label{propo:allorderdeformation}
The tuple $\big({}_A\CC[[\hbar]],\otimes_A,A,\alpha^\hbar,\lambda,\rho,\gamma^\hbar\big)$,
with deformed associator $\alpha^\hbar$ and braiding $\gamma^\hbar$ 
from \eqref{eqn:deformedassociator+braiding}, defines a semi-strict braided monoidal category 
object in $\dgCat_{\bbK[[\hbar]]}^{[-1,0],\mathrm{ps}}$  in the sense of Definition \ref{def:2TermBraidedCat}
if and only if all of the following hold true:
\begin{itemize}
\item[(1)] The identity
\begin{flalign}\label{eqn:Drinfeldassociator}
\Phi(t_{12},t_{23}+ t_{24})\circ \Phi(t_{13}+t_{23},t_{34})\,=\,
\Phi(t_{23},t_{34})\circ \Phi(t_{12}+t_{13},t_{24}+t_{34})\circ \Phi(t_{12},t_{23})
\end{flalign}
holds true strictly in the algebra $\big(\mathrm{PsNat}(\otimes_A^3,\otimes_A^3),\circ,\Id_{\otimes_A^3}\big)$
of $\Ch_{\bbK[[\hbar]]}^{[-1,0]}$-enriched pseudo-natural transformations with multiplication $\circ$.

\item[(2)] The identities
\begin{subequations}\label{eqn:hexagonhbar}
\begin{flalign}\label{eqn:hexagonhbar1}
\Phi(t_{23},t_{13})\circ e^{\frac{\hbar}{2}(t_{12}+t_{13})}\circ \Phi(t_{12},t_{23})
\,=\, e^{\frac{\hbar}{2}t_{13}}\circ \Phi(t_{12},t_{13})\circ e^{\frac{\hbar}{2}t_{12}}
\end{flalign}
and
\begin{flalign}\label{eqn:hexagonhbar2}
\Phi(t_{13},t_{12})^{-1} \circ e^{\frac{\hbar}{2}(t_{13}+t_{23})}\circ \Phi(t_{12},t_{23})^{-1}
\,=\, e^{\frac{\hbar}{2}t_{13}}\circ \Phi(t_{13},t_{23})^{-1}\circ e^{\frac{\hbar}{2}t_{23}}
\end{flalign}
\end{subequations}
hold true strictly in the algebra $\big(\mathrm{PsNat}(\otimes_A^2,\otimes_A^2),\circ,\Id_{\otimes_A^2}\big)$
of $\Ch_{\bbK[[\hbar]]}^{[-1,0]}$-enriched pseudo-natural transformations with multiplication $\circ$.
\end{itemize}
Assuming additionally that $\Phi$ satisfies the inversion property 
$\Phi(\mathfrak{a},\mathfrak{b})^{-1}=\Phi(\mathfrak{b},\mathfrak{a})$, then the two identities
in \eqref{eqn:hexagonhbar} are equivalent.
\end{propo}
\begin{proof}
The triangle identity holds true automatically as a consequence of $\Phi(0,0)=1$
and the fact that all components of $t$ involving the monoidal unit $A\in{}_A\CC[[\hbar]]$ at least once 
are necessarily trivial. (The latter is a consequence of \eqref{eqn:t_ij-identities}.) Using \eqref{eqn:t_ij-identities}
and arguments for strictness of the interchange law for $\circ$ and $\ast$
which are similar to the ones in the proof of Proposition \ref{prop:hexagoninf}, one shows that (1) is equivalent 
to the pentagon identity for $\alpha^\hbar$. With the same kind of computation,
one shows that (2) is equivalent to the two hexagon identities for $\gamma^\hbar$ and $\alpha^\hbar$.
Finally, under the additional assumption $\Phi(\mathfrak{a},\mathfrak{b})^{-1}=\Phi(\mathfrak{b},\mathfrak{a})$,
one easily shows that \eqref{eqn:hexagonhbar1} and \eqref{eqn:hexagonhbar2} are identified
by permuting the indices according to $123\mapsto 321$ and using once again \eqref{eqn:t_ij-identities}.
\end{proof}

\begin{rem}\label{rem:allorderdeformation}
From a superficial perspective, the identities (1) and (2) in Proposition \ref{propo:allorderdeformation}
take exactly the same form as the defining properties of a Drinfeld associator $\Phi$,
see e.g.\ \cite[Eqns.~(24) and (25)]{Cartier} or \cite[Chapter XIX, Eqns.~(8.27)--(8.29)]{Kassel}.
However, there is the following caveat: The  analogous identities for a Drinfeld associator
are formulated in the universal enveloping algebra of the appropriate Drinfeld-Kohno Lie algebra
where the symbols $t_{ij}$ satisfy the infinitesimal braid relations
$t_{ij} = t_{ji}$, $[t_{ij},t_{kl}]=0$ and $[t_{ij},t_{ik}+t_{jk}]=0$, for all pairwise distinct indices $i,j,k,l$.
In contrast to this, our identities
(1) and (2) must hold true in the appropriate algebra of 
$\Ch_{\bbK[[\hbar]]}^{[-1,0]}$-enriched pseudo-natural transformations with multiplication $\circ$.
From \eqref{eqn:t_ij-identities} and \eqref{eqn:t_ij-t_kl-relations},
we know that our $t_{ij}$'s satisfy the first two types of infinitesimal braid relations,
i.e.\ $t_{ij} = t_{ji}$ and $[t_{ij},t_{kl}]=0$ for all pairwise distinct indices $i,j,k,l$,
but as we shall illustrate below they do \textit{not} necessarily satisfy the third
type of relation $[t_{ij},t_{ik}+t_{jk}]=0$ strictly. 
\sk

So the conclusion at this moment is as follows:
Assuming that our $\gamma$-equivariant infinitesimal $2$-braiding $t$ satisfies 
the remaining infinitesimal braid relations $[t_{ij},t_{ik}+t_{jk}]=0$ strictly, then the choice of 
a Drinfeld associator $\Phi$ (in the traditional sense) provides
via Proposition \ref{propo:allorderdeformation} a semi-strict
braided monoidal deformation (to all orders in $\hbar$) 
of the symmetric monoidal category object $\big({}_A\CC,\otimes_A,A,\alpha,\lambda,\rho,\gamma\big)$.
\end{rem}

In light of the previous remark, we make the following observation.
\begin{propo}\label{propo:Gamma4term}
The pseudo-naturality structure of the $\gamma$-equivariant infinitesimal $2$-braiding
$t:\otimes_A\Rightarrow\otimes_A$ induces $\Ch_{\bbK[[\hbar]]}^{[-1,0]}$-enriched modifications
\begin{subequations}\label{eqn:Gamma4term}
\begin{flalign}\label{eqn:Gamma4term1}
\Gamma_{ijk}\,:\,[t_{ij},t_{ik}+t_{jk}]~\xRrightarrow{\,~~\,}~0\,:\, \otimes_A^n~\Longrightarrow~\otimes_A^n\,:\,
{}_A\CC[[\hbar]]^{\boxtimes(n+1)}~\longrightarrow~{}_A\CC[[\hbar]]\quad,
\end{flalign}
for all pairwise distinct $i,j,k\in\{1,\dots,n+1\}$, which enforce weakly the remaining infinitesimal braid relations.
The components (recall Definition \ref{def:modi}) of these modifications are given explicitly by
\begin{flalign}\label{eqn:Gamma4term2}
\resizebox{.9\hsize}{!}{$
\begin{gathered}
\xymatrix@C=17em{
\ar[d]_-{\cong}\bigotimes\limits_{l=1}^{n+1} M_l \ar[r]^-{(\Gamma_{ijk})_{M_1,\dots,M_{n+1}}}~&~ \bigotimes\limits_{l=1}^{n+1} M_l\\
M_i\otimes_A M_j\otimes_A M_k\otimes_A \bigotimes\limits_{\substack{l=1\\l\neq i,j,k}}^{n+1} M_l
\ar[r]_-{t_{(M_i\otimes_A M_j,M_k),(M_i\otimes_A M_j,M_k)} (t_{M_i,M_j}\widetilde{\otimes}\id)\otimes_A\id}~&~
M_i\otimes_A M_j\otimes_A M_k\otimes_A \bigotimes\limits_{\substack{l=1\\l\neq i,j,k}}^{n+1} M_l\ar[u]_-{\cong}
}
\end{gathered}\quad,
$}
\end{flalign}
\end{subequations}
for all $M_1,\dots,M_{n+1}\in {}_A\CC[[\hbar]]$. 
\end{propo}
\begin{proof}
Recalling \eqref{eqn:t_ij-identities}, one observes that 
\begin{flalign}
[t_{ij},t_{ik}+t_{jk}] \,=\, [t_{ij},t_{(ij)k}] \,=\, t_{ij}\circ t_{(ij)k} -t_{(ij)k}\circ t_{ij}\quad.
\end{flalign}
Using the composition formulas from Appendix \ref{app:ChCat}, the component expressions
\eqref{eqn:t_ij} and pseudo-naturality of $t_{(ij)k}$ in the left entry, one computes 
the components of this $\Ch_{\bbK[[\hbar]]}^{[-1,0]}$-enriched pseudo-natural 
transformation and finds that 
\begin{flalign}
\big(t_{ij}\circ t_{(ij)k} -t_{(ij)k}\circ t_{ij}\big)_{M_1,\dots,M_{n+1}} \,=\,\dd_{\hom_A}\big((\Gamma_{ijk})_{M_1,\dots,M_{n+1}}\big)
\end{flalign}
with $\Gamma_{ijk}$ defined in \eqref{eqn:Gamma4term}. By a routine calculation
one then shows that $\Gamma_{ijk}$ satisfies the properties of a 
$\Ch_{\bbK[[\hbar]]}^{[-1,0]}$-enriched modification from Definition \ref{def:modi}.
\end{proof}

\begin{rem}
The reader might have noticed that there are potential ambiguities
in defining the $\Ch_{\bbK[[\hbar]]}^{[-1,0]}$-enriched modifications $\Gamma_{ijk}$.
These arise from the multiple equivalent ways to write the expression $[t_{ij},t_{ik}+t_{jk}]$
by using the identities in \eqref{eqn:t_ij-identities}. 
Such ambiguities do not seem to be present for our $\gamma$-equivariant infinitesimal $2$-braidings.
Indeed, one observes that
\begin{flalign}
\Gamma_{ijk} \,=\,\Gamma_{jik}
\end{flalign}
is symmetric in the first two indices, hence there are no ambiguities
arising from rewriting $[t_{ij},t_{ik}+t_{jk}] = [t_{ji},t_{jk}+t_{ik}]$ by using
the symmetry property $t_{ij} =t_{ji}$.
Furthermore, rewriting $[t_{ij},t_{ik}+t_{jk}] = [t_{ij},t_{(ij)k}] = [t_{ij},t_{k(ij)}]$
by using the identities \eqref{eqn:t_ij-identities}, one finds 
from the pseudo-naturality structure of $t_{k(ij)}$ in the right entry
a modification which agrees with $\Gamma_{ijk}$.
\end{rem}

\begin{rem}\label{rem:outlookcoherence}
For the complete description of a categorified version
of the infinitesimal braid relations, one would also 
require suitable coherence conditions for the modifications $\Gamma_{ijk}$.
One could then attempt to weaken the strict identities (1) and (2) 
from Proposition \ref{propo:allorderdeformation} to identities which hold true up
to suitable modifications, playing the role of a pentagonator and hexagonators
for a weakened variant of the semi-strict braided monoidal category objects from 
Definition \ref{def:2TermBraidedCat}. To the best of our knowledge, this 
is relatively unexplored territory, with the notable exception given by 
exploratory works of Cirio and Faria Martins \cite{Joao}. 
We will not attempt to address and solve these issues in the present paper.
\end{rem}

\begin{ex}\label{ex:Gammaexplicit}
For our infinitesimal $2$-braidings $t$ in Theorem \ref{theo:mainresult},
which are constructed from $2$-shifted Poisson structures $\pi$ on the CDGA $A$, we can use
the formulas from Remark \ref{rem:tfromPoissonexplicitdescription} in order to provide 
a more explicit basis expression for the modifications $\Gamma_{ijk}$
in Proposition \ref{propo:Gamma4term}. Denoting $M=M_i$, $N=M_j$ and $L=M_k$,
we then find for the bottom horizontal arrow in \eqref{eqn:Gamma4term2}
\begin{flalign}
\nn & t_{(M\otimes_A N,L),(M\otimes_A N,L)} (t_{M,N}\widetilde{\otimes}\id)(\oone\otimes w_i\otimes u_q\otimes v_s)\\
&\qquad \,=\, -(-1)^{\vert w_j\vert\,(\vert \und{L}_{s}^{~t}\vert -1)}~
\big\langle \pi^{(2)},
s^{-1}\dd_{\dR}\big((t_{M,N})_{iq}^{~~jr}\big)\otimes_A s^{-1}\dd_{\dR}(\und{L}_{s}^{~t})
\big\rangle\otimes w_j\otimes u_r\otimes v_t\quad,\label{eqn:Gammaexplicit}
\end{flalign}
where the components $(t_{M,N})_{iq}^{~~jr}$ are defined in 
\eqref{eqn:tfromPoissonbasis1} and 
the second term in \eqref{eqn:tfromPoissonbasis2} vanishes because $k=\id:L\to L$
is the identity in the present case.
\end{ex}

To conclude this subsection, we carry out briefly an explicit study of the 
identities (1) and (2) in Proposition \ref{propo:allorderdeformation}
modulo $\hbar^3$ by taking as ansatz the traditional Drinfeld associator
\begin{flalign}
\Phi(\mathfrak{a},\mathfrak{b})\,=\,1 + \tfrac{\hbar^2}{24}\,[\mathfrak{a},\mathfrak{b}]\quad \mathrm{mod}\,\hbar^3\quad,
\end{flalign}
see e.g.\ \cite[Corollary XIX.6.5]{Kassel}. Using the 
identity \eqref{eqn:t_ij-t_kl-relations}, one immediately checks 
that the identity \eqref{eqn:Drinfeldassociator} holds true strictly modulo $\hbar^3$, 
i.e.\ without the need for any modifications.
In contrast to this, the identities \eqref{eqn:hexagonhbar} do not hold true strictly modulo $\hbar^3$.
Concretely, bringing all terms in \eqref{eqn:hexagonhbar1} to the left-hand side, one finds that
\begin{subequations}\label{eqn:higherordermodifications}
\begin{flalign}
\tfrac{\hbar^2}{24}\big([t_{12}, t_{13}+ t_{23}] - [t_{13},t_{12}+ t_{32}]\big)
~\xRrightarrow{\,\tfrac{\hbar^2}{24}(\Gamma_{123}-\Gamma_{132})\,}~0
\end{flalign} 
only holds true weakly up to a certain linear combination of the $\Ch_{\bbK[[\hbar]]}^{[-1,0]}$-enriched modifications
from Proposition \ref{propo:Gamma4term}. Similarly, bringing all terms in \eqref{eqn:hexagonhbar2} to the left-hand side,
one finds that
\begin{flalign}
\tfrac{\hbar^2}{24}\big([t_{23}, t_{21}+ t_{31}] - [t_{13},t_{12}+ t_{32}]\big)
~\xRrightarrow{\,\tfrac{\hbar^2}{24}(\Gamma_{231}-\Gamma_{132})\,}~0
\end{flalign}
\end{subequations}
only holds true weakly up to a certain linear combination of the $\Ch_{\bbK[[\hbar]]}^{[-1,0]}$-enriched modifications
from Proposition \ref{propo:Gamma4term}. This illustrates that the strict identities
from Proposition \ref{propo:allorderdeformation}, and the associated concept of semi-strict
braided monoidal category objects from Definition \ref{def:2TermBraidedCat}, is highly restrictive
even at order $\hbar^2$ and hence they should be relaxed to weaker concepts
by allowing for non-trivial pentagonator and hexagonator modifications.


\section{Examples}
\label{sec:examples}
The aim of this section is to study explicit examples of our 
infinitesimal $2$-braidings from Theorem \ref{theo:mainresult}.
The examples we consider are based on Lie $N$-algebras,
for some natural number $N\geq 1$, which are a special class of 
$L_\infty$-algebras whose underlying cochain complex
\begin{flalign}\label{eqn:LieNcomplex}
L\,=\,\Big(
\xymatrix{
L^{-N+1}\ar[r]^-{\dd_L}~&~ L^{-N+2} \ar[r]^-{\dd_L}~&~\cdots \ar[r]^-{\dd_L}~&~L^{-1} \ar[r]^-{\dd_L}~&~ L^0 
}\Big)\,\in\,\Ch_\bbK
\end{flalign}
is concentrated in non-positive cohomological degrees.
We shall always assume that the vector spaces $L^i$ are finite-dimensional, for all $i$.
The Lie $N$-algebra structure on $L$ is given by a family of 
brackets $\ell_n \in\hom(\wedge^n L,L)^{2-n}$
of degree $2-n$, for all $n\geq 2$, which have to satisfy the
homotopy Jacobi identities, see e.g.\ \cite{Schnitzer} for an excellent review.
By a simple degree counting argument, one observes that $\ell_n=0$ must 
necessarily vanish for all $n> 1+N$, hence the family of 
brackets $\{\ell_n\}_{n\geq 2}$ is bounded for every Lie $N$-algebra.
\sk

Associated to a Lie $N$-algebra $L= (L,\{\ell_n\})$ 
is its Chevalley-Eilenberg CDGA $\CE^\bullet(L)$. The underlying 
$\bbZ$-graded commutative algebra of $\CE^\bullet(L)$ is given by
\begin{flalign}\label{eqn:CE}
\CE^\sharp(L)\,:=\,\Sym\big(L^{\sharp\ast}[-1]\big)
\end{flalign}
and the differential $\dd_{\CE} := \dd_{L^\ast[-1]} + \sum_{n\geq 2} 
\ell_n^{\ast\,[-1]}$ is defined on the generators $L^{\sharp\ast}[-1]$ by dualizing and shifting
the brackets $\ell_n : L^{\sharp\otimes n}\to L^\sharp$.
The Chevalley-Eilenberg CDGA $\CE^\bullet(L)$ of a finite-dimensional 
Lie $N$-algebra is finitely generated and semi-free, hence it satisfies
our hypotheses from Section \ref{sec:DAG}. Note that this CDGA is concentrated in non-negative
cohomological degrees, which allows us to interpret $\CE^\bullet(L)$ 
geometrically as a stacky affine in the sense of \cite[Section 1]{PridhamNotes},
see also \cite{Pridham,PridhamOutline}. Such stacky affines carry only
a non-trivial stacky structure, but a trivial derived structure, which implies
that one can describe them by singly-graded objects, in contrast to the more general
doubly-graded stacky CDGAs from \cite{PridhamNotes,Pridham,PridhamOutline} or 
the graded mixed CDGAs from \cite{CPTVV} which carry both stacky and derived structures.
In the framework of stacky affines, ones interprets $\CE^\bullet(L)$ as the CDGA
of functions on the infinitesimal quotient stack $\mathrm{B}L:=[\mathrm{pt}/L]$, which is an infinitesimal
analogue of the classifying stack $\mathrm{B}G:=[\mathrm{pt}/G]$ of a (higher) group $G$.
\sk

The symmetric monoidal dg-category ${}_{\CE^\bullet(L)}\dgMod^{\mathrm{fgsf}}$
of finitely generated semi-free $\CE^\bullet(L)$-dg-modules
describes representations of the Lie $N$-algebra $L$ and their $\infty$-morphisms.
Our result in Theorem \ref{theo:mainresult} then implies that,
given any $2$-shifted Poisson structure on the CDGA $\CE^\bullet(L)$,
one obtains a first-order deformation of the homotopy $2$-category
$\mathsf{Ho}_2\big({}_{\CE^\bullet(L)}\dgMod^{\mathrm{fgsf}}\big)$ 
of the representation category of $L$ into a braided monoidal category.
Drawing analogies with the standard deformation theory 
of the (strict) representation category $\mathbf{Rep}(\g)$ of an ordinary Lie algebra $\g$ 
and its relationship to the quantum group $U_q\g$, see
e.g.\ \cite{Cartier} and \cite[Chapter XX]{Kassel}, one may interpret 
our deformation constructions as a first step towards a Lie
$N$-algebraic generalization of quantum groups.
\begin{rem}\label{rem:formalmoduli}
From the point of view of formal moduli problems \cite{Lurie} and
Koszul duality in the form reviewed for instance in \cite{Koszul},
it may also seem natural to consider the completed Chevalley-Eilenberg CDGA
$\widehat{\CE}^\bullet(L)$ whose underlying $\bbZ$-graded commutative algebra is given by
\begin{flalign}
\widehat{\CE}^\sharp(L)\,:=\, \widehat{\Sym}\big(L^{\sharp\ast}[-1]\big)\,:=\,\prod_{l\geq 0}
\Sym^l\big(L^{\sharp\ast}[-1]\big)\quad.
\end{flalign}
We would like to note that this completion of the Chevalley-Eilenberg CDGA 
in \eqref{eqn:CE} to formal power series in the generators does 
not alter the structure of our results in this section. The key reason for this is 
that the $2$-shifted Poisson structures on the completed
$\widehat{\CE}^\bullet(L)$ and on the uncompleted $\CE^\bullet(L)$ 
Chevalley-Eilenberg CDGA of a Lie $N$-algebra $L$ are naturally 
identified with each other: The completed $2$-shifted polyvectors on $\widehat{\CE}^\bullet(L)$ 
are given by
\begin{flalign}
\widehat{\Pol}^\sharp\big(\widehat{\CE}^\bullet(L),2\big)\,\cong\, 
\prod_{m,l\geq 0} \Sym^m\big(L^\sharp[-2]\big)\otimes \Sym^l\big(L^{\sharp\ast}[-1]\big)
\end{flalign}
and they differ from $\widehat{\Pol}^\sharp\big(\CE^\bullet(L),2\big)$ in Definition \ref{def:polycompleted}
by completing the direct sum $\bigoplus_{l\geq 0}$ to the direct product $\prod_{l\geq 0}$.
A $2$-shifted Poisson structure on the completed
$\widehat{\CE}^\bullet(L)$ is then given by a formal sum $\pi = \sum_{m\geq 2}\sum_{l\geq 0}\pi^{(m,l)}
\in \widehat{\Pol}\big(\widehat{\CE}^\bullet(L),2\big)^4$, which is also formal 
in $l\geq 0$, satisfying the Maurer-Cartan equation from 
Definition \ref{def:nshiftedpoisson}. By a simple degree counting argument, taking into account
the specific form of the underlying $N$-term complex  \eqref{eqn:LieNcomplex} of $L$, one observes
that the component $\pi^{(m,l)}$ can only be non-trivial if
$(-N+1)\,m \leq 4-2m-l \leq (N-1)\,l$. Hence, for every fixed $m\geq 2$, there are 
only finitely many $l\geq 0$ satisfying these bounds, which implies that 
$\pi = \sum_{m\geq 2}\sum_{l\geq 0}\pi^{(m,l)}
\in \widehat{\Pol}\big(\CE^\bullet(L),2\big)^4$ also defines a $2$-shifted Poisson structure on the uncompleted $\CE^\bullet(L)$.
\end{rem}

\subsection{Ordinary Lie algebras}
\label{subsec:Lie}
Let us start with the simplest
example which is given by the case where the Lie $N$-algebra $L$
is an ordinary (finite-dimensional) Lie algebra $\g=(\g,[\,\cdot\,,\,\cdot\,])$,
i.e.\ $N=1$. This means that 
the underlying $\bbZ$-graded vector space is $\g$, concentrated in degree $0$,
and all brackets but the Lie bracket $\ell_2:=[\,\cdot\,,\,\cdot\,] : \g\otimes \g\to \g$ are trivial.
To simplify our calculations below, we shall choose a basis $\{x_{a}\in\g:a=1,\dots,\dim(\g)\}$
and denote the structure constants by $[x_a,x_b]=f^{c}_{ab}\,x_c$. The Chevalley-Eilenberg
CDGA $\CE^\bullet(\g)$ is then given by the $\bbZ$-graded commutative algebra
\begin{subequations}
\begin{flalign}
\CE^\sharp(\g)\,=\,\bbK[\{\theta^a\}]
\end{flalign}
which is freely generated by $\dim(\g)$-many generators $\theta^a$
of degree $1$, for $a=1,\dots,\dim(\g)$, together with the differential
\begin{flalign}
\dd_{\CE}(\theta^a)\,=\,-\tfrac{1}{2} \,f^{a}_{bc}\,\theta^b\,\theta^c
\end{flalign}
\end{subequations}
which is determined by the structure constants. Here and in what follows we 
shall always use the standard summation convention according to which summations over repeated indices are suppressed.
Note that the square-zero condition $\dd_{\CE}\,\dd_{\CE}(\theta^a)=0$ follows from the Jacobi identity
of the Lie bracket after expressing the latter in terms of the structure constants.
\sk

Since the input datum for Theorem \ref{theo:mainresult} is a 
$2$-shifted Poisson structure on $\CE^\bullet(\g)$ in the sense
of Definition \ref{def:nshiftedpoisson}, we first have to 
understand and characterize the possible choices for such data.
For the $\CE^\bullet(\g)$-dg-module of derivations $\mathsf{T}_{\CE^\bullet(\g)}$
we choose the basis $\{\partial_{a}\in \mathsf{T}_{\CE^\bullet(\g)} : a=1,\dots,\dim(\g)\}$
which is defined by $\partial_a(\theta^b)=\delta_a^b$, for all $a,b$. Note that
each basis element $\partial_a$ has degree $-1$ since all $\theta^b$ are of degree $1$.
The underlying $\bbZ$-graded commutative algebra of the completed $2$-shifted polyvectors from Definition
\ref{def:polycompleted} is then given by
\begin{flalign}
\widehat{\Pol}\big(\CE^\bullet(\g),2\big)^\sharp\,=\,\bbK[\{\theta^a\}][[\{s^3 \partial_a\}]]\quad,
\end{flalign}
where $[[\,\cdot\,]]$ denotes formal power series. Note that the generators 
$s^3\partial_a$ are of degree $3-1=2$ in $\widehat{\Pol}\big(\CE^\bullet(\g),2\big)$.
\begin{propo}\label{propo:2shiftedPoissonLie}
Any $2$-shifted Poisson structure $\pi=\sum_{m\geq 2}\pi^{(m)}\in \widehat{\Pol}\big(\CE^\bullet(\g),2\big)^4$ 
on $\CE^\bullet(\g)$ is of the form
\begin{subequations}\label{eqn:piLie}
\begin{flalign}\label{eqn:piLie1}
\pi \,=\,\pi^{(2)}\,=\,\tfrac{1}{2}\,\pi^{ab}\, s^3\partial_a \,s^3\partial_b\quad, 
\end{flalign}
where $\pi^{ab}\in\bbK$  are coefficients which
must satisfy the $\g$-invariance property
\begin{flalign}\label{eqn:piLie2}
f_{cd}^a \,\pi^{d b} +f_{cd}^b\,\pi^{a d}\,=\,0\quad.
\end{flalign}
\end{subequations}
\end{propo}
\begin{proof}
Using that $\theta^a$ is of degree $1$ and that $s^3\partial_a$ is of degree $2$, 
one finds by a simple degree-counting argument that \eqref{eqn:piLie1} is
the only possible form for a completed $2$-shifted polyvector of degree $4$ and weight $\geq 2$.
The Maurer-Cartan equation from Definition \ref{def:nshiftedpoisson} for such $\pi$,
see also Remark \ref{rem:nshiftedpoisson}, is equivalent to the two individual equations
$\dd_{\widehat{\Pol}}(\pi)=0$ and $\{\pi,\pi\}=0$. The second equation holds true automatically
because \eqref{eqn:piLie1} has constant coefficients, while the first equation
is equivalent to the $\g$-invariance property \eqref{eqn:piLie2}.
To see the latter point explicitly, one uses the Leibniz rule
for the differential $\dd_{\widehat{\Pol}}$ and 
$\dd_{\widehat{\Pol}}(s^3\partial_a) = -f_{ab}^{c}\,\theta^b\,s^3\partial_c$.
\end{proof}

\begin{rem}
Note that the $2$-shifted Poisson structures $\pi$ from 
Proposition \ref{propo:2shiftedPoissonLie} 
are canonically identified with 
invariant symmetric tensors $\pi =\tfrac{1}{2}\,\pi^{ab}\,x_a\,x_b\in (\Sym^2\g)^\g$.
Hence, our results are compatible with the results of Safronov \cite{Safronov}
who has computed the $2$-shifted (and also the $1$-shifted) Poisson
structures on the classifying stack $\mathrm{B}G = [\mathrm{pt}/G]$ of an algebraic group $G$.
\end{rem}

Consider now the homotopy $2$-category
${}_{\CE^\bullet(\g)}\CC= \mathsf{Ho}_2\big({}_{\CE^\bullet(\g)}\dgMod^{\mathrm{fgsf}}\big)$
of the symmetric monoidal dg-category of finitely generated semi-free $\CE^\bullet(\g)$-dg-modules,
which is the object of focus in our deformation construction from Theorem \ref{theo:mainresult}.
Let us first illustrate the sense in which this category describes 
weak representations of the Lie algebra $\g$. By definition, given any object $M\in {}_{\CE^\bullet(\g)}\CC$,
one has a $\bbZ$-graded vector space of generators $W^\sharp$ such that $M^\sharp \cong \CE^\sharp(\g)\otimes W^{\sharp}$.
Choosing any basis $\{w_i\in W^\sharp\}$ for the generators, the differential on $M$
is completely characterized by specifying $\dd_{M}(\oone\otimes w_i)\in M$, for all $i$,
which when expanded in terms of the chosen bases reads as
\begin{flalign}\label{eqn:dMLieexpansion}
\dd_{M}(\oone\otimes w_i)\,=\, \und{M}_i^{~j}\otimes w_j \,=\,\sum_{n=0}^\infty \tfrac{1}{n!}\,M_{a_1\dots a_n i}^{\phantom{a_1\dots a_n i}j}
\,\theta^{a_1}\cdots\theta^{a_n}\otimes w_j\quad,
\end{flalign}
where $M_{a_1\dots a_n i}^{\phantom{a_1\dots a_n i}j}\in\bbK$ are coefficients. 
The square-zero condition $\dd_M\dd_{M}(\oone\otimes w_i)=0$ gives a tower of compatibility conditions
between the coefficients $M_{a_1\dots a_n i}^{\phantom{a_1\dots a_n i}j}$ and the structure constants $f^{a}_{bc}$ 
of the Lie algebra $\g$, which express a concept of Lie algebra representation up to coherent homotopies.
Strict representations of $\g$ on a cochain complex $W$ are recovered in the special case
where $\dd_M(\oone\otimes w_i) = M_{i}^{~j}\otimes w_j + M_{ai}^{~~j}\,\theta^a \otimes w_j$
is of polynomial degree $\leq 1$ in $\theta^a$. The first term in this expression defines
a differential $\dd_W$ on $W^\sharp$ and the second term defines a strict Lie algebra action map 
$\g \otimes W \to W$ which is a cochain map for the differential of $W = (W^\sharp,\dd_W)$.
\sk

In addition to weak representations of the Lie algebra $\g$, the homotopy $2$-category
${}_{\CE^\bullet(\g)}\CC= \mathsf{Ho}_2\big({}_{\CE^\bullet(\g)}\dgMod^{\mathrm{fgsf}}\big)$
also describes a concept of $\infty$-morphisms
between such weak representations, i.e.\ weakly $\g$-equivariant maps.
(The $2$-morphisms in ${}_{\CE^\bullet(\g)}\CC$ are homotopies between such $\infty$-morphisms.)
Indeed, a morphism $h : M\to M^\prime$ in ${}_{\CE^\bullet(\g)}\CC$ is by definition
a cochain map which commutes with the left $\CE^\bullet(\g)$-actions.
Hence, by choosing a basis $\{w_i\in W^\sharp\}$ for $M^\sharp \cong \CE^\sharp(\g)\otimes W^{\sharp}$ 
and a basis $\{w_i^\prime\in W^{\prime\sharp}\}$ for $M^{\prime\sharp} \cong \CE^\sharp(\g)\otimes W^{\prime\sharp}$,
we can expand 
\begin{flalign}\label{eqn:hLieexpansion}
h(\oone\otimes w_i)\,=\,\und{h}_{i}^{~j}\otimes w_j^\prime\,=\, \sum_{n=0}^\infty 
\tfrac{1}{n!}\,h_{a_1\dots a_n i}^{\phantom{a_1\dots a_n i}j}
\,\theta^{a_1}\cdots\theta^{a_n}\otimes w_j^\prime\quad,
\end{flalign}
where $h_{a_1\dots a_n i}^{\phantom{a_1\dots a_n i}j}\in\bbK$ are coefficients.
The cochain map condition $\dd_{M^\prime} \big(h(\oone\otimes w_i)\big) = h\big(\dd_M(\oone\otimes w_i)\big)$ 
gives a tower of compatibility conditions
between the coefficients $h_{a_1\dots a_n i}^{\phantom{a_1\dots a_n i}j}$ 
and the structure constants $f^{a}_{bc}$ of the Lie algebra $\g$, 
which express a concept of $\g$-equivariance up to coherent homotopies.
Strict morphisms $h : M\to M^\prime$ between strict $\g$-representations 
on cochain complexes $W$ and $W^\prime$ are recovered in the special case
where $h(\oone\otimes w_i) = h_{i}^{~j}\otimes w^\prime_j$ is constant in
$\theta^a$. 
\sk

The following result is a direct consequence of Theorem \ref{theo:mainresult}
and Proposition \ref{propo:2shiftedPoissonLie}.
\begin{cor}\label{cor:gLiedeformation}
Every invariant symmetric tensor 
$\pi =\tfrac{1}{2}\,\pi^{ab}\,x_a\,x_b\in (\Sym^2\g)^\g$
defines via Proposition \ref{propo:2shiftedPoissonLie} 
and Theorem \ref{theo:mainresult} a $\gamma$-equivariant infinitesimal 
$2$-braiding $t$ on the homotopy $2$-category ${}_{\CE^\bullet(\g)}\CC$ of the
symmetric monoidal dg-category of finitely generated semi-free $\CE^\bullet(\g)$-dg-modules.
\end{cor}

\begin{rem}\label{rem:gLiedeformation}
Let us recall that explicit formulas for the components of the 
$\gamma$-equivariant infinitesimal $2$-braiding $t$ 
are given in \eqref{eqn:tfromPoissonbasis}. Note that the 
morphism components $t_{M,N}$ in \eqref{eqn:tfromPoissonbasis1}
involve the de Rham differential
\begin{flalign}
(\dd_{\dR}\otimes\id)\big(\dd_{M}(\oone\otimes w_i)\big)\,&=\, 
\sum_{n=1}^\infty \tfrac{1}{(n-1)!}\,M_{a_1\dots a_n i}^{\phantom{a_1\dots a_n i}j}
\,\theta^{a_1}\cdots\theta^{a_{n-1}}\dd_{\dR}(\theta^{a_n})\otimes w_j
\end{flalign}
of \eqref{eqn:dMLieexpansion}, as well as $(\dd_{\dR}\otimes\id)\big(\dd_{N}(\oone\otimes u_q)\big)$.
Hence, the components $t_{M,N}$ in \eqref{eqn:tfromPoissonbasis1} are in general 
non-constant polynomials in $\theta^a$, unless both $M$ and $N$ are strict $\g$-representations.
Similar, the homotopy components $t_{(M,N),(M^\prime,N^\prime)}(h\widetilde{\otimes}k)$ 
in \eqref{eqn:tfromPoissonbasis2} involve the de Rham differential
\begin{flalign}
(\dd_{\dR}\otimes\id)\big(h(\oone\otimes w_i)\big)\,&=\, 
\sum_{n=1}^\infty \tfrac{1}{(n-1)!}\,h_{a_1\dots a_n i}^{\phantom{a_1\dots a_n i}j}
\,\theta^{a_1}\cdots\theta^{a_{n-1}}\dd_{\dR}(\theta^{a_n})\otimes w^\prime_j
\end{flalign}
of \eqref{eqn:hLieexpansion}, as well as
$(\dd_{\dR}\otimes\id)\big(k(\oone\otimes u_q)\big)$.
Hence, the homotopies $t_{(M,N),(M^\prime,N^\prime)}(h\widetilde{\otimes}k)$  
in \eqref{eqn:tfromPoissonbasis2} are in general non-trivial,
unless both $h$ and $k$ are strict morphisms.
\end{rem}

As a direct consequence of Proposition \ref{propo:allorderdeformation}
and Proposition \ref{propo:Gamma4term}, we obtain the following result.
\begin{cor}\label{cor:Lieallorder}
Every invariant symmetric tensor 
$\pi =\tfrac{1}{2}\,\pi^{ab}\,x_a\,x_b\in (\Sym^2\g)^\g$
defines via Corollary \ref{cor:gLiedeformation} and
Proposition \ref{propo:allorderdeformation}
a braided monoidal deformation to all orders in $\hbar$
of the full subcategory of ${}_{\CE^\bullet(\g)}\CC\in \dgCat_{\bbK[[\hbar]]}^{[-1,0],\mathrm{ps}}$ 
which consists of all $\CE^\bullet(\g)$-dg-modules corresponding to strict $\g$-representations.
\end{cor}
\begin{proof}
Combining our observations from
Example \ref{ex:Gammaexplicit} and Remark \ref{rem:gLiedeformation},
we obtain that the modifications $\Gamma_{ijk}$ in Proposition \ref{propo:Gamma4term}
are all trivial for strict $\g$-representations because in this case
$(t_{M,N})_{iq}^{~~jr}$ are constant polynomials in $\theta^a$, hence
$\dd_{\dR}\big((t_{M,N})_{iq}^{~~jr}\big)=0$. This implies that
the hypotheses of Proposition \ref{propo:allorderdeformation} are satisfied,
which yields the desired braided monoidal deformation to all orders in $\hbar$.
\end{proof}

\begin{rem}
Our result in Corollary \ref{cor:Lieallorder} is a homotopical
generalization of the deformation constructions in
\cite{Cartier} and \cite[Chapter XX]{Kassel}. We would like to stress that our generalization 
is non-trivial because, even though we restrict ourselves to strict $\g$-representations,
the morphisms in the \textit{full} subcategory from Corollary \ref{cor:Lieallorder}
are general (i.e.\ non-strict) $\infty$-morphisms between strict $\g$-representations. 
Recalling our observations from Remark \ref{rem:gLiedeformation}, this implies in particular
that the infinitesimal braiding $t$ is a genuine pseudo-natural transformation,
i.e.\ it has non-trivial homotopy components $t_{(M,N),(M^\prime,N^\prime)}(h\widetilde{\otimes}k)$,
and so are the associator and braiding of the deformed category from
Proposition \ref{propo:allorderdeformation} and Corollary \ref{cor:Lieallorder}.
\end{rem}

\subsection{String Lie \texorpdfstring{$2$}{2}-algebras}
\label{subsec:Lie2}
String Lie $2$-algebras are simple examples of Lie $2$-algebras
which arise as shifted central extensions 
determined by an ordinary (finite-dimensional) Lie algebra $\g=(\g,[\,\cdot\,,\,\cdot\,])$, concentrated
in degree $0$, and a $3$-cocycle $\kappa:\bigwedge^3\g\to\bbK$ on $\g$.
Such Lie $2$-algebras appeared first in \cite[Example 6.10]{BaezCrans}, 
however the name ``string Lie $2$-algebras'' has only been established later.
The resulting Lie $2$-algebra $\g_\kappa$
is given by the $\bbZ$-graded vector space defined by $\g_\kappa^0:=\g$,
$\g_\kappa^{-1}:=\bbK$ and $\g_\kappa^i:=0$, for all $i\in\bbZ\setminus \{-1,0\}$,
and the non-trivial brackets read as
\begin{flalign}\label{eqn:stringbrackets}
\ell_2(x,y) \,:=\,[x,y]\quad,\qquad \ell_3(x,y,z)\,:=\, \kappa(x,y,z)\quad,
\end{flalign}
for all $x,y,z\in\g$. Let us choose again a basis $\{x_a\in\g\}$ and denote the structure constants
by $[x_a,x_b]=f^c_{ab}\,x_c$ and $\kappa(x_a,x_b,x_c) = \kappa_{abc}$.
\sk

The Chevalley-Eilenberg CDGA $\CE^\bullet(\g_\kappa)$
of the string Lie $2$-algebra $\g_\kappa$
is then given by the $\bbZ$-graded commutative algebra
\begin{subequations}
\begin{flalign}
\CE^\sharp(\g_\kappa)\,=\,\bbK[\{\theta^a\},\nu]
\end{flalign}
which is freely generated by $\dim(\g)$-many generators $\theta^a$
of degree $1$, for $a=1,\dots,\dim(\g)$, and a single generator $\nu$
of degree $2$, together with the differential
\begin{flalign}
\dd_{\CE}(\theta^a)\,=\,-\tfrac{1}{2} \,f^{a}_{bc}\,\theta^b\,\theta^c\quad,\qquad
\dd_{\CE}(\nu)\,=\,-\tfrac{1}{3!}\,\kappa_{abc}\,\theta^a\,\theta^b\,\theta^c
\end{flalign}
\end{subequations}
which is determined by the structure constants. 
Note that the square-zero condition $\dd_{\CE}\,\dd_{\CE}(\theta^a)=0$ follows from the Jacobi identity
of the Lie bracket and $\dd_{\CE}\,\dd_{\CE}(\nu)=0$ 
follows from the $3$-cocycle condition of $\kappa$.
\sk

Let us now characterize the
$2$-shifted Poisson structures on $\CE^\bullet(\g_\kappa)$ in the sense
of Definition \ref{def:nshiftedpoisson}.
For the $\CE^\bullet(\g_\kappa)$-dg-module of derivations $\mathsf{T}_{\CE^\bullet(\g_\kappa)}$
we choose the basis $\{\{\partial_{a}\},\partial_{\nu}\}$
which is defined by $\partial_a(\theta^b)=\delta_a^b$, $\partial_a(\nu)=0$,
$\partial_\nu (\theta^a)=0$ and $\partial_\nu(\nu)=1$, for all $a,b$. Note that
each basis element $\partial_a$ has degree $-1$ and $\partial_\nu$ has degree $-2$.
The underlying $\bbZ$-graded commutative algebra of the completed $2$-shifted polyvectors from Definition
\ref{def:polycompleted} is then given by
\begin{flalign}
\widehat{\Pol}\big(\CE^\bullet(\g_\kappa),2\big)^\sharp\,=\,\bbK[\{\theta^a\},\nu][[\{s^3 \partial_a\},s^{3}\partial_\nu]]\quad.
\end{flalign}
where $[[\,\cdot\,]]$ denotes formal power series. Note that the generators 
$s^3\partial_a$ are of degree $3-1=2$ 
and that $s^3\partial_\nu$ is of degree $3-2=1$ 
in $\widehat{\Pol}\big(\CE^\bullet(\g_\kappa),2\big)$.
\begin{propo}\label{propo:2shiftedPoissonString}
Any $2$-shifted Poisson structure 
$\pi=\sum_{m\geq 2}\pi^{(m)}\in \widehat{\Pol}\big(\CE^\bullet(\g_\kappa),2\big)^4$ 
on $\CE^\bullet(\g_\kappa)$ is of the form
\begin{subequations}\label{eqn:piString}
\begin{flalign}\label{eqn:piString1}
\pi \,=\,\pi^{(2)}\,=\, \tfrac{1}{2}\,\pi^{ab}\, s^3\partial_a \,s^3\partial_b + \tilde{\pi}_b^a\,\theta^b\,s^3\partial_a\,s^3\partial_{\nu}\quad, 
\end{flalign}
where $\pi^{ab}\in\bbK$ and $\tilde{\pi}_b^a\in\bbK$ are coefficients
which must satisfy the properties
\begin{flalign}
f_{cd}^a\,\pi^{db} + f_{cd}^b\,\pi^{ad}\,&=\,0\quad, \label{eqn:piString2}\\
f_{bc}^d\,\tilde{\pi}^a_d - f_{bd}^a\,\tilde{\pi}^d_c + f_{cd}^a\,\tilde{\pi}^d_b -\kappa_{bcd}\,\pi^{da} \,&=\,0\quad,
\label{eqn:piString3}\\
\pi^{ac}\,\tilde{\pi}_c^b + \pi^{bc}\,\tilde{\pi}_c^a\,&\,=0\quad. \label{eqn:piString4}
\end{flalign}
\end{subequations}
\end{propo}
\begin{proof}
The general form \eqref{eqn:piString1} for a 
completed $2$-shifted polyvector of degree $4$ and weight $\geq 2$
follows from a simple degree counting argument. (Note that 
products of the form $s^3\partial_{\nu}\,s^3\partial_{\nu}$
vanish in the CDGA $\widehat{\Pol}\big(\CE^\bullet(\g_\kappa),2\big)$ 
because $s^3\partial_{\nu}$ has degree $1$, hence $s^3\partial_{\nu}\,s^3\partial_{\nu} = 
-s^3\partial_{\nu}\,s^3\partial_{\nu} =0$.)
The Maurer-Cartan equation from Definition \ref{def:nshiftedpoisson} for such $\pi$,
see also Remark \ref{rem:nshiftedpoisson}, is equivalent to the two individual equations
$\dd_{\widehat{\Pol}}(\pi)=0$ and $\{\pi,\pi\}=0$. 
The first equation $\dd_{\widehat{\Pol}}(\pi)=0$ can be evaluated using
the Leibniz rule for the differential $\dd_{\widehat{\Pol}}$ and
\begin{flalign}
\dd_{\widehat{\Pol}}(s^3\partial_a)\,=\,-f_{ab}^c\,\theta^b\,s^3\partial_c +
\tfrac{1}{2}\,\kappa_{abc}\,\theta^b\,\theta^c\,s^3\partial_\nu\quad,\qquad 
\dd_{\widehat{\Pol}}(s^3\partial_\nu)\,=\,0\quad.
\end{flalign}
From this one finds the first two properties \eqref{eqn:piString2} and \eqref{eqn:piString3}.
The second equation $\{\pi,\pi\}=0$ can be evaluated
using the definition \eqref{eqn:Schoutendef} and properties \eqref{eqn:Schoutenproperties}
of the Schouten-Nijenhuis bracket, which leads to the third property \eqref{eqn:piString4}.
\end{proof}

\begin{rem}\label{rem:2shiftedPoissonString}
As a consequence of \eqref{eqn:piString2}, the first term in \eqref{eqn:piString1} is equivalent to the
datum of an invariant symmetric tensor $\pi =\tfrac{1}{2}\,\pi^{ab}\,x_a\,x_b\in(\Sym^2\g)^\g$
on the underlying Lie algebra $\g$ of the string Lie $2$-algebra $\g_\kappa$.
The second term in \eqref{eqn:piString1} is equivalent to a linear map 
$\tilde{\pi}:\g\to\g\,,~x_a\mapsto \tilde{\pi}_a^b\,x_b$
on the underlying Lie algebra. Property 
\eqref{eqn:piString4} then demands that the linear map $\tilde{\pi}$ annihilates
the invariant symmetric tensor $\pi$ according to
\begin{subequations}\label{eqn:piStringmaps}
\begin{flalign}\label{eqn:piStringmaps1}
(\tilde{\pi}\otimes \id_\g)(\pi) + (\id_\g\otimes\tilde{\pi})(\pi)\,=\,0\quad.
\end{flalign}
The remaining property \eqref{eqn:piString3} can be written equivalently as
\begin{flalign}\label{eqn:piStringmaps2} 
\tilde{\pi}\big([x,y]\big)\,=\, \big[\tilde{\pi}(x),y\big] + \big[x,\tilde{\pi}(y)\big] + 
\big(\kappa(x,y,-)\otimes\id_\g\big)(\pi) + \big(\id_\g\otimes \kappa(x,y,-)\big)(\pi)  \quad,
\end{flalign}
\end{subequations}
for all $x,y\in\g$, which means that the linear map $\tilde{\pi}:\g\to\g$ 
is a derivation of the underlying Lie algebra $\g$, up to corrections arising from
the $3$-cocycle $\kappa:\bigwedge \g^3\to\bbK$ and the invariant symmetric tensor $\pi\in(\Sym^2\g)^\g$.
\sk

Our $2$-shifted Poisson structures on $\CE^\bullet(\g_\kappa)$
are similar, but as it seems not identical, to the quasi-invariant symmetric
tensors on crossed modules of Lie algebras from \cite[Definition 30]{Joao}.
A direct comparison between these two structures is difficult because 
our preferred model for the string Lie $2$-algebra $\g_\kappa$ is given by
a finite-dimensional Lie $2$-algebra with a non-trivial $\ell_3$ bracket \eqref{eqn:stringbrackets}, 
while \cite{Joao} use a different model which is given by an
infinite-dimensional crossed module of Lie algebras 
that encodes the cocycle $\kappa$ less directly.
Our suspicion that these two concepts are slightly different
will be strengthened by the following example.
\end{rem}

\begin{ex}\label{ex:sl2}
In order to illustrate Proposition \ref{propo:2shiftedPoissonString},
let us consider the example where $\g = \mathfrak{sl}_2$ 
is the special linear Lie algebra of order $2$. We work with
the usual basis $\{x_+ ,x_- ,x_3\in\mathfrak{sl}_2\}$ 
whose Lie bracket relations are
\begin{flalign}
[x_+,x_-] \,=\,x_3\quad,\qquad [x_3,x_\pm]\,=\,\pm 2\,x_\pm\quad.
\end{flalign}
Using the rescaled Killing form 
$(\,\cdot\,,\,\cdot\,):\mathfrak{sl}_2\otimes\mathfrak{sl}_2\to\bbK$
whose non-vanishing values are given by $(x_3,x_3) =1$ and $(x_+,x_-)=(x_-,x_+)=\tfrac{1}{2}$,
we define the $3$-cocycle $\kappa(x,y,z):= (x,[y,z])$. In terms of our basis,
this $3$-cocycle is given by
\begin{flalign}
\kappa(x_+,x_-,x_3)\,=\,1
\end{flalign}
and total antisymmetry in its arguments.
This defines the string Lie $2$-algebra $(\mathfrak{sl}_2)_\kappa$.
\sk

Since $\mathfrak{sl}_2$ is a simple Lie algebra, there exists a unique (up to scale)
invariant symmetric tensor, which is concretely given by
\begin{flalign}\label{eqn:sl2invtensor}
\pi \,=\, \lambda\,\Big( x_+ \, x_- +\tfrac{1}{4}\,x_3^2\Big) \,\in\, (\Sym^2 \mathfrak{sl}_2)^{\mathfrak{sl}_2}\quad,
\end{flalign}
where $\lambda\in\bbK$ is any parameter. 
Assuming for the moment that $\lambda\neq 0$ is non-zero, the general solution
for the linear map $\tilde{\pi} : \mathfrak{sl}_2\to\mathfrak{sl}_2$
to the equation \eqref{eqn:piString4}, or 
the equivalent equation \eqref{eqn:piStringmaps1}, is given by
\begin{flalign}
\tilde{\pi}(x_+)\,=\,a\,x_+ + b\,x_3~~,\quad
\tilde{\pi}(x_-)\,=\,-a\,x_- + c\,x_3~~,\quad
\tilde{\pi}(x_3)\,=\,-2\,c\,x_+ -2\,b \,x_-\quad,
\end{flalign}
where $a,b,c\in\bbK$ are arbitrary parameters. Evaluating the equivalent form \eqref{eqn:piStringmaps2}
of the remaining equation \eqref{eqn:piString3} on $x=x_+$ and $y=x_-$, one finds 
that it does not admit any solutions for $\lambda\neq 0$. This implies that, 
for every $2$-shifted Poisson structure on $\CE^\bullet((\mathfrak{sl}_2)_\kappa)$ as in
Proposition \ref{propo:2shiftedPoissonString}, the 
invariant symmetric tensor \eqref{eqn:sl2invtensor} must necessarily be trivial $\pi=0$.
In this case the equations in \eqref{eqn:piString},
see also \eqref{eqn:piStringmaps} for their equivalent forms, simplify to the single condition
that $\tilde{\pi}:\mathfrak{sl}_2\to\mathfrak{sl}_2$ is a Lie algebra derivation,
i.e.\ $\tilde{\pi}([x,y]) = [\tilde{\pi}(x),y] + [x,\tilde{\pi}(y)]$,
for all $x,y\in\mathfrak{sl}_2$. Hence, the $2$-shifted Poisson structures on 
$\CE^\bullet((\mathfrak{sl}_2)_\kappa)$ are given by the Lie algebra 
derivations $\tilde{\pi}$ on $\mathfrak{sl}_2$. Taking into account
the natural concept of homotopies between shifted Poisson structures, 
see e.g.\ \cite{Safronov}, it follows that every $2$-shifted
Poisson structure on $\CE^\bullet((\mathfrak{sl}_2)_\kappa)$ 
is homotopic to the trivial one because every 
Lie algebra derivation of $\mathfrak{sl}_2$ is inner. We would like to thank
the anonymous reviewer for pointing out to us this important observation.
\sk

This behavior is drastically different to the case of ordinary Lie algebras $\g$ 
from Proposition \ref{propo:2shiftedPoissonLie} and the
quasi-invariant symmetric tensors on crossed modules of Lie algebras
from \cite[Theorem 36]{Joao} and \cite[Section 5.2]{Wagemann}.
\end{ex}

\begin{ex}\label{ex:Heisenberg}
The triviality result for $2$-shifted Poisson structures 
from Example \ref{ex:sl2} generalizes easily to the case
of the string Lie $2$-algebra of \textit{any} simple Lie algebra $\g$.
This motivates us to search for non-trivial $2$-shifted Poisson structures 
in the case where the underlying Lie algebra $\g$ is not simple.
Let us consider as a concrete example the Heisenberg Lie algebra, which is the $3$-dimensional
Lie algebra $\mathfrak{heis}=\bbK[x,y,z]$ with
the only non-trivial Lie bracket relations among the generators
given by $[x,y]=-[y,x] =z$. Choosing the $3$-cocycle
$\kappa(x,y,z) =1$, we obtain the associated  Heisenberg string Lie $2$-algebra $\mathfrak{heis}_\kappa$.
\sk

The most general invariant symmetric tensor on the Heisenberg Lie algebra is given by
\begin{subequations}\label{eqn:heisenberg}
\begin{flalign}
\pi\,=\, \tfrac{\lambda}{2}\,z^2\,\in\,(\Sym^2 \mathfrak{heis})^{\mathfrak{heis}}\quad,
\end{flalign}
where $\lambda\in \bbK$ is any parameter which we choose to be non-zero $\lambda\neq 0$.
The linear map $\tilde{\pi} : \mathfrak{heis}\to\mathfrak{heis}$ entering
the $2$-shifted Poisson structure then satisfies the required equation 
\eqref{eqn:piStringmaps1} if and only if $\tilde{\pi}(z) = 0$. Analyzing
the remaining equation \eqref{eqn:piStringmaps2}, one finds that $\tilde{\pi}$
must be of the form
\begin{flalign}
\tilde{\pi}(x) \,=\,a\,x+b\,y +c\,z~~,\quad
 \tilde{\pi}(y) \,=\,d\,x- (a+\lambda)\,y + e\,z~~,\quad
 \tilde{\pi}(z)\,=\,0\quad,
\end{flalign}
\end{subequations}
where $a,b,c,d,e\in\bbK$ are free parameters. 
This provides a multiparameter family of $2$-shifted
Poisson structures on the Chevalley-Eilenberg CDGA $\CE^\bullet(\mathfrak{heis}_\kappa)$ 
of the Heisenberg string Lie $2$-algebra $\mathfrak{heis}_\kappa$.
\end{ex}

The following result is a direct consequence of Theorem \ref{theo:mainresult},
Proposition \ref{propo:2shiftedPoissonString} and Remark \ref{rem:2shiftedPoissonString}.
\begin{cor}\label{cor:gStringdeformation}
Every pair $(\pi,\tilde{\pi})$ consisting of a 
invariant symmetric tensor $\pi =\tfrac{1}{2}\,\pi^{ab}\,x_a\,x_b\in (\Sym^2\g)^\g$
and a linear map $\tilde{\pi} : \g\to\g\,,~x_a\mapsto \tilde{\pi}_a^b\,x_b$
which satisfies the properties in \eqref{eqn:piStringmaps}
defines via Proposition \ref{propo:2shiftedPoissonString}
and Theorem \ref{theo:mainresult} a $\gamma$-equivariant infinitesimal 
$2$-braiding $t$ on the homotopy $2$-category ${}_{\CE^\bullet(\g_\kappa)}\CC$ of the
symmetric monoidal dg-category of finitely generated semi-free $\CE^\bullet(\g_\kappa)$-dg-modules.
\end{cor}

\begin{rem}
Recalling the explicit formulas for the components of the 
$\gamma$-equivariant infinitesimal $2$-braiding $t$ 
from \eqref{eqn:tfromPoissonbasis}, one immediately observes 
by using also \eqref{eqn:piString1} that the morphism components $t_{M,N}$ 
are in general non-constant polynomials in $\theta^a$ 
in the case where $\tilde{\pi}^a_b\neq 0$ is non-trivial. 
(As illustrated in Example \ref{ex:Heisenberg}, this seems to be the most relevant case.)
Hence, the modifications $\Gamma_{ijk}$ from Proposition \ref{propo:Gamma4term} and Example
\ref{ex:Gammaexplicit} are in general non-trivial. This even holds true for
representation of the string Lie $2$-algebra $\g_\kappa$ which are strict
in the sense that the expansions as in \eqref{eqn:dMLieexpansion} 
of the differential $\dd_M$ are polynomials in $\theta^a$ and $\nu$ of degree $\leq 1$.
As a consequence, we currently do not have a meaningful generalization 
of the all-order deformation result from Corollary \ref{cor:Lieallorder}
to the case of $\CE^\bullet(\g_\kappa)$. Hence, further work as outlined in
Remark \ref{rem:outlookcoherence} is needed in order to extend
our result in Corollary \ref{cor:gStringdeformation} to a braided monoidal
deformation up to all orders in $\hbar$.
\end{rem}


\section*{Acknowledgments}
We would like to thank Jon Pridham and the anonymous reviewer for useful comments.
C.K.\ is supported by an EPSRC doctoral studentship. 
A.S.\ gratefully acknowledges the support of 
the Royal Society (UK) through a Royal Society University 
Research Fellowship (URF\textbackslash R\textbackslash 211015)
and Enhancement Grants (RF\textbackslash ERE\textbackslash 210053 and 
RF\textbackslash ERE\textbackslash 231077).


\appendix

\section{Higher-categorical framework}
\label{app:categorical}
In this appendix, we introduce a suitable 
higher-categorical framework in which the deformation constructions
from Section \ref{sec:2braidings} take place. 
The basic objects of our interest are categories 
which are enriched over the symmetric monoidal category $\Ch_R^{[-1,0]}$ 
of $2$-term cochain complexes (concentrated in degrees $-1$ and $0$) 
of modules over a commutative $\bbK$-algebra $R$. 
We will show below that these objects assemble
into a suitably weakened kind of $3$-category $\dgCat_R^{[-1,0],\mathrm{ps}}$ whose $1$-morphisms
are $\Ch_R^{[-1,0]}$-enriched functors, $2$-morphisms are 
$\Ch_R^{[-1,0]}$-enriched \textit{pseudo}-natural transformations and 
$3$-morphisms are $\Ch_R^{[-1,0]}$-enriched modifications. 
The crucial point here is that we allow for pseudo-natural transformations,
in contrast to only strictly natural ones, as these are required in our study
of infinitesimal $2$-braidings in Section \ref{sec:2braidings}. 
This $3$-category carries a symmetric monoidal
structure, which allows one to define a concept of $\Ch_R^{[-1,0]}$-enriched
braided monoidal categories whose associators, unitors and braidings are 
enriched \textit{pseudo}-natural transformations, as required in Section \ref{sec:2braidings}.
We assume that the reader has some familiarity with 
enriched category theory, see e.g.\ \cite{Kelly} for a comprehensive treatment, 
or \cite[Chapter 1.3]{Yau} and \cite[Chapter 3]{Riehl} for more concise introductions.

\subsection{\texorpdfstring{$3$}{3}-categories with pseudo-functorial composition}
\label{app:3Cat}
Let us recall that a strict $3$-category can be defined as a category which is
enriched over the category $\2Cat^{\mathrm{str}}$ of strict $2$-categories and strict $2$-functors.
We require a slightly weakened variant of this concept where the enrichment
category is generalized to the category $\2Cat^{\mathrm{ps}}$ of strict $2$-categories
and pseudo-functors. See e.g.\ \cite[Theorem 4.1.30]{Yau} for a detailed
description of this category and, in particular, of its strictly associative and unital composition.
The category $\2Cat^{\mathrm{ps}}$ admits finite products and a terminal object,
given by the $2$-category $\mathbf{1}\in \2Cat^{\mathrm{ps}}$ consisting of a single $0$-cell
and its identity $1$-cell and $2$-cell, hence it is symmetric monoidal with respect to the 
Cartesian symmetric monoidal structure.
\begin{defi}\label{def:3Cat}
We denote by 
\begin{flalign}
\3Cat\,:=\,\2Cat^{\mathrm{ps}}\mhyphen\Cat
\end{flalign}
the $2$-category whose
\begin{itemize}
\item objects are $\2Cat^{\mathrm{ps}}$-enriched categories,
\item $1$-morphisms are $\2Cat^{\mathrm{ps}}$-enriched functors, and
\item $2$-morphisms are $\2Cat^{\mathrm{ps}}$-enriched natural transformations.
\end{itemize}
\end{defi}

To illustrate how an object $\C\in \3Cat$ admits an interpretation
in terms of a weakened kind of $3$-category, let us spell out, in some detail, 
the data defining $\C$. By definition of an enriched category, $\C$ has a class 
of objects and, for each pair of objects $c,c^\prime\in\C$, a strict $2$-category
$\C(c,c^\prime)\in\2Cat^{\mathrm{ps}}$ of morphisms. Since the $2$-category
$\C(c,c^\prime)$ consists of $0$-cells, $1$-cells and $2$-cells, there are three layers
of morphisms in $\C\in \3Cat$, as required for a $3$-category. The enriched category $\C$
comes further with identity morphisms $\id_c : \mathbf{1}\to \C(c,c)$, or equivalently
$0$-cells $\id_c\in \C(c,c)$, and composition morphisms
\begin{flalign}
\ast \,:\, \C(c^\prime,c^{\prime\prime})\times \C(c,c^{\prime}) ~\longrightarrow~\C(c,c^{\prime\prime})
\end{flalign}
in $\2Cat^{\mathrm{ps}}$, which in our case are pseudo-functors
between strict $2$-categories. The compositions $\ast$ must be strictly associative 
and strictly unital with respect to the identities $\id_c$.
Note that the relaxation to pseudo-functors makes our variant of $3$-categories 
from Definition \ref{def:3Cat} more general than strict $3$-categories, 
for which the analogous composition morphisms $\ast$ must be strict $2$-functors.
\sk

The $2$-category $\3Cat$ carries a symmetric monoidal structure
which is determined by the Cartesian symmetric monoidal structure of the enrichment category $\2Cat^{\mathrm{ps}}$,
see e.g.\ \cite[Section 1.4]{Kelly} for the details. The monoidal product $\C\times \D\in\3Cat$
of $\C,\D\in\3Cat$ is the $\2Cat^{\mathrm{ps}}$-enriched category
whose objects are pairs $(c,d)\in\C\times \D$ of objects $c\in\C$ and $d\in\D$
and whose morphisms $(\C\times \D)((c,d),(c^\prime,d^\prime)):=\C(c,c^\prime)\times \D(d,d^\prime)
\in\2Cat^{\mathrm{ps}}$ are given by the Cartesian symmetric monoidal structure on $\2Cat^{\mathrm{ps}}$.
The monoidal unit $\mathcal{I}\in \3Cat$ is the $\2Cat^{\mathrm{ps}}$-enriched category
with a single object, say $\star\in \mathcal{I}$, and morphisms $\mathcal{I}(\star,\star):=
\mathbf{1}\in\2Cat^{\mathrm{ps}}$ the terminal object. The symmetric braiding
\begin{subequations}
\begin{flalign}
\mathrm{flip}_{\C,\D}^{}\,:\,\C\times \D ~\longrightarrow~\D\times\C
\end{flalign}
is the $\2Cat^{\mathrm{ps}}$-enriched functor which acts on objects by flipping 
$(c,d)\mapsto (d,c)$ and on morphisms through the symmetric braiding on $\2Cat^{\mathrm{ps}}$
\begin{flalign}
(\C\times \D)\big((c,d),(c^\prime,d^\prime)\big)\,=\, \C(c,c^\prime)\times \D(d,d^\prime)~
\stackrel{\cong}{\longrightarrow}~ \D(d,d^\prime)\times \C(c,c^\prime)\,=\, (\D\times\C)\big((d,c),(d^\prime,c^\prime)\big)\quad.
\end{flalign}
\end{subequations}
Making use of this symmetric monoidal structure on $\3Cat$,
one can define a concept of symmetric monoidal category objects in $\3Cat$.
\begin{defi}\label{def:SM3Cat}
A \textit{symmetric monoidal category object in $\3Cat$} is a tuple 
$(\C,\boxtimes,u,\alpha,\lambda,\rho,\tau)$
consisting of the following data:
\begin{itemize}
\item[(i)] an object $\C\in \3Cat$,
\item[(ii)] a $1$-morphism $\boxtimes : \C\times\C\to \C$ in $\3Cat$ (called monoidal product),
\item[(iii)] a $1$-morphism $u : \mathcal{I}\to \C$ in $\3Cat$ (called monoidal unit),
\item[(iv)] invertible $2$-morphisms
\begin{equation}
\begin{tikzcd}
	{\mathcal{C}\times\mathcal{C}\times \mathcal{C}} && {\mathcal{C}\times \mathcal{C}} \\
	\\
	{\mathcal{C}\times\mathcal{C}} && {\mathcal{C}}
	\arrow["{\boxtimes\times\mathrm{id}_\mathcal{C}}", from=1-1, to=1-3]
	\arrow["{\mathrm{id}_{\mathcal{C}}\times\boxtimes}"', from=1-1, to=3-1]
	\arrow["\alpha"', shorten <=12pt, shorten >=12pt, Rightarrow, from=1-3, to=3-1]
	\arrow["\boxtimes", from=1-3, to=3-3]
	\arrow["\boxtimes"', from=3-1, to=3-3]
\end{tikzcd}\qquad\begin{tikzcd}
	& {\mathcal{C}} \\
	{\mathcal{I}\times\mathcal{C}} && {\mathcal{C}\times\mathcal{I}} \\
	& {\mathcal{C}\times\mathcal{C}}
	\arrow["\cong"', from=1-2, to=2-1]
	\arrow["\cong", from=1-2, to=2-3]
	\arrow["{u\times \mathrm{id}_\mathcal{C}}"', from=2-1, to=3-2]
	\arrow["{\mathrm{id}_\mathcal{C}\times u}", from=2-3, to=3-2]
	\arrow[""{name=0, anchor=center, inner sep=0}, "\boxtimes"{description}, from=3-2, to=1-2]
	\arrow["\lambda", shorten <=5pt, shorten >=9pt, Rightarrow, from=2-1, to=0]
	\arrow["\rho"', shorten <=5pt, shorten >=9pt, Rightarrow, from=2-3, to=0]
\end{tikzcd}
\end{equation}
in $\3Cat$ (called associator, left unitor and right unitor),

\item[(v)] an invertible $2$-morphism 
\begin{equation}
\begin{tikzcd}
	{\mathcal{C}\times\mathcal{C}} && {\mathcal{C}} \\
	& {\mathcal{C}\times\mathcal{C}}
	\arrow[""{name=0, anchor=center, inner sep=0}, "\boxtimes", from=1-1, to=1-3]
	\arrow["{\mathrm{flip}_{\mathcal{C},\mathcal{C}}^{}}"', from=1-1, to=2-2]
	\arrow["\boxtimes"', from=2-2, to=1-3]
	\arrow["\tau", shorten <=6pt, shorten >=3pt, Rightarrow, from=0, to=2-2]
\end{tikzcd}
\end{equation}
in $\3Cat$ (called symmetric braiding).
\end{itemize}
These data have to satisfy the usual coherence conditions for symmetric monoidal categories
internal to the $2$-category $\3Cat$.
\end{defi}

\subsection{A symmetric monoidal \texorpdfstring{$3$}{3}-category of \texorpdfstring{$2$}{2}-term truncated dg-categories}
\label{app:ChCat}
Let $R$ be a commutative, associative and unital $\bbK$-algebra and consider
the symmetric monoidal category $\Ch_R$ of cochain complexes of $R$-modules
from  Subsection \ref{subsec:cochain}.
We denote by
\begin{flalign}
\Ch_R^{[-1,0]}\,\subseteq\, \Ch_R
\end{flalign}
the full subcategory whose objects are all cochain complexes $V\in \Ch_R$
concentrated in degrees $\{-1,0\}$, i.e.\ $V^i=0$ for all $i\in \bbZ\setminus\{-1,0\}$.
We shall often denote such objects by
\begin{flalign}
V \,=\, \Big(\xymatrix@C=2em{
V^{-1} \,\ar[r]^-{\dd_V} \,&\, V^0
}\Big)
\,\in\, \Ch_R^{[-1,0]} \quad.
\end{flalign}
Note that the tensor product $\otimes_R$ on $\Ch_R$ does not directly restrict
to the full subcategory $\Ch_R^{[-1,0]}\subseteq \Ch_R$ because
$V\otimes_R W\in \Ch_R$ is, in general, concentrated in degrees $\{-2,-1,0\}$,
for $V,W\in \Ch_R^{[-1,0]}$. This issue can be resolved by composing $\otimes_R$ with
the good truncation functor $\tau^{[-1,0]} : \Ch_R \to \Ch_{R}^{[-1,0]}$ to define
the truncated tensor product
\begin{subequations}\label{eqn:truncatedtensor}
\begin{flalign}
\widetilde{\otimes}_R\,:\,
\xymatrix@C=3em{
\Ch_R^{[-1,0]}\times \Ch_R^{[-1,0]}\ar[r]^-{\subseteq} ~&~ \Ch_R\times\Ch_R\ar[r]^-{\otimes_R} ~&~
\Ch_R \ar[r]^-{\tau^{[-1,0]}} ~&~ \Ch_R^{[-1,0]}
}\quad.
\end{flalign}
Explicitly, given two objects $V,W\in \Ch_R^{[-1,0]}$, the truncated tensor product reads as
\begin{flalign}
V\widetilde{\otimes}_R W\,=\, \bigg(
\xymatrix@C=4em{
\frac{\big(V^{-1}\otimes_R W^0\big) \oplus \big(V^0\otimes_R W^{-1}\big)}{\dd_{V\otimes_R W}\big(V^{-1}\otimes_R W^{-1}\big)} \ar[r]^-{\dd_{V\otimes_R W}} ~&~
V^0\otimes_R W^0
}
\bigg)\quad.
\end{flalign}
\end{subequations}
One directly checks that $\Ch_R^{[-1,0]}$ forms a symmetric monoidal category
with respect to the truncated tensor product $\widetilde{\otimes}_R$, the monoidal unit $R\in \Ch_R^{[-1,0]}$
and the symmetric braiding
\begin{flalign}
\widetilde{\gamma}_{V,W}^{}\,:\, V\widetilde{\otimes}_R W~\longrightarrow~W\widetilde{\otimes}_R V~,~~
v\otimes_R w~\longmapsto~ w\otimes_R v\quad,
\end{flalign}
where all Koszul signs are trivial
because there are no odd-odd degree combinations in the truncated tensor product \eqref{eqn:truncatedtensor}.
For later use, let us observe that the symmetric monoidal categories
$(\Ch_R,\otimes_R,R,\gamma)$ and $(\Ch_R^{[-1,0]},\widetilde{\otimes}_R,R,\widetilde{\gamma})$
are related by a lax symmetric monoidal functor.
\begin{lem}\label{lem:laxtruncation}
The truncation functor $\tau^{[-1,0]} : \Ch_R \to \Ch_{R}^{[-1,0]}$ is canonically lax symmetric monoidal.
\end{lem}
\begin{proof}
The functor $\tau^{[-1,0]}$ preserves the monoidal units strictly, i.e.\ $\tau^{[-1,0]}(R)=R$,
and the lax structure $\tau^{[-1,0]}(V)\,\widetilde{\otimes}_R\,\tau^{[-1,0]}(W)\to 
\tau^{[-1,0]}\big(V\otimes_R W\big)$ for the tensor product of $V,W\in\Ch_R$ 
is given by the obvious cochain map (drawn vertically)
\begin{flalign}
\begin{gathered}
\xymatrix@C=4em@R=2em{
\ar[d] \frac{\big(\tfrac{V^{-1}}{\dd_V V^{-2}}\otimes_R \mathsf{Z}^0(W)\big) \oplus \big(\mathsf{Z}^0(V)\otimes_R \tfrac{W^{-1}}{\dd_W W^{-2}}\big)}{\dd_{V\otimes_R W} \big(V^{-1}\otimes_R W^{-1}\big)} \ar[r]^-{\dd_{V\otimes_R W}} ~&~ \mathsf{Z}^0(V)\otimes_R \mathsf{Z}^0(W)
\ar[d]\\
\frac{(V\otimes_R W)^{-1}}{\dd_{V\otimes_R W}\big(V\otimes_R W\big)^{-2}}\ar[r]_-{\dd_{V\otimes_R W}}~&~\mathsf{Z}^0(V\otimes_R W)
}
\end{gathered}\quad,
\end{flalign}
where $\mathsf{Z}^0(U) := \ker(\dd_U : U^0 \to U^1)$ denotes the submodule of $0$-cocycles
associated with a cochain complex $U\in \Ch_R$.
\end{proof}

The goal of the remaining part of this subsection is to provide details
for the following result.
\begin{propo}\label{prop:Chtrunccat}
There exists a $3$-category
\begin{flalign}
\dgCat_R^{[-1,0],\mathrm{ps}} \,\in\,\3Cat
\end{flalign}
in the sense of Definition \ref{def:3Cat} whose
\begin{itemize}
\item objects are $\Ch_R^{[-1,0]}$-enriched categories,
\item $1$-morphisms are $\Ch_R^{[-1,0]}$-enriched functors,
\item $2$-morphisms are $\Ch_R^{[-1,0]}$-enriched pseudo-natural transformations, and
\item $3$-morphisms are $\Ch_R^{[-1,0]}$-enriched modifications.
\end{itemize}
This $3$-category admits a symmetric monoidal structure in the sense of
Definition \ref{def:SM3Cat} which agrees on objects, $1$-morphisms 
and strictly natural $2$-morphisms with the usual one from enriched
category theory, see e.g.\ \cite[Section 1.4]{Kelly}.
\end{propo}

The objects and $1$-morphisms in $\dgCat_R^{[-1,0],\mathrm{ps}}$ are defined 
according to the standard definitions of enriched category theory.
The precise definitions of the $2$-morphisms and $3$-morphisms are as follows.
\begin{defi}\label{def:pseudonatural}
Let $F,G : \CC\to\DD$ be two $\Ch_R^{[-1,0]}$-enriched functors between two
$\Ch_R^{[-1,0]}$-enriched categories $\CC$ and $\DD$. A \textit{$\Ch_R^{[-1,0]}$-enriched pseudo-natural
transformation} $\zeta : F\Rightarrow G$ consists of the following data:
\begin{itemize}
\item[(i)] For each object $c\in \CC$, an element $\zeta_c \in \DD(F(c),G(c))^0$ 
of degree $0$ in the $\Hom$-object $\DD(F(c),G(c))\in \Ch_R^{[-1,0]}$.

\item[(ii)] For each pair of objects $c,c^\prime\in \CC$, an $R$-linear map
\begin{flalign}\label{eqn:doubleindexedcomponents}
\zeta_{c,c^\prime}\,:\, \CC(c,c^\prime)^0~\longrightarrow~\DD(F(c),G(c^\prime))^{-1}\quad.
\end{flalign}
\end{itemize}
These data have to satisfy the following properties:
\begin{itemize}
\item[(1)] For all $h\in\CC(c,c^\prime)^0$,
\begin{subequations}
\begin{flalign}
G(h)\, \zeta_c - \zeta_{c^\prime}\,F(h)\,=\,\dd_{\DD}\big( \zeta_{c,c^\prime}(h)\big)\quad,
\end{flalign}
and for all $k\in \CC(c,c^\prime)^{-1}$,
\begin{flalign}
G(k)\, \zeta_c - \zeta_{c^\prime}\,F(k)\,=\, \zeta_{c,c^\prime}\big(\dd_\CC(k)\big)\quad.
\end{flalign}
\end{subequations}

\item[(2)] For all $c\in \CC$,
\begin{flalign}
\zeta_{c,c}(\id_c) \,=\,0\quad.
\end{flalign}

\item[(3)] For all $h\in \CC(c,c^\prime)^0$ and $h^\prime\in \CC(c^\prime,c^{\prime\prime})^0$,
\begin{flalign}
\zeta_{c,c^{\prime\prime}}\big(h^\prime\,h\big)\,=\, \zeta_{c^\prime,c^{\prime\prime}}(h^\prime)\, F(h) + G(h^\prime)\,\zeta_{c,c^\prime}(h)\quad.
\end{flalign}
\end{itemize}
\end{defi}

\begin{rem}
Note that in the special case where $\zeta_{c,c^\prime}=0$, for all $c,c^\prime\in \CC$, 
Definition \ref{def:pseudonatural} reduces to the usual definition of a $\Ch_R^{[-1,0]}$-enriched (strict) 
natural transformation.
\end{rem}

\begin{defi}\label{def:modi}
Let $\zeta,\eta : F\Rightarrow G: \CC\to \DD$ be two 
$\Ch_R^{[-1,0]}$-enriched pseudo-natural transformations.
A \textit{$\Ch_R^{[-1,0]}$-enriched modification} $\Gamma : \zeta \Rrightarrow \eta$
consists of the following data:
\begin{itemize}
\item[(i)] For each object $c\in \CC$, an element $\Gamma_c \in \DD(F(c),G(c))^{-1}$ 
of degree $-1$ in the $\Hom$-object $\DD(F(c),G(c))\in \Ch_R^{[-1,0]}$.
\end{itemize}
These data have to satisfy the following properties:
\begin{itemize}
\item[(1)] For all $c\in \CC$, 
\begin{flalign}
\zeta_c-\eta_c\,=\,\dd_{\DD}(\Gamma_c)\quad.
\end{flalign}

\item[(2)] For all $h\in \CC(c,c^\prime)^0$,
\begin{flalign}
\Gamma_{c^\prime} \,F(h) + \zeta_{c,c^\prime}(h)\,=\,\eta_{c,c^\prime}(h)+ G(h)\,\Gamma_c\quad.
\end{flalign}
\end{itemize}
\end{defi}

In the following paragraphs we describe explicitly
the symmetric monoidal $3$-category structure of $\dgCat_R^{[-1,0],\mathrm{ps}}$ 
in the sense of Definitions \ref{def:3Cat} and \ref{def:SM3Cat}.
To avoid overloading this appendix, we shall only spell out the relevant data 
but not the verification of their properties,  which are mostly straightforward calculations.

\paragraph{The homomorphism $2$-categories:} Given two $\Ch_R^{[-1,0]}$-enriched
categories $\CC$ and $\DD$, we denote by $\dgCat_R^{[-1,0],\mathrm{ps}}\big(\CC,\DD\big)\in \2Cat^{\mathrm{ps}}$
the strict $2$-category whose objects are $\Ch_R^{[-1,0]}$-enriched functors,
$1$-morphisms are $\Ch_R^{[-1,0]}$-enriched pseudo-natural transformations and
$2$-morphisms are $\Ch_R^{[-1,0]}$-enriched modifications. The composition
of two $1$-morphisms
\begin{subequations}
\begin{equation}
\begin{tikzcd}
	{\mathbf{C}} && {\mathbf{D}}
	\arrow[""{name=0, anchor=center, inner sep=0}, "G"{description}, from=1-1, to=1-3]
	\arrow[""{name=1, anchor=center, inner sep=0}, "F", curve={height=-32pt}, from=1-1, to=1-3]
	\arrow[""{name=2, anchor=center, inner sep=0}, "H"',curve={height=32pt}, from=1-1, to=1-3]
	\arrow["\zeta"', shorten <=5pt, shorten >=5pt, Rightarrow, from=1, to=0]
	\arrow["\eta"', shorten <=5pt, shorten >=5pt, Rightarrow, from=0, to=2]
\end{tikzcd}~~\stackrel{\circ}{\longmapsto}~~
\begin{tikzcd}
	{\mathbf{C}} && {\mathbf{D}}
	\arrow[""{name=1, anchor=center, inner sep=0}, "F", curve={height=-32pt}, from=1-1, to=1-3]
	\arrow[""{name=2, anchor=center, inner sep=0}, "H"', curve={height=32pt}, from=1-1, to=1-3]
	\arrow["\eta\circ \zeta"', shorten <=5pt, shorten >=5pt, Rightarrow, from=1, to=2]
\end{tikzcd}
\end{equation}
in the $2$-category $\dgCat_R^{[-1,0],\mathrm{ps}}\big(\CC,\DD\big)$ is defined by
\begin{flalign}
(\eta\circ\zeta)_c\,&:=\, \eta_c\,\zeta_c\quad,\\
(\eta\circ\zeta)_{c,c^\prime}(h)\,&:=\, \eta_{c,c^\prime}(h)\,\zeta_c + \eta_{c^\prime}\,\zeta_{c,c^\prime}(h)\quad,
\end{flalign}
\end{subequations}
for all $c,c^\prime\in\CC$ and $h\in \CC(c,c^\prime)^0$. This composition
is strictly associative and strictly unital with respect to the identity $1$-morphisms
$\Id_F : F\Rightarrow F : \CC\to \DD$ defined by $(\Id_F)_c = \id_{F(c)}$ and
$(\Id_{F})_{c,c^\prime}=0$.
\sk

The horizontal composition of two $2$-morphisms
\begin{subequations}
\begin{equation}
\begin{tikzcd}
	{\mathbf{C}} && {\mathbf{D}}
	\arrow[""{name=0, anchor=center, inner sep=0}, "G"{description}, from=1-1, to=1-3]
	\arrow[""{name=1, anchor=center, inner sep=0}, "H"', curve={height=36pt}, from=1-1, to=1-3]
	\arrow[""{name=2, anchor=center, inner sep=0}, "F", curve={height=-36pt}, from=1-1, to=1-3]
	\arrow[""{name=3, anchor=center, inner sep=0}, "{\,\zeta^\prime}", curve={height=-16pt}, shorten <=5pt, shorten >=5pt, Rightarrow, from=2, to=0]
	\arrow[""{name=4, anchor=center, inner sep=0}, "\zeta\,"', curve={height=16pt}, shorten <=5pt, shorten >=5pt, Rightarrow, from=2, to=0]
	\arrow[""{name=5, anchor=center, inner sep=0}, "\eta\,"', curve={height=16pt}, shorten <=5pt, shorten >=5pt, Rightarrow, from=0, to=1]
	\arrow[""{name=6, anchor=center, inner sep=0}, "{\,\eta^\prime}", curve={height=-16pt}, shorten <=5pt, shorten >=5pt, Rightarrow, from=0, to=1]
	\arrow["\Gamma", shorten <=5pt, shorten >=5pt, triple, from=4, to=3]
	\arrow["\Theta", shorten <=5pt, shorten >=5pt, triple, from=5, to=6]
\end{tikzcd}
~~\stackrel{\circ}{\longmapsto}~~
\begin{tikzcd}
	{\mathbf{C}} &&& {\mathbf{D}}
	\arrow[""{name=1, anchor=center, inner sep=0}, "H"', curve={height=36pt}, from=1-1, to=1-4]
	\arrow[""{name=2, anchor=center, inner sep=0}, "F", curve={height=-36pt}, from=1-1, to=1-4]
	\arrow[""{name=3, anchor=center, inner sep=0}, "{\,\eta^\prime \circ \zeta^\prime}", curve={height=-20pt}, shorten <=5pt, shorten >=5pt, Rightarrow, from=2, to=1]
	\arrow[""{name=4, anchor=center, inner sep=0}, "\eta\circ \zeta\,"', curve={height=20pt}, shorten <=5pt, shorten >=5pt, Rightarrow, from=2, to=1]
	\arrow["\Theta\circ \Gamma", shorten <=5pt, shorten >=5pt, triple, from=4, to=3]
\end{tikzcd}
\end{equation}
in the $2$-category $\dgCat_R^{[-1,0],\mathrm{ps}}\big(\CC,\DD\big)$ is defined by
\begin{flalign}
(\Theta\circ\Gamma)_c\,:=\, \Theta_c\,\zeta_c + \eta^\prime_c\,\Gamma_c\quad,
\end{flalign}
\end{subequations}
for all $c\in \CC$. This composition is strictly associative and strictly unital
with respect to the identity $2$-morphisms
$\ID_{\Id_F} : \Id_F\Rrightarrow \Id_F : F\Rightarrow F : \CC\to \DD$ defined by $(\ID_{\Id_F})_c = 0$.
\sk

The vertical composition of two $2$-morphisms
\begin{subequations}
\begin{equation}
\begin{tikzcd}
	{\mathbf{C}} &&& {\mathbf{D}}
	\arrow[""{name=0, anchor=center, inner sep=0}, "G"', curve={height=36pt}, from=1-1, to=1-4]
	\arrow[""{name=1, anchor=center, inner sep=0}, "F", curve={height=-36pt}, from=1-1, to=1-4]
	\arrow[""{name=2, anchor=center, inner sep=0}, "{\,\zeta^{\prime\prime}}", curve={height=-28pt}, shorten <=8pt, shorten >=8pt, Rightarrow, from=1, to=0]
	\arrow[""{name=3, anchor=center, inner sep=0}, "\zeta\,"', curve={height=28pt}, shorten <=8pt, shorten >=8pt, Rightarrow, from=1, to=0]
	\arrow[""{name=4, anchor=center, inner sep=0}, "{\zeta^\prime}"{description}, shorten <=6pt, shorten >=6pt, Rightarrow, from=1, to=0]
	\arrow["\Gamma", shorten <=4pt, shorten >=4pt, triple, from=3, to=4]
	\arrow["{\Gamma^\prime}", shorten <=4pt, shorten >=4pt, triple, from=4, to=2]
\end{tikzcd}
~~\longmapsto~~
\begin{tikzcd}
	{\mathbf{C}} &&& {\mathbf{D}}
	\arrow[""{name=0, anchor=center, inner sep=0}, "G"', curve={height=36pt}, from=1-1, to=1-4]
	\arrow[""{name=1, anchor=center, inner sep=0}, "F", curve={height=-36pt}, from=1-1, to=1-4]
	\arrow[""{name=2, anchor=center, inner sep=0}, "{\,\zeta^{\prime\prime}}", curve={height=-28pt}, shorten <=8pt, shorten >=8pt, Rightarrow, from=1, to=0]
	\arrow[""{name=3, anchor=center, inner sep=0}, "\zeta\,"', curve={height=28pt}, shorten <=8pt, shorten >=8pt, Rightarrow, from=1, to=0]
	\arrow["\Gamma^\prime\,\Gamma", shorten <=4pt, shorten >=4pt, triple, from=3, to=2]
\end{tikzcd}
\end{equation}
in the $2$-category $\dgCat_R^{[-1,0],\mathrm{ps}}\big(\CC,\DD\big)$ is defined by
\begin{flalign}
(\Gamma^\prime\,\Gamma)_c \,:=\, \Gamma^\prime_c + \Gamma_c\quad,
\end{flalign}
\end{subequations}
for all $c\in\CC$. This composition is strictly associative and strictly unital
with respect to the identity $2$-morphisms
$\ID_{\zeta} : \zeta\Rrightarrow \zeta : F\Rightarrow F : \CC\to \DD$ defined by $(\ID_{\zeta})_c = 0$.

\paragraph{The composition pseudo-functors and identities:} Given three $\Ch_R^{[-1,0]}$-enriched
categories $\CC$, $\DD$ and $\EE$, we denote by 
\begin{flalign}
\ast\,:\,\dgCat_R^{[-1,0],\mathrm{ps}}\big(\DD,\EE\big) \times 
\dgCat_R^{[-1,0],\mathrm{ps}}\big(\CC,\DD\big) ~\longrightarrow~
\dgCat_R^{[-1,0],\mathrm{ps}}\big(\CC,\EE\big)
\end{flalign}
the composition pseudo-functor between the homomorphism $2$-categories
from the previous paragraph.
On objects in the homomorphism $2$-categories,
this pseudo-functor acts in terms of the usual composition of $\Ch_R^{[-1,0]}$-enriched functors,
which we will denote by
\begin{equation}
\begin{tikzcd}
	{\mathbf{C}} & {\mathbf{D}} & {\mathbf{E}}
	\arrow["F", from=1-1, to=1-2]
	\arrow["F^\prime", from=1-2, to=1-3]
\end{tikzcd}
~~\stackrel{\ast}{\longmapsto}~~
\begin{tikzcd}
	{\mathbf{C}} & {\mathbf{E}}
	\arrow["F^\prime\ast F", from=1-1, to=1-2]
\end{tikzcd}\quad.
\end{equation}
On $1$-morphisms in the homomorphism $2$-categories, the pseudo-functor $\ast$ acts as
\begin{subequations}
\begin{equation}
\begin{tikzcd}
	{\mathbf{C}} && {\mathbf{D}} && {\mathbf{E}}
	\arrow[""{name=0, anchor=center, inner sep=0}, "F", curve={height=-32pt}, from=1-1, to=1-3]
	\arrow[""{name=1, anchor=center, inner sep=0}, "G"', curve={height=32pt}, from=1-1, to=1-3]
	\arrow[""{name=2, anchor=center, inner sep=0}, "{F^\prime}", curve={height=-32pt}, from=1-3, to=1-5]
	\arrow[""{name=3, anchor=center, inner sep=0}, "{G^\prime}"', curve={height=32pt}, from=1-3, to=1-5]
	\arrow["\zeta\,"', shorten <=6pt, shorten >=6pt, Rightarrow, from=0, to=1]
	\arrow["{\,\zeta^\prime}"', shorten <=6pt, shorten >=6pt, Rightarrow, from=2, to=3]
\end{tikzcd}
~~\stackrel{\ast}{\longmapsto}~~
\begin{tikzcd}
	{\mathbf{C}} && {\mathbf{E}}
	\arrow[""{name=0, anchor=center, inner sep=0}, "F^\prime \ast F", curve={height=-32pt}, from=1-1, to=1-3]
	\arrow[""{name=1, anchor=center, inner sep=0}, "G^\prime\ast G"', curve={height=32pt}, from=1-1, to=1-3]
	\arrow["\zeta^\prime\ast \zeta\,"', shorten <=6pt, shorten >=6pt, Rightarrow, from=0, to=1]
\end{tikzcd}
\end{equation}
with
\begin{flalign}
(\zeta^\prime\ast\zeta)_c\,&:=\,\zeta^\prime_{G(c)}\,F^\prime(\zeta_c)\quad,\\
(\zeta^\prime\ast\zeta)_{c,c^\prime}(h)\,&:=\,\zeta^\prime_{G(c),G(c^\prime)}\big(G(h)\big)\,F^\prime(\zeta_c)
+\zeta^\prime_{G(c^\prime)}\,F^\prime\big(\zeta_{c,c^\prime}(h)\big)\quad,
\end{flalign}
\end{subequations}
for all $c,c^\prime\in\CC$ and $h\in \CC(c,c^\prime)^0$. 
On $2$-morphisms in the homomorphism $2$-categories, the pseudo-functor $\ast$ acts as
\begin{subequations}
\begin{equation}
\begin{tikzcd}
	{\mathbf{C}} && {\mathbf{D}} && {\mathbf{E}}
	\arrow[""{name=0, anchor=center, inner sep=0}, "F", curve={height=-32pt}, from=1-1, to=1-3]
	\arrow[""{name=1, anchor=center, inner sep=0}, "G"', curve={height=32pt}, from=1-1, to=1-3]
	\arrow[""{name=2, anchor=center, inner sep=0}, "{F^\prime}", curve={height=-32pt}, from=1-3, to=1-5]
	\arrow[""{name=3, anchor=center, inner sep=0}, "{G^\prime}"', curve={height=32pt}, from=1-3, to=1-5]
	\arrow[""{name=4, anchor=center, inner sep=0}, "\zeta\,"', curve={height=18pt}, shorten <=7pt, shorten >=7pt, Rightarrow, from=0, to=1]
	\arrow[""{name=5, anchor=center, inner sep=0}, "\,\eta", curve={height=-18pt}, shorten <=7pt, shorten >=7pt, Rightarrow, from=0, to=1]
	\arrow[""{name=6, anchor=center, inner sep=0}, "{\,\zeta^\prime}"', curve={height=18pt}, shorten <=7pt, shorten >=7pt, Rightarrow, from=2, to=3]
	\arrow[""{name=7, anchor=center, inner sep=0}, "{\, \eta^\prime}", curve={height=-18pt}, shorten <=7pt, shorten >=7pt, Rightarrow, from=2, to=3]
	\arrow["\Gamma", shorten <=5pt, shorten >=5pt, triple, from=4, to=5]
	\arrow["{\Gamma^\prime}", shorten <=5pt, shorten >=5pt, triple, from=6, to=7]
\end{tikzcd}
~~\stackrel{\ast}{\longmapsto}~~
\begin{tikzcd}
	{\mathbf{C}} &&& {\mathbf{E}}
	\arrow[""{name=0, anchor=center, inner sep=0}, "F^\prime \ast F", curve={height=-32pt}, from=1-1, to=1-4]
	\arrow[""{name=1, anchor=center, inner sep=0}, "G^\prime\ast G"', curve={height=32pt}, from=1-1, to=1-4]
	\arrow[""{name=4, anchor=center, inner sep=0}, "\zeta^\prime \ast \zeta\,"', curve={height=18pt}, shorten <=7pt, shorten >=7pt, Rightarrow, from=0, to=1]
	\arrow[""{name=5, anchor=center, inner sep=0}, "\,\eta^\prime\ast \eta", curve={height=-18pt}, shorten <=7pt, shorten >=7pt, Rightarrow, from=0, to=1]
	\arrow["\Gamma^\prime\ast \Gamma", shorten <=5pt, shorten >=5pt, triple, from=4, to=5]
\end{tikzcd}
\end{equation}
with
\begin{flalign}
(\Gamma^\prime\ast\Gamma)_c\,:=\,\Gamma^\prime_{G(c)}\,F^\prime(\zeta_c) + \eta^\prime_{G(c)}\, F^\prime(\Gamma_c)\quad,
\end{flalign}
\end{subequations}
for all $c\in\CC$.
\sk

The pseudo-functor coherences for the identities
are trivial, i.e.\ $\Id_{F^\prime\ast F} = \Id_{F^\prime}\ast \Id_F$,
but those for the compositions are not. Given any composable diagram of the form
\begin{subequations}\label{eqn:compositioncoherences}
\begin{equation}
\begin{tikzcd}
	{\mathbf{C}} && {\mathbf{D}} && {\mathbf{E}}
	\arrow[""{name=0, anchor=center, inner sep=0}, "F", curve={height=-32pt}, from=1-1, to=1-3]
	\arrow[""{name=1, anchor=center, inner sep=0}, "G"{description}, from=1-1, to=1-3]
	\arrow[""{name=2, anchor=center, inner sep=0}, "H"', curve={height=32pt}, from=1-1, to=1-3]
	\arrow[""{name=3, anchor=center, inner sep=0}, "{F^\prime}", curve={height=-32pt}, from=1-3, to=1-5]
	\arrow[""{name=4, anchor=center, inner sep=0}, "{H^\prime}"', curve={height=32pt}, from=1-3, to=1-5]
	\arrow[""{name=5, anchor=center, inner sep=0}, "{G^\prime}"{description}, from=1-3, to=1-5]
	\arrow["\zeta"', shorten <=5pt, shorten >=5pt, Rightarrow, from=0, to=1]
	\arrow["\eta"', shorten <=5pt, shorten >=5pt, Rightarrow, from=1, to=2]
	\arrow["{\zeta^\prime}"', shorten <=5pt, shorten >=5pt, Rightarrow, from=3, to=5]
	\arrow["{\eta^\prime}"', shorten <=5pt, shorten >=5pt, Rightarrow, from=5, to=4]
\end{tikzcd}\quad,
\end{equation}
the corresponding pseudo-functor coherence
\begin{flalign}
\ast_{(\eta^\prime,\eta),(\zeta^\prime,\zeta)}^{}\,:\,(\eta^\prime\ast \eta)\circ (\zeta^\prime\ast \zeta)~\xRrightarrow{\,~~\,}~(\eta^\prime\circ\zeta^\prime)\ast(\eta\circ\zeta)
\end{flalign}
is given by the $\Ch_R^{[-1,0]}$-enriched modification
whose components are defined by
\begin{flalign}
\big(\ast_{(\eta^\prime,\eta),(\zeta^\prime,\zeta)}^{}\big)_c\,:=\,\eta^\prime_{H(c)}\,\zeta^\prime_{G(c),H(c)}(\eta_c)\,F^\prime(\zeta_c)\quad,
\end{flalign}
\end{subequations}
for all $c\in \CC$.
\sk

The composition pseudo-functors $\ast$ are strictly associative and strictly unital
with respect to the identities $\id_\CC \in \dgCat_R^{[-1,0],\mathrm{ps}}\big(\CC,\CC\big)$ 
given by the identity $\Ch_R^{[-1,0]}$-enriched functors $\id_\CC:\CC\to \CC$.

\paragraph{The symmetric monoidal structure:} The monoidal product
\begin{flalign}
\boxtimes\,:\, 
\dgCat_R^{[-1,0],\mathrm{ps}} \times 
\dgCat_R^{[-1,0],\mathrm{ps}}
~\longrightarrow~
\dgCat_R^{[-1,0],\mathrm{ps}}
\end{flalign}
is given by the $\2Cat^{\mathrm{ps}}$-enriched functor which assigns to a pair
of objects $\CC,\DD\in \dgCat_R^{[-1,0],\mathrm{ps}}$ the 
usual $\Ch_R^{[-1,0]}$-enriched tensor product category $\CC\boxtimes \DD\in  \dgCat_R^{[-1,0],\mathrm{ps}}$
and to a pair of $1$-morphisms $F : \CC\to \CC^\prime$ and $G : \DD\to\DD^\prime $
in $\dgCat_R^{[-1,0],\mathrm{ps}}$ the 
usual $\Ch_R^{[-1,0]}$-enriched tensor product functor 
$F\boxtimes G : \CC\boxtimes \DD\to \CC^\prime\boxtimes \DD^\prime$
from \cite[Section 1.4]{Kelly}.  
\sk

The monoidal product 
of a pair of $2$-morphisms
\begin{subequations}
\begin{equation}
\begin{tikzcd}
	{\mathbf{C}} && {\mathbf{C}^\prime}
	\arrow[""{name=0, anchor=center, inner sep=0}, "F", curve={height=-32pt}, from=1-1, to=1-3]
	\arrow[""{name=1, anchor=center, inner sep=0}, "{F^\prime}"', curve={height=32pt}, from=1-1, to=1-3]
	\arrow["\zeta\,"', shorten <=6pt, shorten >=6pt, Rightarrow, from=0, to=1]
\end{tikzcd}
~~
\begin{tikzcd}
	{\mathbf{D}} && {\mathbf{D}^\prime}
	\arrow[""{name=0, anchor=center, inner sep=0}, "G", curve={height=-32pt}, from=1-1, to=1-3]
	\arrow[""{name=1, anchor=center, inner sep=0}, "{G^\prime}"', curve={height=32pt}, from=1-1, to=1-3]
	\arrow["\eta\,"', shorten <=6pt, shorten >=6pt, Rightarrow, from=0, to=1]
\end{tikzcd}
~~\stackrel{\boxtimes}{\longmapsto}~~
\begin{tikzcd}
	{\mathbf{C}\boxtimes \mathbf{D}} && {\mathbf{C}^\prime\boxtimes \mathbf{D}^\prime}
	\arrow[""{name=0, anchor=center, inner sep=0}, "F\boxtimes G", curve={height=-32pt}, from=1-1, to=1-3]
	\arrow[""{name=1, anchor=center, inner sep=0}, "{F^\prime\boxtimes G^\prime}"', curve={height=32pt}, from=1-1, to=1-3]
	\arrow["\zeta\boxtimes \eta\,"', shorten <=6pt, shorten >=6pt, Rightarrow, from=0, to=1]
\end{tikzcd}
\end{equation}
in  $\dgCat_R^{[-1,0],\mathrm{ps}}$ is defined by
\begin{flalign}
(\zeta\boxtimes \eta)_{(c,d)}\,&:=\,\zeta_c\widetilde{\otimes}_R\eta_d\quad,\\
(\zeta\boxtimes \eta)_{(c,d),(c^\prime,d^\prime)}\big(h\widetilde{\otimes}_R k\big)\,&:=\,
\zeta_{c,c^\prime}(h)\widetilde{\otimes}_R G^\prime(k)\,\eta_d + \zeta_{c^\prime}\,F(h)\widetilde{\otimes}_R\eta_{d,d^\prime}(k)\quad,
\end{flalign}
\end{subequations}
for all $c,c^\prime\in\CC$, $d,d^\prime\in \DD$, $h\in\CC(c,c^\prime)^0$
and $k\in\DD(d,d^\prime)^0$. Note that for two $\Ch_R^{[-1,0]}$-enriched 
strictly natural transformations, i.e.\ $\zeta_{c,c^\prime}=0$ and $\eta_{d,d^\prime}=0$
for all objects $c,c^\prime\in\CC$ and $d,d^\prime\in\DD$, the tensor
product $\zeta\boxtimes \eta$ agrees with the usual one from enriched category theory.
\sk

The monoidal product of a pair of $3$-morphisms
\begin{subequations}
\begin{equation}
\begin{tikzcd}
	{\mathbf{C}} && {\mathbf{C}^\prime}
	\arrow[""{name=0, anchor=center, inner sep=0}, "F", curve={height=-32pt}, from=1-1, to=1-3]
	\arrow[""{name=1, anchor=center, inner sep=0}, "{F^\prime}"', curve={height=32pt}, from=1-1, to=1-3]
	\arrow[""{name=2, anchor=center, inner sep=0}, "\zeta\,"', curve={height=18pt}, shorten <=7pt, shorten >=7pt, Rightarrow, from=0, to=1]
	\arrow[""{name=3, anchor=center, inner sep=0}, "{\,\zeta^\prime}", curve={height=-18pt}, shorten <=7pt, shorten >=7pt, Rightarrow, from=0, to=1]
	\arrow["\Gamma", shorten <=5pt, shorten >=5pt, triple, from=2, to=3]
\end{tikzcd}
~~
\begin{tikzcd}
	{\mathbf{D}} && {\mathbf{D}^\prime}
	\arrow[""{name=0, anchor=center, inner sep=0}, "G", curve={height=-32pt}, from=1-1, to=1-3]
	\arrow[""{name=1, anchor=center, inner sep=0}, "{G^\prime}"', curve={height=32pt}, from=1-1, to=1-3]
	\arrow[""{name=2, anchor=center, inner sep=0}, "\eta\,"', curve={height=18pt}, shorten <=7pt, shorten >=7pt, Rightarrow, from=0, to=1]
	\arrow[""{name=3, anchor=center, inner sep=0}, "{\,\eta^\prime}", curve={height=-18pt}, shorten <=7pt, shorten >=7pt, Rightarrow, from=0, to=1]
	\arrow["\Theta", shorten <=5pt, shorten >=5pt, triple, from=2, to=3]
\end{tikzcd}
~~\stackrel{\boxtimes}{\longmapsto}~~
\begin{tikzcd}
	{\mathbf{C}\boxtimes \mathbf{D}} &&& {\mathbf{C}^\prime \boxtimes \mathbf{D}^\prime}
	\arrow[""{name=0, anchor=center, inner sep=0}, "F\boxtimes G", curve={height=-32pt}, from=1-1, to=1-4]
	\arrow[""{name=1, anchor=center, inner sep=0}, "{F^\prime}\boxtimes G^\prime"', curve={height=32pt}, from=1-1, to=1-4]
	\arrow[""{name=2, anchor=center, inner sep=0}, "\zeta\boxtimes \eta\,"', curve={height=22pt}, shorten <=7pt, shorten >=7pt, Rightarrow, from=0, to=1]
	\arrow[""{name=3, anchor=center, inner sep=0}, "{\,\zeta^\prime\boxtimes \eta^\prime}", curve={height=-22pt}, shorten <=7pt, shorten >=7pt, Rightarrow, from=0, to=1]
	\arrow["\Gamma\boxtimes \Theta", shorten <=5pt, shorten >=5pt, triple, from=2, to=3]
\end{tikzcd}
\end{equation}
in  $\dgCat_R^{[-1,0],\mathrm{ps}}$ is defined by
\begin{flalign}
(\Gamma\boxtimes \Theta)_{(c,d)}\,:=\,\Gamma_c\widetilde{\otimes}_R \eta_d + \zeta^\prime_c \widetilde{\otimes}_R\Theta_d\quad,
\end{flalign}
\end{subequations}
for all $c\in\CC$ and $d\in\DD$.
\sk

The monoidal unit $u :\mathcal{I}\to\dgCat_R^{[-1,0],\mathrm{ps}}$
is the $\Ch_R^{[-1,0]}$-enriched category
\begin{flalign}
\mathbf{R}\,\in\,  \dgCat_R^{[-1,0],\mathrm{ps}}
\end{flalign}
consisting of a single object, say $\star$, with morphisms
$\mathbf{R}(\star,\star):= R\in \Ch_R^{[-1,0]}$ given by
the monoidal unit of the enrichment category.
\sk

The associator and unitors for this monoidal structure 
are canonically given by the ones of the enrichment 
symmetric monoidal category $\Ch_R^{[-1,0]}$.
The component at $\CC,\DD\in \dgCat_R^{[-1,0],\mathrm{ps}}$ of the symmetric braiding 
\begin{subequations}\label{eqn:tauChCat}
\begin{flalign}
\tau_{\CC,\DD}\,:\,\CC\boxtimes \DD ~\longrightarrow~\DD \boxtimes \CC
\end{flalign}
is the $\Ch_R^{[-1,0]}$-enriched functor which sends objects
$(c,d)\in \CC\boxtimes \DD $ to $(d,c)\in \DD\boxtimes \CC$
and acts on morphisms via the symmetric braiding of the enrichment 
symmetric monoidal category $\Ch_R^{[-1,0]}$
\begin{flalign}
\begin{gathered}
\xymatrix@C=5em@R=1.5em{
(\CC\boxtimes \DD)\big((c,d),(c^\prime,d^\prime)\big) \ar@{=}[d]\ar@{-->}[r]~&~\ar@{=}[d]
(\DD\boxtimes \CC)\big((d,c),(d^\prime,c^\prime)\big)\\
\CC(c,c^\prime)\widetilde{\otimes}_{R}\DD(d,d^\prime) \ar[r]_-{\widetilde{\gamma}_{\CC(c,c^\prime),\DD(d,d^\prime)}}
~&~\DD(d,d^\prime)\widetilde{\otimes}_{R}\CC(c,c^\prime)
}
\end{gathered}\quad.
\end{flalign}
\end{subequations}
Note that the symmetric braiding satisfies the strict hexagon identities, i.e.\
\begin{subequations}\label{eqn:tauChCat-Hexs}
\begin{flalign}
\tau_{\CC\boxtimes\DD,\EE}\,&=\,(\tau_{\CC,\EE}\boxtimes\id_{\DD})\ast(\id_{\CC}\boxtimes\tau_{\DD,\EE})\quad,\\
\tau_{\CC,\DD\boxtimes\EE}\,&=\,(\id_{\DD}\boxtimes\tau_{\CC,\EE})\ast(\tau_{\CC,\DD}\boxtimes\id_{\EE})\quad,
\end{flalign}
\end{subequations}
for all $\CC,\DD,\EE\in \dgCat_R^{[-1,0],\mathrm{ps}}$.
Moreover, the symmetric braiding $\tau$ is a $\2Cat^{\mathrm{ps}}$-enriched natural 
transformation, which in particular implies that it is strictly natural with respect to 
$\Ch_R^{[-1,0]}$-enriched functors $F : \CC\to \CC^\prime$, i.e.\
\begin{flalign}
\label{eqn:tauChCat-natural}
(\id_\DD\boxtimes F)\ast\tau_{\CC,\DD}=\tau_{\CC^\prime,\DD}\ast(F\boxtimes \id_\DD)\quad,
\end{flalign}
for all $\DD\in \dgCat_R^{[-1,0],\mathrm{ps}}$.
\sk

To conclude this subsection, we provide a convenient
definition of braided monoidal category objects
in the symmetric monoidal $3$-category $\dgCat_R^{[-1,0],\mathrm{ps}}$.
Such objects can be defined with various levels of weakness since
our ambient category $\dgCat_R^{[-1,0],\mathrm{ps}}$ is a $3$-category.
For the purpose of our first-order deformation quantizations in 
Section \ref{sec:2braidings}, it suffices to consider the following semi-strict 
variant.\footnote{A weaker variant of braided monoidal category objects in
$\dgCat_R^{[-1,0],\mathrm{ps}}$, allowing for non-trivial pentagonator and hexagonator
modifications, however seems to be required for the higher-order deformations 
explored in Subsection \ref{subsec:higherorders}.}
\begin{defi}\label{def:2TermBraidedCat}
A (semi-strict) \textit{braided monoidal category object in $\dgCat_R^{[-1,0],\mathrm{ps}}$} 
is a tuple $(\CC,\otimes,1,\alpha,\lambda,\rho,\gamma)$
consisting of the following data:
\begin{itemize}
\item[(i)] a $\Ch_R^{[-1,0]}$-enriched category $\CC$,

\item[(ii)] a $\Ch_R^{[-1,0]}$-enriched functor $\otimes : \CC\boxtimes\CC\to \CC$ (called monoidal product),

\item[(iii)] a $\Ch_R^{[-1,0]}$-enriched functor $1 : \mathbf{R}\to \CC$ (called monoidal unit),

\item[(iv)] invertible $\Ch_R^{[-1,0]}$-enriched pseudo-natural transformations 
$\alpha : \otimes\ast (\otimes\boxtimes \id_\CC) \Rightarrow \otimes\ast (\id_\CC\boxtimes \otimes)$
(called associator), $\lambda : \otimes \ast(1\boxtimes\id_\CC) \Rightarrow\id_\CC$ (called left unitor)
and $\rho: \otimes \ast(\id_\CC\boxtimes 1)\Rightarrow\id_\CC$ (called right unitor),

\item[(v)] an invertible $\Ch_R^{[-1,0]}$-enriched pseudo-natural transformation
$\gamma : \otimes \Rightarrow\otimes\ast \tau_{\CC,\CC}$ (called braiding).
\end{itemize}
These data have to satisfy the usual coherence conditions for braided monoidal categories
internal to the $3$-category $\dgCat_R^{[-1,0],\mathrm{ps}}$ strictly,
i.e.\ the pentagon, triangle and hexagon diagrams commute strictly 
up to equality of $\Ch_R^{[-1,0]}$-enriched pseudo-natural 
transformations.
\sk

A (semi-strict) \textit{symmetric monoidal category object in $\dgCat_R^{[-1,0],\mathrm{ps}}$} 
is a braided monoidal category object $(\CC,\otimes,1,\alpha,\lambda,\rho,\gamma)$
with the property that the braiding $\gamma$ squares strictly to the identity,
i.e.\ the $\circ$-composition of
\begin{flalign}
\xymatrix@C=3.5em{
\otimes \ar@{=>}[r]^-{\gamma} ~&~\otimes\ast\tau_{\CC,\CC} \ar@{=>}[r]^-{\gamma\ast\Id_{\tau_{\CC,\CC}}}
~&~\otimes \ast \tau_{\CC,\CC}\ast \tau_{\CC,\CC} \,=\, \otimes
}
\end{flalign}
is equal to the identity $\Ch_R^{[-1,0]}$-enriched pseudo-natural transformation $\Id_\otimes$.
\end{defi}

\subsection{Homotopy \texorpdfstring{$2$}{2}-categories of symmetric monoidal dg-categories}
\label{app:Ho2Cat}
Let $R$ be a commutative $\bbK$-algebra and denote by $\dgCat_R$ the $2$-category
of $\Ch_R$-enriched categories (dg-categories), $\Ch_R$-enriched functors (dg-functors)
and $\Ch_R$-enriched strict natural transformations (dg-natural transformations).
We shall regard $\dgCat_R$ as a $3$-category (with trivial $3$-morphisms) 
in the sense of Definition \ref{def:3Cat}.
The lax symmetric monoidal truncation functor  $\tau^{[-1,0]}: \Ch_R\to \Ch_R^{[-1,0]}$
from Lemma \ref{lem:laxtruncation} induces a change of base $3$-functor
\begin{flalign}\label{eqn:homotopy2category}
\mathsf{Ho}_2\,:\, \dgCat_R~\longrightarrow~\dgCat_R^{[-1,0]}\,\subseteq\,\dgCat_R^{[-1,0],\mathrm{ps}}\quad,
\end{flalign}
which sends dg-natural transformations to strict
$\Ch_R^{[-1,0]}$-enriched pseudo-natural transformations.
Observe that, at the level of objects, this $3$-functor assigns to a dg-category $\CC\in \dgCat_R$, 
with objects $\CC_0$ and morphism cochain complexes $\CC(c,c^\prime)\in\Ch_R$, 
its homotopy $2$-category $\mathsf{Ho}_2(\CC)\in \dgCat_R^{[-1,0],\mathrm{ps}}$, 
i.e.\ the $\Ch_R^{[-1,0]}$-enriched category with the same objects $\CC_0$ and morphisms
given by the truncated cochain complexes $\tau^{[-1,0]}\big(\CC(c,c^\prime)\big)\in\Ch_R^{[-1,0]}$.
\sk

Since the enrichment category $\Ch_R$ of cochain complexes of $R$-modules is symmetric monoidal
with respect to the relative tensor product $\otimes_R$, there exists an induced
symmetric monoidal structure on $\dgCat_R$, see e.g.\ \cite[Section 1.4]{Kelly}.
The $3$-functor \eqref{eqn:homotopy2category} assigning homotopy $2$-categories
is lax symmetric monoidal with the lax structure 
$\mathsf{Ho}_2(\CC)\boxtimes \mathsf{Ho}_2(\DD) \to \mathsf{Ho}_2(\CC\boxtimes \DD)$
induced from the lax structure of the good truncation functor in
Lemma \ref{lem:laxtruncation}, for all $\CC,\DD\in\dgCat_R$.
For later use, we record the following direct consequence of this result.
\begin{cor}
Let $(\CC,\otimes,1,\gamma)$ be a braided (or symmetric) monoidal dg-category.
Then the homotopy $2$-category $\mathsf{Ho}_2(\CC)\in \dgCat_R^{[-1,0],\mathrm{ps}}$
is canonically a braided (or symmetric) monoidal category object in
the symmetric monoidal $3$-category $\dgCat_R^{[-1,0],\mathrm{ps}}$ from Proposition \ref{prop:Chtrunccat}.
\end{cor}

\begin{rem}\label{rem:Hostrict}
Note that the braided (or symmetric) monoidal category objects in $\dgCat_R^{[-1,0],\mathrm{ps}}$ 
which are obtained from this construction always have strictly natural
associators, unitors and braidings. Our symmetric monoidal 
$3$-category $\dgCat_R^{[-1,0],\mathrm{ps}}$ from Proposition \ref{prop:Chtrunccat}
is designed to provide more flexibility through the concept 
of enriched pseudo-natural transformations, see also Definition 
\ref{def:2TermBraidedCat}, which will be relevant 
for our deformation constructions in Section \ref{sec:2braidings}.
\end{rem}



\begin{thebibliography}{10}

\bibitem[BC04]{BaezCrans}
J.~C.~Baez and A.~S.~Crans,
``Higher-dimensional algebra VI: Lie $2$-algebras,''
Theory Appl.\ Categ.\ \textbf{12}, 492--538 (2004)
[arXiv:math/0307263 [math.QA]].


\bibitem[BSZ13]{Zhu}
C.~Bai, Y.~Sheng and C.~Zhu,
``Lie $2$-bialgebras,''
Comm.\ Math.\ Phys.\ \textbf{320}, 149--172 (2013)
[arXiv:1109.1344 [math-ph]].


\bibitem[BV15]{Voronov}
D.~Bashkirov and A.~A.~Voronov,
``$r_\infty$-matrices, triangular $L_\infty$-bialgebras and quantum $\infty$-groups,''
in: P.~Kielanowski, P.~Bieliavsky, A.~Odzijewicz, M.~Schlichenmaier and T.~Voronov (eds.),
{\it Geometric Methods in Physics}, 
Trends in Mathematics, Birkh\"auser, Cham (2015)
[arXiv:1412.2413 [math.QA]].


\bibitem[BRW23]{Bordemann}
M.~Bordemann, A.~Rivezzi and T.~Weigel,
``A gentle introduction to Drinfel'd associators,''
\textit{to appear in Reviews in Mathematical Physics}
[arXiv:2304.07012 [math.QA]].


\bibitem[CPTVV17]{CPTVV}
D.~Calaque, T.~Pantev, B.~To{\"e}n, M.~Vaqui{\'e} and G.~Vezzosi,
``Shifted Poisson structures and deformation quantization,''
J.\ Topol.\ \textbf{10}, no.\ 2, 483--584 (2017)
[arXiv:1506.03699 [math.AG]].


\bibitem[Car93]{Cartier}
P.~Cartier,
``Construction combinatoire des invariants de Vassiliev-Kontsevich des n\oe uds,''
C.\ R.\ Acad.\ Sci.\ Paris S{\'e}r.\ I Math.\ \textbf{316}, no.\ 11, 1205--1210 (1993).


\bibitem[Che24]{ChenPhD}
H.~Chen,
{\it Hopf $2$-algebras: Homotopy higher symmetries in physics},
PhD thesis, University of Waterloo (2024).
\url{http://hdl.handle.net/10012/20681}


\bibitem[CG23]{Girelli}
H.~Chen and F.~Girelli,
``Categorical quantum groups and braided monoidal $2$-categories,''
arXiv:2304.07398 [math.QA].


\bibitem[CL24]{Chen}
H.~Chen and J.~Liniado,
``Higher gauge theory and integrability,''
Phys.\ Rev.\ D \textbf{110}, no.\ 8, 086017 (2024)
[arXiv:2405.18625 [hep-th]].


\bibitem[CR08]{ChuangRouquier}
J.~Chuang and R.~Rouquier, 
``Derived equivalences for symmetric groups and $\mathfrak{sl}_2$-categorification,'' 
Ann.\ of\ Math.\ \textbf{167}, no.\ 1, 245--298 (2008) 
[arXiv:0407205 [math.RT]].


\bibitem[CFM12]{Joao1}
L.~S.~Cirio and J.~Faria Martins,
``Categorifying the Knizhnik-Zamolodchikov connection,''
Differential Geom.\ Appl.\ \textbf{30}, no.\ 3, 238--261 (2012)
[arXiv:1106.0042 [hep-th]].


\bibitem[CFM15]{Joao}
L.~S.~Cirio and J.~Faria Martins,
``Infinitesimal $2$-braidings and differential crossed modules,''
Adv.\ Math.\ \textbf{277}, 426--491 (2015)
[arXiv:1309.4070 [math.CT]].


\bibitem[CFM17]{Joao2}
L.~S.~Cirio and J.~Faria Martins,
``Categorifying the $\mathfrak{sl}(2, \mathbb{C})$ Knizhnik-Zamolodchikov 
connection via an infinitesimal $2$-Yang-Baxter operator in the string Lie-$2$-algebra,''
Adv.\ Theor.\ Math.\ Phys.\ \textbf{21}, 147--229 (2017)
[arXiv:1207.1132 [hep-th]].


\bibitem[CF94]{CraneFrenkel}
L.~Crane and I.~B.~Frenkel, 
``Four-dimensional topological quantum field theory, Hopf categories, and the canonical bases,'' 
J.\ Math.\ Phys.\ \textbf{35}, no.\ 10, 5136--5154 (1994) 
[arXiv:9405183 [hep-th]]. 


\bibitem[Dri90]{Drinfeld90} 
V.~G.~Drinfeld, 
``On quasitriangular quasi-Hopf algebras and on a group that is closely 
connected with ${\rm Gal}(\overline{\mathbb{Q}}/{\mathbb{Q}})$,''
Algebra i Analiz \textbf{2}, no.\ 4, 149--181 (1990);
Leningrad Math.\ J.\ \textbf{2}, no.\ 4, 829--860 (1991).


\bibitem[ES02]{EtingofSchiffmann}
P.~Etingof and O.~Schiffmann, 
{\it Lectures on Quantum Groups}, 
Lectures in Mathematical Physics, International Press, Somerville, MA (2002).


\bibitem[GLST20]{Koszul}
A.~Guan, A.~Lazarev, Y.~Sheng and R.~Tang,
``Review of deformation theory II: a homotopical approach,''
Adv.\ Math.\ (China) \textbf{49}, no.\ 3, 278--298  (2020)
[arXiv:1912.04028 [math.AT]].


\bibitem[JY21]{Yau}
N.~Johnson and D.~Yau,
{\it $2$-dimensional Categories},
Oxford University Press, Oxford (2021)
[arXiv:2002.06055 [math.CT]].


\bibitem[KKMP24]{KKMP}
E.~Karlsson, C.~Keller, L.~M\"uller and J.~Pulmann,
``Deformation quantization via categorical factorization homology,''
arXiv:2410.12516 [math.QA].


\bibitem[Kas95]{Kassel}
C.~Kassel,
{\it Quantum Groups},
Grad.\ Texts in Math.\ \textbf{155},
Springer-Verlag, New York (1995).


\bibitem[Kel82]{Kelly}
G.~M.~Kelly,
{\it Basic concepts of enriched category theory},
London Math.\ Soc.\ Lecture Note Ser.\ \textbf{64},
Cambridge University Press, Cambridge-New York (1982).
[Reprints in Theory and Applications of Categories \textbf{10}, 1--136 (2005).]


\bibitem[KL09]{KhovanovLauda}
M.~Khovanov and A.~D.~Lauda, 
``A diagrammatic approach to categorification of quantum groups. I,'' 
Represent.\ Theory \textbf{13}, 309--347 (2009) 
[arXiv:0803.4121 [math.QA]].


\bibitem[KS24]{Schnitzer}
A.~Kraft and J.~Schnitzer,
``An introduction to $L_\infty$-algebras and their homotopy theory for the working mathematician,''
Rev.\ Math.\ Phys.\ \textbf{36}, no.\ 01, 2330006 (2024)
[arXiv:2207.01861 [math.QA]].


\bibitem[LurX]{Lurie}
J.~Lurie,
{\it Derived algebraic geometry X: Formal moduli problems}.
\url{http://www.math.harvard.edu/~lurie/papers/DAG-X.pdf}.


\bibitem[Maj95]{Majid}
S.~Majid,
{\it Foundations of Quantum Group Theory},
Cambridge University Press, Cambridge (1995). 


\bibitem[Maj12]{Majid2}
S.~Majid,
``Strict quantum $2$-groups,''
arXiv:1208.6265 [math.QA].


\bibitem[PTVV13]{PTVV}
T.~Pantev, B.~To{\"e}n, M.~Vaqui{\'e} and G.~Vezzosi,
``Shifted symplectic structures,''
Publ.\ Math.\ Inst.\ Hautes {\'E}tudes Sci.\ \textbf{117}, 271--328 (2013)
[arXiv:1111.3209 [math.AG]].


\bibitem[Pri17]{Pridham}
J.~P.~Pridham,
``Shifted Poisson and symplectic structures on derived $N$-stacks,''
J.\ Topol.\ \textbf{10}, no.\ 1, 178--210 (2017)
[arXiv:1504.01940 [math.AG]].


\bibitem[Pri18]{PridhamOutline}
J.~P.~Pridham,
``An outline of shifted Poisson structures and deformation quantisation in derived differential geometry,''
arXiv:1804.07622 [math.DG].


\bibitem[Pri24]{PridhamNotes}
J.~P.~Pridham,
``A note on {\'e}tale atlases for Artin stacks 
and Lie groupoids, Poisson structures and quantisation,''
J.\ Geom.\ Phys.\ \textbf{203}, 105266  (2024)
[arXiv:1905.09255 [math.AG]].


\bibitem[Rie14]{Riehl}
E.~Riehl,
{\it Categorical Homotopy Theory},
New Math.\ Monogr.\ \textbf{24},
Cambridge University Press, Cambridge (2014).


\bibitem[RW15]{Wagemann}
S.~Rivi{\`e}re and F.~Wagemann,
``On the string Lie algebra,''
Algebr.\ Represent.\ Theory \textbf{18}, no.\ 4, 1071--1099 (2015)
[arXiv:1407.6806 [math.QA]].


\bibitem[Rou08]{Rouquier}
R.~Rouquier, 
``$2$-Kac-Moody algebras,''
arXiv:0812.502 [math.RT].


\bibitem[Saf21]{Safronov}
P.~Safronov,
``Poisson-Lie structures as shifted Poisson structures,''
Adv.\ Math.\ \textbf{381}, Paper No.\ 107633 (2021)
[arXiv:1706.02623 [math.AG]].


\bibitem[SV24]{Schenkel}
A.~Schenkel and B.~Vicedo,
``5d 2-Chern-Simons theory and 3d integrable field theories,''
Commun.\ Math.\ Phys.\ \textbf{405}, no.\ 12, 293 (2024)
[arXiv:2405.08083 [hep-th]].

\end{thebibliography}
\end{document}